\documentclass[final,12pt]{colt2023} % Anonymized submission
\title[CLT Analysis of Sampling Algorithms]{Utilising the CLT Structure in Stochastic Gradient based Sampling : Improved Analysis and Faster Algorithms}

\usepackage{times}
\usepackage[utf8]{inputenc}
\usepackage{amsfonts}
\usepackage{amsmath}
\usepackage{amssymb}
\usepackage{physics}
\usepackage{bbm}
\usepackage{comment}
\usepackage{xcolor}
\definecolor{OliveGreen}{rgb}{0,0.6,0}

\usepackage{nicefrac}
\newtheorem{assumption}{Assumption}
\newcommand{\red}[1]{{\color{red}#1}}
\newcommand{\green}[1]{{\color{OliveGreen}#1}}

\newcommand{\bR}{\mathbb{R}}
\newcommand{\bP}{\mathbb{P}}
\newcommand{\vA}{\mathbf{A}}
\newcommand{\tvN}{\Tilde{\mathbf{N}}}

\newcommand{\vx}{\mathbf{x}}
\newcommand{\hvx}{\hat{\vx}}
\newcommand{\hvX}{\hat{\mathbf{X}}}
\newcommand{\hvz}{\hat{\vz}}
\newcommand{\vy}{\mathbf{y}}
\newcommand{\hvy}{\hat{\vy}}
\newcommand{\vY}{\mathbf{Y}}
\newcommand{\disteq}{\stackrel{d}{=}}
\newcommand{\vz}{\mathbf{z}}
\newcommand{\vZ}{\mathbf{Z}}

\newcommand{\vg}{\mathbf{g}}

\newcommand{\vm}{\mathbf{m}}
\newcommand{\bE}{\mathbb{E}}
\newcommand{\dotp}[2]{\left\langle #1, #2 \right \rangle}
\newcommand{\cF}{\mathcal{F}}
\newcommand{\cG}{\mathcal{G}}

\newcommand{\cN}{\mathcal{N}}
\newcommand{\vK}{\mathbf{K}}
\newcommand{\vN}{\mathbf{N}}
\newcommand{\vW}{\mathbf{W}}

\newcommand{\vX}{\mathbf{X}}

\newcommand{\vB}{\mathbf{B}}
\newcommand{\vI}{\mathbf{I}}

\newcommand{\iidsim}{{\overset{\mathrm{iid}}{\sim}}}

\newcommand{\heps}{\hat{\epsilon}}
\newcommand{\TV}{\textrm{TV}}

\newcommand{\wass}[1]{\mathcal{W}_{#1}}
\newcommand{\mubar}{\Bar{\mu}}
\newcommand{\vG}{\mathbf{G}}

\newcommand{\tilvy}{\Tilde{\vy}}

\newcommand{\KL}[2]{\mathsf{KL}\left(#1\bigr|\bigr|#2\right)}
\newcommand{\Renyi}[3]{\mathcal{R}_{#3}\left(#1\bigr|\bigr|#2\right)}

\newcommand{\FD}[2]{\mathsf{FD}\left(#1\bigr|\bigr|#2\right)}

\newcommand{\bEvent}[1]{\mathbb{I}_{\left\{ #1 \right\}}}
\newcommand{\bProb}[1]{\mathbb{P}\left\{ #1 \right\}}
\newcommand{\Thetatilde}{\Tilde{\Theta}}
\newcommand{\Otilde}{\Tilde{O}}

\newcommand{\pistar}{\pi^*}
\newcommand{\CLSI}{\lambda_{\mathsf{LSI}}}
\newcommand{\CPI}{\lambda_{\mathsf{PI}}}
\newcommand{\CLO}[1]{\lambda_{\mathsf{LO}\left( #1 \right)}}

\newcommand{\ddt}{\frac{\dd}{\dd t}}

\newcommand{\Law}[1]{\mathrm{Law}\left(#1\right)}

 \newcommand{\ad}[1]{}
 \newcommand{\dn}[1]{}
 \newcommand{\ar}[1]{}

\newcommand{\aniket}[1]{}
\newcommand{\anant}[1]{}
\newcommand{\dheeraj}[1]{}

\newcommand{\hSigma}{\hat{\Sigma}}

\coltauthor{%
 \Name{Aniket Das}\Email{ketd@google.com}\\
 \addr Google Research, \\
 Bangalore, India
 \AND
 \Name{Dheeraj Nagaraj} \Email{dheerajnagaraj@google.com}\\
 \addr Google Research, \\
 Bangalore, India 
 \AND
 \Name{Anant Raj}
 \Email{anant.raj@inria.fr}\\
 \addr University of Illinois Urbana-Champaign\\
       \& Inria, Ecole Normale Sup\'{e}rieure
}

\begin{document}

\maketitle

\begin{abstract}%
\sloppy
We consider stochastic approximations of sampling algorithms, such as Stochastic Gradient Langevin Dynamics (SGLD) and the Random Batch Method (RBM) for Interacting Particle Dynamcs (IPD). We observe that the noise introduced by the stochastic approximation is nearly Gaussian due to the Central Limit Theorem (CLT) while the driving Brownian motion is exactly Gaussian. We harness this structure to absorb the stochastic approximation error inside the diffusion process, and obtain improved convergence guarantees for these algorithms. For SGLD, we prove the first stable convergence rate in KL divergence without requiring uniform warm start, assuming the target density satisfies a Log-Sobolev Inequality. Our result implies superior first-order oracle complexity compared to prior works, under significantly milder assumptions. We also prove the first guarantees for SGLD under even weaker conditions such as H\"{o}lder smoothness and Poincare Inequality, thus bridging the gap between the state-of-the-art guarantees for LMC and SGLD. Our proof techniques also motivate the first known quantitative analysis of covariance correction, a procedure which corrects for the additional noise introduced by the stochastic approximation by rescaling the strength of the diffusion. Finally, we apply our techniques to analyze RBM, and significantly improve upon the guarantees in prior works (such as removing exponential dependence on horizon), under minimal assumptions. \\

\end{abstract}

%We consider stochastic approximations of sampling algorithms such as Langevin Monte Carlo (LMC) and Interacting Particle Dynamics (IPD), using random batches. The noise induced by the random batches is near-Gaussian due to the central limit theorem (CLT) while the driving Brownian motion is exactly Gaussian. Using this structure, we show that the error produced by the stochastic approximation can be incorporated inside the diffusion process driving the algorithm in order to obtain convergence guarantees.  We prove the first non-exploding, KL convergence bounds for Stochastic Gradient Langevin Dynamics (stochastic approximation of LMC) with significantly fewer assumptions and better dimension dependence for the first order oracle complexity (an improvement from $d^4$ to $d^{1.5}$) under the assumption of log-Sobolev inequality and establish convergence under even weaker conditions like the Poincare inequality. Thus, we bridge the gap between the guarantees for LMC and SGLD, and improves upon the oracle complexity of variance reduced SGLD proposed recently. Our analysis leads to a new algorithm that corrects for the additional noise from the random batches to give us faster convergence. We show that covariance-corrected SGLD enjoys even faster convergence. We show that similar techniques can be used to analyze random batch methods for IPD with fewer assumptions and remove the exponential dependence on the horizon required in prior works.  
\begin{keywords}%
Langevin Monte Carlo, SGLD, Sampling, CLT, Random Batch Method.
\end{keywords}

\section{Introduction}
The task of simulating stochastic systems or sampling from a target distribution in continuous domains via discretizations of Stochastic Differential Equations (SDEs) is a fundamental problem in machine learning, theoretical computer science, statistical physics and scientific computing \citep{parisi1981correlation, robert1999monte, frenkel2001understanding, shreve2005stochastic, lee2022manifold}. This problem finds applications in several domains such as Bayesian inference \citep{welling2011bayesian}, generative modelling \citep{ho2020denoising}, differential privacy \citep{gopi2022private} and algorithmic geometry \citep{kannan-lovasz-simonovits}. Popular algorithms for this problem include Langevin Monte Carlo (LMC) \citep{parisi1981correlation}, Hamiltonian Monte Carlo (HMC) \citep{neal2011mcmc}, Stein Variational Gradient Descent (SVGD) \citep{liu2016stein} and Interacting Particle Dynamics (IPD) \citep{carrillo2021consensus}. 
In this work, we study the stochastic approximations of Langevin Monte Carlo (LMC) and Interacting Particle Dynamics (IPD). \ref{eqn:LMC} aims to sample from a target distribution over $\bR^d$, whose density $\pistar(\vx) \propto \exp(-F(\vx))$ is known only upto a normalizing constant. This is achieved via an Euler discretization of the Langevin SDE with a step-size $\eta > 0$ as follows:
\begin{align}
\tag{LMC}
    \vx_{k+1} \leftarrow \vx_k - \eta \nabla F(\vx_k) + \sqrt{2 \eta} \epsilon_k, \ \epsilon_k \sim \cN(0, \vI) 
\label{eqn:LMC}
\end{align}
While \ref{eqn:LMC} is a popular algorithm in various statistical applications, there exist several practical settings where computing $\nabla F$ might be infeasible, or even intractable, whereas an unbiased estimate of $\nabla F$ is easily computable. Perhaps the most popular example is that of problems with a finite-sum structure, where $F(\vx) = \nicefrac{1}{n} \sum_{i=1}^{n} f_i(\vx)$. Here, a stochastic gradient (or a random batch gradient) can be computed as $\frac{1}{B}\sum_{i=1}^{B}\nabla f_{I_i}(\vx)$ for $I_1,\dots,I_B \ \iidsim \ \mathsf{Uniform}([n])$ and used in place of $\nabla F(\vx)$ in the \ref{eqn:LMC} update. The resulting algorithm is a stochastic approximation of \ref{eqn:LMC} called Stochastic Gradient Langevin Dynamics (SGLD) \citep{welling2011bayesian}, and is widely used in practice in large-scale problems. However, theoretical analysis of its convergence properties is relatively unexplored. Most prior works on analyzing the convergence of SGLD require restrictive assumptions like dissipativity and warm start \citep{raginsky2017non, quanquan-sgld}, with some being specific to finite-sum problems with smooth components \citep{suzuki-VR-SGLD-2022}. On the contrary, the convergence of LMC is well characterized \citep{vempalawibisono19, erdogdu2021-tail-growth, sinho-lmc-poincare-lsi} under a variety of isoperimetric conditions (e.g. Poincare Inequality, Log Sobolev Inequality) which are significantly weaker than the assumptions used to analyze SGLD. This situation is in stark contrast to that of optimization where Gradient Descent (GD) and Stochastic Gradient Descent (SGD) are analyzed under (nearly) identical assumptions.  \\

\noindent IPD is an algorithm that simulates the aggregation-diffusion dynamics of $n$ particles with pairwise interactions, and is actively used for simulating physical systems, sampling, and optimization (see Appendix~\ref{sec:lit_review_more}). When implemented naively, IPD incurs a per-step complexity of $O(n^2)$ which can be prohibitive for several problems. To ameliorate this, the Random Batch Method (RBM) performs a stochastic approximation of the inter-particle interactions by only considering the interaction of each particle with a random subset of $B$ particles, thereby reducing the per-step complexity to $O(nB)$.
\begin{equation}
\tag{IPD}
    \vx_{k+1}^{i} = \vx_k^{i} + \eta \vg^{i}_k(\vx_k^{i}) +\frac{\eta}{n} \sum_{j =1}^{n} \vK^{ij}_{k}(\vx^{i}_k,\vx^{j}_k) + \sqrt{\eta}\sigma \epsilon^{i}_k,  \quad \forall \ i\in[n], \epsilon_k^{i} \sim \cN(0,\vI)
\label{eqn:IPD}
\end{equation}
For both SGLD and RBM, the stochastic approximation takes the form of a conditionally i.i.d. average, and hence the error produced by the stochastic approximation is nearly Gaussian due to the Central Limit Theorem (CLT). This is in addition to, and independent of the Brownian motion driving the system, which is typically of higher magnitude. The key focus of our work is to utilize the i.i.d average structure (or CLT structure) of the stochastic approximation and apply non-asymptotic CLTs (or CLT-like arguments) to understand the interaction between the two independent sources of stochasticity. To this end, we show that the stochastic approximation noise can be effectively \emph{absorbed inside} the Brownian motion, and use this insight to derive state-of-the-art guarantees for these algorithms. Motivated by our analysis, we also design a novel covariance-correction strategy which compensates for the stochastic approximation error by  adaptively rescaling the diffusion term, and show that this leads to faster convergence without increase in computational complexity. 

\subsection{Contributions}
\begin{table}[t]
\label{tab:sgld-table}
\begin{center}
\label{tab:results-table}
  \small
    \begin{tabular}
    {|c|c|c|c|c|c|}
    \hline
    Result & Algorithm & Assumptions & Metric & Complexity \\
    \hline
    \cite{raginsky2017non}
    %\hyperlink{cite.raginsky2017non}{\shortstack{Raginsky et al. \\ (2017)}} 
    & SGLD & \shortstack{\red{Component Smooth}, \red{Dissipative}} & $\wass{2}$ & \normalsize $\tfrac{\textrm{poly}(d)}{\epsilon^4}$\small (\red{Unstable}) \\ [0.5ex]
    \hline
    \cite{quanquan-sgld} & SGLD & \shortstack{\red{Dissipative}, \red{Warm Start}\\ \red{Component Smooth} }  & $\TV$ & \normalsize $\tfrac{d^4}{\epsilon^2}$\small 
 (\red{Unstable}) \\ [0.5ex]
    \hline
    \hyperlink{cite.suzuki-VR-SGLD-2022}{\shortstack{Kinoshita and \\ Suzuki (2022)}}
     & VR-SGLD & \shortstack{Finite-Sum, LSI\\ \red{Component Smooth}} & \normalsize$\sqrt{\mathsf{KL}}$ & $\tfrac{d \sqrt{n}}{\epsilon^2}$ \small  (Stable)\\ [0.5ex]
    \hline
    \textbf{Theorem \ref{thm:sgld-lsi-stable}} & SGLD &  \shortstack{Smooth, LSI, $4^{\textrm{th}}$ moment} & $\sqrt{\mathsf{KL}}$ & \normalsize $\tfrac{d^{1.5}}{\epsilon^2}$ \small (Stable)  \\ [0.5ex]
    \hline
    \textbf{Theorem \ref{thm:absgld-lsi-stable}} & AB-SGLD & \shortstack{Finite-Sum, Smooth, LSI} & $\sqrt{\mathsf{KL}}$ & \normalsize $\tfrac{d^{1.5}}{\epsilon^2}$ \small (Stable) \\ [0.5ex]
    \hline
    \textbf{Corollary \ref{cor:cc-sgld-unstable-lo}} & CC-SGLD & Smooth, LSI, $6^{\textrm{th}}$ moment & $\TV$ & \normalsize $\tfrac{d^{\nicefrac{4}{3}}}{\epsilon^2}$\small (\red{Unstable})\\ [0.5ex]
    \hline
    \textbf{Theorem \ref{thm:sgld-fd-and-pi}} & Averaged SGLD & Smooth, \green{\textbf{PI}}, $4^{\textrm{th}}$ moment & $\TV$ & \normalsize $\tfrac{d^{2.5}}{\epsilon^4}$ \small (Stable)\\ [0.5ex]
    \hline
 %   \textbf{Theorem \ref{thm:sgld-fd-and-pi}} & SGLD & Smooth, $4^{\textrm{th}}$ moment & $\mathsf{FD}$ & $\nicefrac{d^{2.5}}{\epsilon^2}$ (Stable)  \\
 %   \hline
     \textbf{Corollary \ref{cor:sgld-unstable-lo}} & SGLD & Smooth, \green{\textbf{PI}}, $6^{\textrm{th}}$ moment & $\TV$ & \normalsize $\tfrac{d^{3.5}}{\epsilon^2}$ \small (\red{Unstable})\\ [0.5ex]
    \hline
    \textbf{Corollary \ref{cor:cc-sgld-unstable-lo}} & CC-SGLD & Smooth, \green{\textbf{PI}}, $6^{\textrm{th}}$ moment & $\TV$ & $\tfrac{d^{\nicefrac{10}{3}}}{\epsilon^2}$ (\red{Unstable})\\
    \hline
    \end{tabular}
    \normalsize
\end{center}
\vspace{-0.5cm}
\caption{Comparison with prior works for SGLD. $d,\epsilon$ and $n$ respectively mean the dimension, error and number of components (for finite-sum problems). Note that 1. dissipativity implies LSI and $p^{\textrm{th}}$ moment bounds 2.  $\sqrt{\mathsf{KL}}$ is a stronger `metric' than $\TV$ and under LSI it is stronger than $\wass{2}$}.
\vspace{-1cm}
\end{table}
\normalsize
Our work develops non-asymptotic CLTs (our CLT-like arguments) to quantify the approximate Gaussianity of the noise introduced by the random batch-based stochastic approximations used in SGLD and RBM, leading to state-of-the-art convergence guarantees and algorithmic improvements. Our key contributions are summarized as follows.

\paragraph{SOTA Guarantee for Smooth SGLD under LSI} When $F$ is smooth and $\pistar$ satisfies a Log-Sobolev Inequality (LSI), we prove that the oracle complexity (i.e. number of stochastic gradient evaluations) of SGLD to attain last-iterate $\epsilon$-convergence in KL divergence is $\Otilde(\nicefrac{d^{1.5}}{\epsilon})$, assuming every iterate has $O(d^2)$ $4^{\textrm{th}}$ moment. When $F$ has a finite-sum structure, we remove the $4^{\mathsf{th}}$ moment assumption by using adaptive batch-sizes. Compared to prior works, our result obtains a significantly improved oracle complexity in a stronger metric, without imposing restrictive assumptions like dissipativity, component smoothness and warm start (see Table 1). Our last-iterate error bounds are stable, i.e. they do not diverge as the number of iterations $k \to \infty$. Our results for SGLD also improve upon the best-known rates for more sophisticated algorithms such as variance-reduced SGLD (VR-SGLD).

\paragraph{Stationarity Guarantee for Smooth SGLD} In the above setting, without any isoperimetric assumptions on $\pistar$, we obtain a stable $O(\nicefrac{d^{2.5}}{\epsilon^2})$ oracle complexity for average-iterate $\epsilon$-convergence of SGLD in Fisher Divergence (FD). In the sampling literature, convergence in FD is considered analogous to first-order stationarity in nonconvex optimization \citep{sinho-non-log-concave-lmc, sinho-fisher-lb}, making our result the first known stationarity guarantee for smooth SGLD. As a corollary, we derive a stable oracle complexity of $O(\nicefrac{d^{2.5}}{\epsilon^4})$ for $\epsilon$-convergence in Total Variation (TV) whenever $\pistar$ satisfies a Poincare Inequality (PI), providing the first known rate for SGLD under PI. Prior to this work, such guarantees were only known for LMC \citep{sinho-non-log-concave-lmc}. Thus, our results significantly bridge the gap between SGLD and LMC.

\paragraph{Statistical Indistinguishability of SGLD and LMC} Without imposing any smoothness or isoperimetric assumptions, we obtain trajectory-level KL divergence bounds between SGLD and LMC, which show that the two algorithms are nearly statistically indistinguishable. This result leads to a highly general technique for obtaining last-iterate TV guarantees for SGLD under a variety of settings. As a corollary, we obtain a $\Thetatilde(\nicefrac{d^{3.5}}{\epsilon^2})$ last-iterate oracle complexity for SGLD when $F$ is smooth and $\pistar$ satisfies PI.

\paragraph{Covariance Correction} We apply our proof techniques to analyze Covariance Correction, a popular technique for implementing stochatic gradient MCMC algorithms (dating back to at least \cite{ma2015complete} \footnote{The reference \cite{ma2015complete} came to the authors' notice after the publication of this work}) which involves compensating for the excess noise due to stochastic approximation by appropriately rescaling the diffusion term. By obtaining the first known convergence guarantees for SGLD with covariance correction (CC-SGLD), we rigorously establish that covariance correction indeed leads to faster convergence without increase in computational complexity. In particular, for CC-SGLD, we obtain last-iterate TV guarantees of $\Thetatilde(\nicefrac{d^{4/3}}{\epsilon^2})$ under smoothness and LSI, and $\Thetatilde(\nicefrac{d^{10/3}}{\epsilon^2})$ under smoothness and PI.

\paragraph{Analysis of RBM} Assuming only boundedness of the inter-particle interactions $\vK^{ij}_k$, we derive an $\Otilde(\nicefrac{\eta^2 nK}{B^2})$ trajectory level upper bound between $K$ iterations of RBM (run with $n$ particles, batch-size $B$ and step-size $\eta$) and IPD. We improve this to $\Otilde(\nicefrac{\eta^2 nK}{B^3})$ for Covariance Corrected RBM (CC-RBM). Our results (see Table 2) significantly improves upon prior works that either make stringent assumptions on $\vg_k, \vK^{ij}_k$ and its derivatives \citep{jin2020random}, or suffer an exponential dependence on $\eta K$ \citep{jin2021convergence}.

\begin{table}[t]
\begin{center}
\label{tab:results-table-IPD}
  \small
    \begin{tabular}
    {|c|c|c|c|c|c|}
    \hline
    Result & Algorithm & Assumptions & Metric & Bound \\
    \hline
    \cite{jin2020random}
    %\hyperlink{cite.raginsky2017non}{\shortstack{Raginsky et al. \\ (2017)}} 
    & RBM & \shortstack{$-\vg^{i}_k(\vx) = \nabla V(\vx)$, $\vK_k^{ij}(\vx,\vy) = H(\vx-\vy)$\\ $\|\nabla H\| \leq L$, $\nabla^2 V \succeq r\vI$, $r > 2L$ } & \shortstack{averaged \\$\wass{2}$} & \normalsize $\sqrt{\frac{\eta}{B}+\eta^2}$ \\
    \hline
    \cite{jin2021convergence} & RBM & \shortstack{poly growth of $\vg_k^{i},\nabla \vg_k^{i}$ \\ $\vK_t^{ij}, \nabla \vK_k^{ij}, \nabla^2 \vK_k^{ij}$ are bounded} & \shortstack{averaged \\$\wass{2}$} & $e^{CK\eta}\sqrt{\tfrac{\eta}{B}+\eta^2}$ \\
    \hline 
    Theorem~\ref{thm:rbm-stat-indistinguishability} & RBM & \shortstack{Finite $\vg_k^{ij}$, bounded $\vK_k^{ij}$ } & $\sqrt{\mathsf{KL}}$ & $\frac{\sqrt{n\eta^2 K}}{B}$\\
    \hline
    Theorem~\ref{thm:cc-rbm-stat-indistinguishability} & CC-RBM & \shortstack{Finite $\vg_k^{ij}$, bounded $\vK_k^{ij}$} & $\sqrt{\mathsf{KL}}$ & $\frac{\sqrt{n\eta^2 K}}{B^{\nicefrac{3}{2}}}$\\
    \hline
    \end{tabular}
    \normalsize
\end{center}
\vspace{-0.5cm}
\caption{Comparison with prior works for RBM (prior works consider convergence to continuous time version of IPD). The function $H$ is bounded uniformly. The exponential bound can be inferred by following the proof of \cite[Theorem 3.1]{jin2021convergence}. Averaged $\wass{2}$ refers to average of $\wass{2}$ distance of laws of individual particles and $\mathsf{KL}$ is the KL divergence between the laws of entire trajectories of all particles.}
\vspace{-0.5cm}
\end{table}
\normalsize

\subsection{Related Work}
Sampling algorithms such LMC and IPD have been extensively studied by prior works. Below we give a concise review of the relevant literature and refer to Appendix~\ref{sec:lit_review_more} for a detailed discussion
\paragraph{LMC and SGLD} 
In typical Bayesian inference settings, the function $F$ in \ref{eqn:LMC} is an empirical average of $n$ sample functions, which can be expensive to evaluate as $n$ is generally very large. For such large-scale problems, SGLD has emerged as the sampling algorithm of choice \citep{welling2011bayesian}. While some prior works have investigated the convergence properties of SGLD \citep{raginsky2017non, quanquan-sgld}, they generally impose stringent assumptions like dissipativity and warm start. On the contrary, the convergence of LMC is well understood under much weaker isoperimetric assumptions \citep{vempalawibisono19,sinho-lmc-poincare-lsi}.

\paragraph{IPD and RBM} IPD can be computationally prohibitive whenever the number of particles $n$ is large. To ameliorate this, the Random Batch Method, which considers the interaction of each particle only with a random subset of other particles, was first proposed in \cite{jin2020random}, which also provided convergence guarantees under stringent regularity conditions on the confining and interaction forces (see Table \ref{tab:results-table-IPD}). These conditions were relaxed in \cite{jin2021convergence} at the cost of an exponential dependence on the horizon $\eta K$. Subsequent works \citep{ko2021uniform, daus2021random} also analyze RBM in specialized settings. 

\subsection{Notation and Organization}
By $\KL{\nu}{\mu}$, we denote the KL-divergence between probability measures $\nu$ and $\mu$. $\FD{\nu}{\mu} = \mathbb{E}_{\nu} \left[\|\nabla \ln (\nu/\mu)\|^2\right]$ denotes the Fisher Divergence between $\nu$ and $\mu$. Whenever $\mu$ and $\nu$ are probability measures over $\mathbb{R}^d$ with finite $p^{\mathsf{th}}$ moments, we denote their $p$-Wasserstein distance with respect to the Euclidean metric by $\wass{p}(\mu,\nu)$. Whenever $X \sim \mu$ and $Y \sim \nu$, we use $\KL{X}{Y}$ to denote $\KL{\mu}{\nu}$ and $\wass{p}(X,Y)$ to mean $\wass{p}(\mu,\nu)$. Similarly, given a sigma algebra $\mathcal{G}$ over the same measure space as the random variables $X,Y$, we use $\KL{X}{Y|\mathcal{G}}$ to denote the $\mathcal{G}$ measurable random variable $\KL{\Law{X|\mathcal{G}}}{\Law{Y|\mathcal{G}}}$. The first order oracle complexity of any algorithm refers to the number of stochastic gradient computations, i.e., the number of computations of the for, $\nabla f(\vx, \xi)$ (see Section~\ref{sec:prelim} for a formal construction). We use $\stackrel{d}{=}$ to denote equality in distribution and $\stackrel{d}{\approx}$ to denote approximate equality in distribution. We use $[n]$ and $(n)$ to denote the sets $\{1, \dots, n \}$ and $\{ 0, \dots, n-1 \}$ respectively. By $C$ we denote universal constants which can change in every appearance. By $\lesssim$ we denote $\leq$ up-to universal constants.  

\noindent In Sections~\ref{sec:prelim} and~\ref{sec:assumptions}, we define the exact problem setup and state the set of assumptions which we use. In Section~\ref{sec:tech_results}, we discuss our main technical results and in Section~\ref{sec:main_results} we apply them to establish the convergence of SGLD, RBM, CC-SGLD, and CC-RBM and discuss their significance. 
\section{Preliminaries}
\label{sec:prelim}
\textbf{LMC and SGLD}
%\label{subsec:lmc_sgld}
We consider the problem of sampling from a distribution over $\bR^d$ with density $\pistar(\vx) \propto \exp(-F(\vx))$. Under minimal assumptions, Langevin SDE $\dd \vx_t = -\nabla F(\vx_t) \dd t + \sqrt{2}\dd B_t$ admits $\pistar$ as the stationary distribution \citep{roberts1996exponential}. Here, $\dd B_t$ denotes the standard Brownian increment in $\bR^d$ \citep{stroock1997multidimensional}. The forward Euler discretization of this SDE with step size $\eta > 0$ leads to \ref{eqn:LMC}. Stochastic Gradient Langevin Dynamics (SGLD) is a stochastic approximation of the \ref{eqn:LMC} updates, which arises as follows : For $k \in \mathbb{N}\cup\{0\}$ and $i \in \mathbb{N}$, let $\xi_{k, i} \ \iidsim \ \mathbb{P}_{\xi}$ be a sequence of random variables (supported on the domain $\Xi$) that are sampled independently of $\vx_0$, $(\epsilon_k)_{k\in \mathbb{N}\cup\{0\}}$. Given access to a first order stochastic gradient oracle, which inputs $\vx \in \bR^d$ and batch size $B \in \mathbb{N}$, and outputs stochastic gradients $\nabla f(\vx,\xi_{k,1}),\dots,\nabla f(\vx,\xi_{k,B})$ such that $\mathbb{E}\left[\nabla f(\vx_k,\xi_{k,1})\bigr|\vx_k = \vx\right] = \nabla F(\vx)$, the updates of \ref{eqn:SGLD} are given by:
\begin{equation}
\tag{SGLD}
    \hvx_{k+1} \leftarrow \hvx_k - \frac{\eta}{B}\sum_{i=1}^{B} \nabla f(\hvx_k,\xi_{k,i}) + \sqrt{2 \eta} \epsilon_k, \ \epsilon_k \sim \cN(0, \vI)
\label{eqn:SGLD}
\end{equation}
\noindent \textbf{Covariance Correction} When conditioned on $\hvx_k = \vx$, the stochastic approximation noise in \ref{eqn:SGLD}, which is given by $\frac{1}{B}\sum_{i=1}^{B}(\nabla f(\hvx_k,\xi_{k,i}) - \nabla F(\hvx_k))$ is close in distribution to $\cN(0, \Sigma(\vx))$ due to the Central Limit Theorem (CLT). Thus, \ref{eqn:SGLD} updates can be approximately expressed as:
\begin{equation*}
    \hvx_{k+1} \stackrel{d}{\approx} \hvx_k - \eta \nabla F(\hvx_k) + \eta \sqrt{\Sigma(\hvx_k)}\Tilde{\epsilon}_k +\sqrt{2\eta}\epsilon_k 
\end{equation*}
where $\Tilde{\epsilon}_k \sim \cN(0,\vI)$ independent of $\epsilon_k$. Since $\sqrt{\eta\tfrac{\Sigma(\hvx_k)}{2}}\Tilde{\epsilon}_k +\epsilon_k \sim \cN\left(0, \vI + \tfrac{\eta\Sigma(\hvx_k)}{2} \right)$, we conclude that SGLD approximately resembles LMC but driven with higher noise covariance  $\vI + \frac{\eta\Sigma(\hvx_k)}{2}$. Thus, given an estimate $\hSigma(\hvx_k)$ of $\Sigma(\hvx_k)$ (computed using another random batch of size $B$) one can correct for this additional noise via a covariance corrected update as follows:
\begin{align*}
    \hvx_{k+1} = \hvx_k - \frac{\eta}{B}\sum_{i=1}^{B} \nabla f(\hvx_k,\xi_{k,i}) + \sqrt{2 \eta\bigr(\vI - \tfrac{\eta \hat{\Sigma}(\hvx_k)}{2}\bigr)} \epsilon_k
\end{align*}
The estimator $\hSigma(\hvx_k)$ is specified in Section \ref{subsec:cc-sgld-rbm}. To ameliorate the computational expense of computing the matrix square root, we apply the linearization $\sqrt{\vI - \nicefrac{\eta \hat{\Sigma}(\hvx_k)}{2}} \approx \vI - \nicefrac{\eta \hSigma(\hvx_k)}{4}$. The resulting linearized update, called Covariance Corrected SGLD (or CC-SGLD) is as follows,  
\begin{equation}
\tag{CC-SGLD}
    \hvx_{k+1} = \hvx_k - \frac{\eta}{B}\sum_{i=1}^{B} \nabla f(\hvx_k,\xi_{k,i}) + \sqrt{2 \eta}\left(\vI - \tfrac{\eta \hat{\Sigma}(\hvx_k)}{4}\right)\epsilon_k
\label{eqn:CC-SGLD}
\end{equation}
As we shall show in Section \ref{subsec:cc-sgld-rbm}, the estimator $\hSigma(\hvx_k)$ is chosen to be a sum of rank 1 matrices such that $(\vI - \tfrac{\eta \hat{\Sigma}(\hvx_k)}{4})\epsilon_k$ can be evaluated with $O(dB)$ computational cost. Thus, each step of \ref{eqn:CC-SGLD} has the same oracle complexity and computational complexity as that of \ref{eqn:SGLD}.
\paragraph{IPD, RBM and CC-RBM}
Let $\vx_0^{1},\dots,\vx_0^{n} \in \bR^d$ be the initial positions of $n$ particles drawn from an arbitrary initial distribution. We denote their positions at time $k \in \mathbb{N}\cup\{0\}$ by $\vx_k^{i}$ for $i \in [n]$. Let $\epsilon_i^{k} \ \iidsim \ \cN(0, \vI)$ be independent of the initial positions. Fixing a step-size $\eta > 0$ and diffusion strength $\sigma >0$, the resulting discrete-time aggregation-diffusion dynamics for the system, driven by an external force $\vg_k^{i}:\mathbb{R}^{d}\to \mathbb{R}^d$ and interaction forces  $\vK_k^{ij} : \mathbb{R}^d\times\mathbb{R}^d \to \mathbb{R}^d$, is known as Intracting Particle Dynamics \eqref{eqn:IPD}. Popular choices for $\vg_k^{i}$ and $\vK_k^{ij}$ are $\vg_k^{i} = -\nabla \phi(\vx^i_k)$ and $\vK_k^{ij}(\vx,\vy) = -\nabla_\vx \psi(\|
 \vx - \vy\|)$, where $\phi$ and $\psi$ are the confining and interaction potentials respectively. Implementing $K$ steps of \ref{eqn:IPD} involves $\Theta(n^2 K)$ evaluations of $\vK^{ij}_{k}$ which may be prohibitively expensive. To mitigate this, we define the Random Batch Method (RBM) and Covariance Corrected Random Batch Method (CC-RBM). At each time step $k$, the Random Batch Method draws $I^{i1}_k, \dots,I^{iB}_k \ \iidsim \ \mathsf{Uniform}([n])$ for $B \in [n]$, independent of everything else, and deploys a stochastic approximation of the inter-particle interactions as, $\hat{\vK}^{i}_k(\hvx_k) = \frac{1}{B}\sum_{j=1}^{B}\vK_k^{iI^{ij}_k}(\hvx_k^{i},\hvx_k^{I^{ij}_k})$. The RBM updates are then given by
\begin{equation}
\tag{RBM}
    \hat{\vx}_{k+1}^{i} = \hat{\vx}_k^{i} + \eta \vg_k^{i}(\hat{\vx}_k^{i}) + \eta \hat{\vK}^{i}_k(\hat{\vx}_k) + \sqrt{\eta}\sigma\epsilon_k^{i}
\label{eqn:RBM}
\end{equation}
Covariance Corrected RBM additionally samples $J^{i1}_k, \bar{J}^{i1}_k, \dots, J^{iB^{\prime}}_k, \bar{J}^{iB^{\prime}}_k \ \iidsim \ \mathsf{Uniform}([n])$ for some $B^{\prime} \in [n]$, and constructs the covariance estimator for $\hat{\vK}^{i}_k$ as follows,
\small
\begin{align*}
    \hat{\Sigma}_k^{i} := \frac{1}{2BB^{\prime}}\sum_{l=1}^{B^{\prime}}\left(\vK^{iJ^{il}}_k(\hvx_k^{i},\hvx_k^{J^{il}_k})-\vK^{i\bar{J}^{il}}_k(\hvx_k^{i},\hvx_k^{\bar{J}^{il}_k})\right)\left(\vK^{iJ^{il}}_k(\hvx_k^{i},\hvx_k^{J^{il}_k})-\vK^{i\bar{J}^{il}}_k(\hvx_k^{i},\hvx_k^{\bar{J}^{il}_k})\right)^{\intercal}
\end{align*} \normalsize
The covariance corrected update for RBM is then given by
\begin{align*}
    \hat{\vx}_{k+1}^{i} = \hat{\vx}_k^{i} + \eta \vg_k^{i}(\hat{\vx}_k^{i}) + \eta \hat{\vK}^{i}_k(\hat{\vx}_k) + \sigma\sqrt{\eta}\sqrt{\vI - \nicefrac{\eta}{\sigma^2}\hat{\Sigma}_k^{i}}\epsilon_k^{i}
\end{align*}
Similar to~\ref{eqn:CC-SGLD}, approximating the matrix square-root as $\sqrt{\vI - \nicefrac{\eta}{\sigma^2}\hat{\Sigma}_k^{i}} \approx \vI - \nicefrac{\eta}{2\sigma^2}\hat{\Sigma}_k^{i}$ gives us the following update for Covariance Corrected RBM (or CC-RBM).
\begin{equation}
\tag{CC-RBM}
    \hat{\vx}_{k+1}^{i} = \hat{\vx}_k^{i} + \eta \vg_k^{i}(\hat{\vx}_k^{i}) + \eta \hat{\vK}^{i}_k(\hat{\vx}_k) + \sigma\sqrt{\eta}\left(\vI - \tfrac{\eta}{2\sigma^2}\hat{\Sigma}_k^{i}\right)\epsilon_k^{i}
\label{eq:covrbm}
\end{equation}
Since $\hat{\Sigma}^{i}_k$ is a sum of rank-1 matrices, the cost of evaluating $(\vI - \tfrac{\eta}{2\sigma^2}\hat{\Sigma}_k^{i})\epsilon_k^{i}$ is $O(dB)$. Thus, \ref{eq:covrbm} has the same computational complexity as that of \ref{eqn:RBM}.
\begin{comment}
\begin{remark}
While covariance correction may seem computationally expensive, requiring $O(d^2B)$ computations per step, we can make the first order approximation $(\vI - \eta \hat{\Sigma})^{\nicefrac{1}{2}} \approx \vI - \nicefrac{\eta \hat{\Sigma}}{2}$. This is a good approximation since $\nicefrac{\eta \hat{\Sigma}}{2} = O(\frac{\eta}{B})$. Notice that the covariance estimators considered in Theorems~\ref{thm:cc-sgld-stat-indistinguishable} and~\ref{thm:cc-rbm-stat-indistinguishability} are such that $\hat{\Sigma}\epsilon_k$ can be computed with $O(dB)$ computations. While this can be analyzed using the techniques established in the current work, we do not do so for the sake of clarity. The comparison upper bounds with and without co-variance correction (Theorems~\ref{thm:cc-rbm-stat-indistinguishability} and~\ref{thm:rbm-stat-indistinguishability}) suggests that when the number of particles $n \gg d$, covariance correction even with $O(d^2B)$ computation can be computationally efficient.
\end{remark}
\end{comment}
\section{Assumptions}
\label{sec:assumptions}
Our analysis of SGLD makes the following expected smoothness assumption on the potential $F$, which is weaker than that of prior works which assume component smoothness \citep{raginsky2017non, quanquan-sgld, suzuki-VR-SGLD-2022}, i.e., smoothness of $f(\vx, \xi) \ \forall \ \xi \in \Xi$. 
\begin{assumption}[Expected H\"{o}lder Smoothness]
\label{as:stoc-grad-holder}
$F(\vx)$ is $s$-H\"{o}lder smooth (or $s$-gradient H\"{o}lder continuous) for some $s \in (0,1]$, i.e., there exists a constant $L \geq 0$ such that $\nabla F(\vx)$ satisfies:
\begin{equation}
\tag{$s$-H\"{o}lder}
\norm{\nabla F(\vx) - \nabla F(\vy)} \leq L \norm{\vx - \vy}^{s}, \ \ \forall \ \vx, \vy \in \bR^d
\label{eqn:defn-holder-smoothness}
\end{equation}
When $s=1$, $F(\vx)$ is said to be $L$-smooth (or $L$-gradient Lipschitz), i.e.,
\begin{equation}
\tag{$L$-smooth}
\norm{\nabla F(\vx) - \nabla F(\vy)} \leq L \norm{\vx - \vy}, \ \ \forall \ \vx, \vy \in \bR^d
\label{eqn:defn-smoothness}
\end{equation}
\end{assumption}
We use the following growth condition to control the stochastic gradients. As shown in Appendix \ref{app-subsec:grad-growth}, this assumption is significantly weaker than that of prior works. \citep{raginsky2017non, quanquan-sgld, suzuki-VR-SGLD-2022}.
\begin{assumption}[Almost-sure Stochastic Gradient Growth]
\label{as:stoc-grad-growth}
For $M, G \geq 0$, the functions $f(\vx, \xi)$ satisfy:
\begin{equation}
\tag{lin-growth a.s}
\norm{\nabla f(\vx, \xi) - \nabla F(\vx)} \leq M \norm{\vx} + G, \ \ \forall \ \vx \in \bR^d, \xi \in \Xi 
\label{eqn:defn-grad-growth}
\end{equation}
\end{assumption}
For the analysis of SGLD in the \ref{eqn:defn-smoothness} setting, our techniques generalize to the following relaxed assumption on the stochastic gradients which is satisfied for a large class of problems.
\begin{assumption}[Subgaussian Stochastic Gradient Growth]
\label{as:stoc-grad-subg}
For $M, G \geq 0$, the functions $f(\vx, \xi)$ satisfy the following norm-subgaussianity condition \citep{norm-subgaussian-shortnote} for the stochastic gradients:
\begin{equation}
\tag{lin-growth subG}
\bP\left\{\norm{\nabla f(\vx, \xi) - \nabla F(\vx)} \geq t \biggr| \vx \right\} \leq 2 \exp\left(-\frac{t^2}{2 \left[M \norm{\vx} + G\right]^2}\right) \ \forall \ \vx \in \bR^d
\label{eqn:defn-grad-growth-subg}
\end{equation}
\end{assumption}
\begin{comment}
We note that, under the mild assumption that $\sup_{\xi} \norm{\nabla f(\vx_0, \xi)} \leq C$ for some $\vx_0 \in \bR^d$ and $C \geq 0$, \ref{eqn:defn-holder-smoothness} implies \ref{eqn:defn-grad-growth} with $\zeta \leq s$. However, separating Assumption \ref{as:stoc-grad-growth} from Assumption \ref{as:stoc-grad-holder} allows us to derive sharper rates for problems where $\zeta < s$. For instance, if $f(\vx, \xi)$ is both Lipschitz and gradient Lipschitz, then \ref{eqn:defn-holder-smoothness} is satisfied with $s = 1$ but \ref{eqn:defn-grad-growth} is satisfied with $\zeta = 0$. \\
\end{comment}
To establish the convergence of SGLD to the target distribution, we impose the following isoperimetric assumption on $\pistar$ proposed by \cite{latala-oleskiewicz-paper}, which interpolates between the well known Poincare and Log-Sobolev Inequalities \citep{bakry2014analysis}.
%To establish convergence guarantees for SGLD, we require the following isoperimetric assumption on the target distribution $\pistar$.
\begin{assumption}[Lata\l{}a-Oleskiewicz Inequality]
\label{as:target-lo}
The target $\pistar$ satisfies the Lata\l{}a-Oleskiewicz inequality of order $\alpha \in [1,2]$, i.e., for every smooth function $g : \bR^d \to \bR$, $\pistar$ satisfies the following inequality for some constant $\CLO{\alpha}$
\begin{equation}
\tag{$\alpha$-LO}
\sup_{p \in (1, 2)} \frac{\bE_{\pistar}\left[g^2\right] - \bE_{\pistar}\left[g^p\right]^{\nicefrac{2}{p}} }{(2 - p)^{2 - \nicefrac{2}{\alpha}}} \leq \frac{1}{\CLO{\alpha}} \bE_{\pistar}\left[\norm{\nabla g}^2\right]
\label{eqn:defn-latala-oleskiewicz}
\end{equation}
When $\alpha = 1$, \ref{eqn:defn-latala-oleskiewicz} is equivalent to the Poincare Inequality with constant $\CPI = \CLO{1}$,
\begin{equation}
\tag{PI}
\bE_{\pistar}\left[g^2\right] - \bE_{\pistar}[g]^2 \leq \frac{1}{\CPI} \bE_{\pistar}\left[\norm{\nabla g}^2\right]
\label{eqn:defn-poincare}
\end{equation}
When $\alpha = 2$, \ref{eqn:defn-latala-oleskiewicz} reduces to the Logarithmic Sobolev Inequality with constant $\CLSI = \CLO{2}$:
\begin{equation}
\tag{LSI}
\bE_{\pistar}\left[g^2 \log\left(g^2\right)\right] - \bE_{\pistar}\left[g^2 \right] \log\left(\bE_{\pistar}\left[g^2 \right] \right) \leq \frac{1}{\CLSI} \bE_{\pistar}\left[\norm{\nabla g}^2\right]
\label{eqn:defn-log-sobolev}
\end{equation}
\end{assumption}
We highlight that the exponent $s$ in \ref{eqn:defn-holder-smoothness} and the order $\alpha$ in \ref{eqn:defn-latala-oleskiewicz} must satisfy the relation $\alpha \leq 1 + s$. Finally, we impose the following mild technical condition on the iterates of SGLD.
\begin{assumption}[Moment Growth]
\label{as:sgld-moment-growth}
The iterates of $\hvx_1, \dots, \hvx_K$ SGLD satisfy the following $p$-moment growth condition for some $p \geq 1$.
\begin{equation}
\tag{$p$-moment growth}
\bE[\norm{\hvx_k}^q] \leq C_q d^{\nicefrac{q}{2}}, \ \ \forall \ q \in [0, p]
\label{eqn:defn-moment-growth}
\end{equation}
where $C_q \geq 0 \ \forall \ q \in [0, p]$ and can be an arbitrary non-negative increasing function of $q$. 
\end{assumption}
As we shall show in Appendix \ref{app-subsec:assumption-moment-bounds}, the \ref{eqn:defn-moment-growth} condition is milder than assumptions like dissipativity and strong convexity outside of a compact set used in prior works \citep{quanquan-sgld, raginsky2017non, cheng2020sharp}. On the contrary, we assume \ref{eqn:defn-moment-growth} only for $p=4$ (in Theorems \ref{thm:sgld-lsi-stable} and \ref{thm:sgld-fd-and-pi}) or $p=6$ (in Theorems \ref{thm:sgld-stat-indistinguishable} and \ref{thm:cc-sgld-stat-indistinguishable}), making our assumptions much weaker than that of prior works. Furthermore, for finite-sum problems, we remove this assumption by using adaptive batch-sizes (Theorem \ref{thm:absgld-lsi-stable}).
\section{Technical Results}
\label{sec:tech_results}
Our main results on SGLD and RBM, presented in Section~\ref{sec:main_results}, are divided into three major groups:
\begin{enumerate}
    \item Stable convergence results for SGLD which use a conditional CLT type result (Lemma \ref{lem:noise-control-chisquare}) to establish a differential inequality for the KL divergence to the target distribution (Lemma \ref{lem:sgld-descent-lemma}).
    \item A sharp convergence analysis of the law of the trajectory of SGLD (and RBM), to the law of the trajectory of LMC (and IPD), without any smoothness or isoperimetric assumptions, using a novel Wasserstein CLT (Lemma \ref{lem:wass-clt-new}) and associated entropic CLT.
    \item Analysis of CC-SGLD and CC-RBM using the above entropic CLTs.
\end{enumerate}
In this section, we establish our key technical results like Lemmas \ref{lem:kl-flow-interpolating}, \ref{lem:noise-control-chisquare} and \ref{lem:wass-clt-new}.
\subsection{Conditional CLT Type Analysis and Descent EVI}
\label{subsec:evi_descent}
\sloppy 
Using $\vN(\hvx_k, \xi_k) := \frac{1}{B} \sum_{i=1}^{n} \nabla f(\hvx_k, \xi_{k, j}) - \nabla F(\hvx_k)$ to denote the stochastic gradient noise, we can express the SGLD updates as follows : 
\begin{equation}
\label{eqn:sgld-noisy-update}
    \hvx_{k+1} = \hvx_k - \eta (\nabla F(\hvx_k) + \vN(\hvx_k, \xi_k)) + \sqrt{2 \eta} \epsilon_k, \ \  \epsilon_k \sim \cN(0, \vI)
\end{equation}
Adapting the arguments of \cite{vempalawibisono19} to SGLD, we construct a piecewise-continuous stochastic process $(\vx_t)_{t \in [0, K \eta]}$ defined as follows for any  $d$-dimensional Brownian motion $B_t$ (which is independent of $\vx_0$):
\begin{align*}
    \Law{\vx_0} &= \Law{\hvx_0}   \\
    \dd \vx_t &= -\left[\nabla F(\vx_{k\eta}) + \vN(\vx_{k\eta}, \xi_k)\right] \dd t + \sqrt{2} \dd B_t, \ t \in [k\eta, (k+1)\eta], \ k \in [K]
\end{align*}
Note that $\vx_t \stackrel{d}{=} \vx_{k\eta} - (t-k\eta) (\nabla F(\vx_{k\eta}) + \vN(\vx_{k\eta}, \xi_k)) + \sqrt{2(t - k\eta)} \vz_t, \ \vz_t \sim \cN(0, \vI)$ where $t \in [k\eta, (k+1)\eta]$, $k \in (K)$. An inductive argument shows that $\Law{\vx_{k\eta}} = \Law{\hvx_k}$. To this end, we call $(\vx_t)_{t \in [0, K \eta]}$ an \emph{interpolating process} for SGLD, and write $\mu_t = \Law{\vx_t}$. We shorten $\vN(\vx_{k\eta}, \xi_k)$ to $\vN$, whenever clear from the context. Our analysis begins by establishing a differential inequality for the time-evolution of KL divergence along the interpolating process. We present the proof of this lemma in Appendix \ref{proof:lem-kl-flow-interpolating}.
\begin{lemma}
\label{lem:kl-flow-interpolating}
Assume $\KL{\mu_0}{\pistar} < \infty$. Then, for any $k \in (K)$ and $t \in [k\eta, (k+1)\eta]$,
\begin{align*}
    \ddt \KL{\mu_t}{\pistar} \leq -\frac{1}{2}\FD{\mu_t}{\pistar} + \bE\left[\norm{\nabla F(\vx_t) - \nabla F(\vx_{k\eta})}^2\right] + \bE\left[\norm{\bE[\vN | \vx_{k\eta}, \vx_t]}^2\right]
\end{align*}
\end{lemma}
The term $\bE\left[\norm{\nabla F(\vx_t) - \nabla F(\vx_{k\eta})}^2\right]$ in Lemma \ref{lem:kl-flow-interpolating} corresponds to the discretization error of SGLD, which can be controlled in a straightforward manner. The key technical challenge is to sharply bound the term $\bE\left[\norm{\bE[\vN | \vx_{k\eta}, \vx_t]}^2\right]$, which corresponds to the error due to stochastic approximation. When conditioned on $\vx_{k\eta}$, $\vN$ is an average of $B$ i.i.d random vectors, and hence, is approximately Gaussian due to the Central Limit Theorem. In fact, if $\vN_j|\vx_{k\eta} \sim \cN(0,\Sigma)$ were exactly Gaussian, then $\vN$ and $\vx_{t}$ would be jointly Gaussian conditioned on $\vx_{k\eta}$. In this case, standard results on conditional expectations of Gaussians would imply that:
$\bE\left[\norm{\bE[\vN | \vx_{k\eta}, \vx_t]}^2\right] \leq \frac{4(t-k\eta)(\Tr(\Sigma))^2}{B^2} $, where $\mathbb{E}\left[\Tr(\Sigma)^2\right] \lesssim (M^4 C_{4} d^{2 } + G^4)$ under the \ref{eqn:defn-grad-growth-subg} and \ref{eqn:defn-moment-growth} conditions for $p=4$.  However, since $\vN$ need not be Gaussian in general, we require a conditional CLT type result \citep{dedecker2002necessary} to sharply control the stochastic approximation error. To this end, we derive sharp bounds for $\norm{\bE[\vN | \vx_{k\eta}, \vx_t]}^2$ by exploiting the fact that $\vN$ is conditionally an i.i.d. sum of zero-mean vectors. The proof of this result is presented in Appendix \ref{proof:lem-noise-control-chisquare}. %\vN = \frac{1}{B} \sum_{j=1}^{B} \vN_j$ where $\vN_j$ are zero-mean i.i.d random variables when conditioned on $\vx_{k\eta}$. %As we shall now demonstrate, the Gaussianity assumption on $\vN$ is not necessary to derive sharp upper bounds on $\bE\left[\norm{\bE[\vN | \vx_{k\eta}, \vx_t]}^2\right]$. %In particular, appropriately utilizing the CLT structure of the stochastic gradient noise sufficient to establish a sharp upper bound similar to Lemma \ref{lem:noise-control-gaussian}. To this end, we establish the following, 

\begin{lemma}
\label{lem:noise-control-chisquare}
Let the \ref{eqn:defn-grad-growth-subg} and \ref{eqn:defn-moment-growth} conditions be satisfied with $p = 4 $. Then, 
\begin{align*}
    \bE\left[\norm{\bE[\vN | \vx_{k\eta}, \vx_t]}^2\right] \leq \frac{C(t - k\eta)(M^4 C_{4 } d^{2 } + G^4)}{B^2}
\end{align*}
\end{lemma}
Equipped with Lemmas \ref{lem:kl-flow-interpolating} and \ref{lem:noise-control-chisquare}, we derive the following differential inequality for the time evolution of the KL divergence along the trajectory of the interpolating process. The proof of this result is presented in Appendix \ref{proof:lem-sgld-descent-lemma}
\begin{lemma}%[Descent EVI for SGLD]
\label{lem:sgld-descent-lemma}
Let the \ref{eqn:defn-smoothness}, \ref{eqn:defn-grad-growth-subg} and \ref{eqn:defn-moment-growth} conditions be satisfied with $p = 4 $. Then, for $\eta \leq \nicefrac{1}{6L}$, the following inequality is satisfied for any $k \in (K)$ and  $t \in [k\eta, (k+1)\eta]$,\small
\begin{align*}
    \ddt \KL{\mu_t}{\pistar} &\leq -\frac{1}{4}\FD{\mu_t}{\pistar} + CL^2 (t - k \eta)d + \frac{CL(t - k \eta)(M^2 C_{2 } d  + G^2)}{B} \\
    &+ \frac{C(t - k \eta) (M^4 C_{4 } d^{2 } + G^4)}{B^2}
\end{align*}\normalsize
\end{lemma}

\subsection{Random Function Representation and Entropic CLT based Analysis}
\label{subsec:entropic_clt}

We now describe our technical approach to the analysis of SGLD and RBM via the random function representation and Gaussian-smoothed Wasserstein CLTs. While the description below focuses on SGLD, our analysis of CC-SGLD, RBM and CC-RBM applies similar arguments. 

\noindent Since LMC is a discrete-time Markov Process (with iterates $\vx_0,\vx_1,\dots$), it admits a random function representation, i.e., for any $K \in \mathbb{N}$, there exists a measurable function $H_K$ such that $\vx_{1:K+1} = H_K(\vx_0, \epsilon_0, \dots, \epsilon_K)$. Moreover, from equation~\eqref{eqn:sgld-noisy-update}, we recall that the SGLD iterates can be expressed as $\hvx_{k+1} = \hvx_{k} - \eta \nabla F(\vx_k) + \sqrt{2\eta}(\heps_k + \sqrt{\nicefrac{\eta}{2}} \vN_k), \ \heps_k \sim \cN(0, \vI), \ \vN_k = \vN(\hvx_k, \xi_k)$. Comparing with the update rule of~\ref{eqn:LMC}, we conclude that \emph{LMC and SGLD must admit the same random function representation}, i.e., $\hvx_{1:K+1} = H_K(\hvx_0, \hvz_1, \dots, \hvz_K)$ where $\hvz_k = \heps_k + \sqrt{\nicefrac{\eta}{2}} \vN_k$. Applying the data-processing inequality and the chain rule for KL divergence, and using the fact that $\Law{\vx_0} = \Law{\hvx_0}$, we conclude the following:
\begin{align*}
    \KL{\hvx_{1:K}}{\vx_{1:K}} \leq \sum_{k=0}^{K} \bE\left[ \KL{ \hvz_k | \cF_k}{\epsilon_k} \right],
\end{align*}
where $\cF_k = \sigma(\hvx_{0:K}, \hvz_{0:K-1})$. Since $\hvz_k = \heps_k + \sqrt{\nicefrac{\eta}{2}} \vN_k$ where $\vN_k | \cF_k$ is approximately Gaussian and $\heps_k$ is exactly Gaussian, controlling the term $\KL{ \hvz_k | \cF_k}{\epsilon_k}$ involves establishing an entropic CLT for $\hvz_k | \cF_k$. To this end, we prove a reverse $\mathsf{T}_2$-type inequality for Gaussian convolutions to control this term as $\KL{ \hvz_k | \cF_k}{\epsilon_k} \leq \wass{2}^2( \sqrt{\vI + \eta \Sigma_k} \vX, \vY | \cF_k) + \wass{2}^2( \vX + \sqrt{\eta}\vN_k , \sqrt{\vI + \eta \Sigma_k} \vY | \cF_k)$ where $\Sigma_k = \bE\left[\vN_k \vN^{T}_k | \hvx_k\right]$ and $\vX, \vY \ \iidsim \ \cN(0, \vI)$. The first term is bounded by directly computing the Wasserstein distance between zero-mean Gaussians. To control the second term, we quantify the approximate Gaussianity of $\vN_k$ by developing a novel Wasserstein CLT for Gaussian convolutions of bounded random vectors, which may be of independent interest. This result, which is an extension of \cite{zhai2018clt}, is proved in Appendix \ref{proof:lem-wass-clt-new}  
\begin{lemma}[$\wass{2}$ CLT for Gaussian Convolutions of Bounded RVs]
\label{lem:wass-clt-new}
Let $\vY_1, \dots, \vY_B$ be i.i.d random vectors in $\bR^d$ satisfying $\bE\left[ \vY_i \right] = 0$, $\bE \left[\vY_i \vY^T_i \right] = \Sigma_{\vY}$ and $\norm{\vY_i} \leq \beta$ almost surely, and let $\vY = \frac{1}{\sqrt{B}} \sum_{i=1}^{B} \vY_i$. Furthermore, let $\vX$ and $\vZ$ be sampled iid from $\cN(0, \vI)$, independent of $\vY_1, \dots, \vY_B$. Then, the following holds for $\beta^2 \leq \nicefrac{1}{5}$ and $\norm{\Sigma_{\vY}}_2 \leq \nicefrac{1}{5d}$
\begin{align*}
    \wass{2}^2(\sqrt{\vI - \Sigma_{\vY}}\vX + \vY,  \vZ) \leq \frac{25 \beta^6 d (1 + \log(B))^2}{B}
\end{align*}
\end{lemma}
Our entropic CLT analysis does not require any smoothness assumptions on $F$ or isoperimetric assumptions on $\pistar$, and only uses the \ref{eqn:defn-grad-growth} and \ref{eqn:defn-moment-growth} assumptions for $p = 6$. 
\section{Results for SGLD and RBM}
\label{sec:main_results}
\subsection{Convergence of SGLD under Smoothness and LSI}
 Our first result is a stable convergence guarantee for the last iterate of SGLD with respect to KL divergence when $F$ is smooth and $\pistar$ satisfies \ref{eqn:defn-log-sobolev}. The proof, which is presented in Appendix \ref{proof:thm-sgld-lsi-stable}, follows from Lemma \ref{lem:sgld-descent-lemma} and uses the fact that \ref{eqn:defn-log-sobolev} is a gradient dominance condition for $\KL{.}{\pistar}$ in the space of measures \citep{ottovillanithm}, i.e., $\FD{\mu_t}{\pistar} \geq 2 \CLSI \KL{\mu_t}{\pistar}$.
 \begin{theorem}[Convergence of SGLD under Smoothness and LSI]
\label{thm:sgld-lsi-stable}    
Let the \ref{eqn:defn-smoothness}, \ref{eqn:defn-grad-growth-subg} and \ref{eqn:defn-moment-growth} conditions be satisfied with $p = 4$ and let $\pistar$ satisfy \ref{eqn:defn-log-sobolev}. Then, for $\eta \leq \nicefrac{\CLSI}{6L^2}$, the last iterate of SGLD satisfies the following.\small
\begin{align*}
    \KL{\Law{\hvx_{K}}}{\pistar} &\leq e^{-\nicefrac{\CLSI \eta K}{2}} \KL{\Law{\hvx_{0}}}{\pistar} + \frac{CL^2 \eta d}{\CLSI} + \frac{CL\eta(M^2 C_{2} d + G^2)}{B\CLSI} \nonumber\\&\quad+ \frac{C\eta (M^4 C_{4} d^{2 } + G^4)}{\CLSI B^2} 
\end{align*}\normalsize
\end{theorem}
The first term $e^{-\nicefrac{\CLSI \eta K}{2}} \KL{\Law{\hvx_{0}}}{\pistar}$ arises from the error due to initialization which decays exponentially fast. The second term $\frac{CL^2 \eta d}{\CLSI}$, which is independent of the batch-size $B$ and diminishes as $\eta \to 0$, corresponds to the inherent bias of LMC \citep{bernton2018-lmc-jko, wibisono2018-sampling-as-opt}. The remaining terms, which diminish with increasing batch-size $B$ for any fixed $\eta > 0$, encapsulate the error due to stochastic approximation. 

\paragraph{Oracle Complexity} Lemma 1 of \cite{vempalawibisono19} shows that $\KL{\Law{\hvx_{0}}}{\pistar} = O(d)$ is easily ensured via appropriate Gaussian initialization. Thus, for any $\epsilon \leq 1$, Theorem \ref{thm:sgld-lsi-stable} implies that $\eta = O\left(\nicefrac{\epsilon}{d}\right)$, $B = O(\sqrt{d})$ and $K = \Otilde\left(\nicefrac{d}{\epsilon}\right)$ ensures $\KL{\Law{\hvx_{K}}}{\pistar} \leq \epsilon$. Thus, Theorem \ref{thm:sgld-lsi-stable} implies an oracle complexity of $\Otilde\left(\nicefrac{d^{1.5}}{\epsilon}\right)$ for $\epsilon$-convergence in KL divergence, which translates to $\Otilde\left(\nicefrac{d^{1.5}}{\epsilon^2}\right)$ oracle complexity for $\epsilon$-convergence in $\TV$ and $\wass{2}$ via Pinsker's inequality and Talagrand's inequality respectively.  

% We refer to Appendix \ref{oracle:thm-absgld-lsi-stable} for a detailed discussion. 
\paragraph{Stability and Tightness of the Analysis}
For any $\eta \leq \nicefrac{\CLSI}{6L^2}$, and any batch size $B \geq 1$ the rate obtained in Theorem \ref{thm:sgld-lsi-stable} is stable, i.e., does not diverge as $K \to \infty$. Moreover, Theorem \ref{thm:sgld-lsi-stable} doesn't require any uniform warm start assumption. In the noise-free setting (i.e., $M, G=0$) Theorem \ref{thm:sgld-lsi-stable} recovers the result of \cite{vempalawibisono19} for LMC upto constant factors.
\\

\noindent When $F$ admits a finite-sum structure, i.e., $F(\vx) = \nicefrac{1}{n} \sum_{i=1}^{n} f_i(\vx)$, we propose an adaptive batch-size schedule for SGLD, for which we establish results similar to Theorem~\ref{thm:sgld-lsi-stable} without the \ref{eqn:defn-moment-growth} assumption. The algorithm, called AB-SGLD, is as follows:
\begin{enumerate}
    \item  $B_k := \min \{ n, 1 + \left \lceil{M \norm{\hvx_k} + G}\right \rceil  \}$ ; Sample indices $i_1, \dots, i_{B_k} \ \iidsim \ \mathsf{Uniform}([n])$
    \item $\hvx_{k+1} \leftarrow \hvx_k - \nicefrac{\eta}{B_k} \sum_{j=1}^{B_k} \nabla f_{i_j}(\hvx_k) + \sqrt{2 \eta} \epsilon_k$ where $\epsilon_k \sim \cN(0, \vI)$
\end{enumerate}
We refer to Appendix \ref{app-sec:absgld-smooth-stable-convergence} for the proof of Theorem~\ref{thm:absgld-lsi-stable}. For simplicity, we assume $\eta \leq \nicefrac{1}{8}$ and $M \geq L \geq 1$.
\begin{theorem}[Stable Convergence of AB-SGLD]
\label{thm:absgld-lsi-stable}
Let the \ref{eqn:defn-smoothness} and \ref{eqn:defn-grad-growth-subg} conditions be satisfied, and let $\pistar$ satisfy \ref{eqn:defn-log-sobolev}. Then, for $\eta \leq \frac{\CLSI^2}{8L(LM + 128M^2)}$ and any $K \in \mathbb{N}$:
\small
\begin{align*}
    \KL{\Law{\hvx_{K}}}{\pistar} &\leq e^{-\nicefrac{\CLSI \eta K}{4}}\KL{\Law{\hvx_{0}}}{\pistar} + \frac{C\eta}{\CLSI} \left(L^2 d + M^2 \vm^2_2 + G^2 \right) + \frac{CL^2 \eta^2}{\CLSI} \left(M\vm_1 + M + G \right)
\end{align*} \normalsize
where $\vm_p = \bE_{\pistar}[\|\vx\|^p]^{\nicefrac{1}{p}}$. Furthermore, the amortized batch size $\Bar{B} = \nicefrac{1}{K} \sum_{k=0}^{K} \bE\left[B_k\right]$ satisfies,
\small
\begin{align*}
    \Bar{B} \lesssim 2 + G + \frac{M}{\CLSI^{\nicefrac{3}{2}}\eta K}\sqrt{\KL{\Law{\hvx_0}}{\pistar}} + \frac{M \sqrt{\eta}}{\CLSI}\left(L \sqrt{d} + M \vm_2 + G\right) + \frac{L \eta}{\CLSI} \sqrt{ M \vm_1 + M + G}
\end{align*}\normalsize
\end{theorem}
\paragraph{Expected Oracle Complexity} Since \ref{eqn:defn-log-sobolev} implies subgaussianity \citep{van2014probability}, it is reasonable to consider $\vm_1, \vm_2 = O(\sqrt{d})$. Moreover, appropriate Gaussian initialization ensures $\KL{\Law{\hvx_0}}{\pistar} = O(d)$. Hence, for any $\epsilon \leq 1$, setting $\eta = O(\nicefrac{\epsilon}{d})$ and $K = \Otilde\left(\nicefrac{d}{\epsilon}\right)$ ensures $\epsilon$-convergence in KL divergence. Under this setting, the amortized batch size of AB-SGLD is $\Bar{B} = \Otilde(\sqrt{d})$. Hence, we conclude that the expected oracle complexity of AB-SGLD to achieve $\epsilon$-convergence in KL is $K \Bar{B} = \Otilde\left(\nicefrac{d^{1.5}}{\epsilon}\right)$, which further implies  $\Otilde\left(\nicefrac{d^{1.5}}{\epsilon^2}\right)$ oracle complexity for $\epsilon$-convergence in $\TV$ and $\wass{2}$.

%\paragraph{Batch Size} Both SGLD and AB-SGLD require batch size $B = O(\sqrt{d})$. In comparison, \cite{quanquan-sgld} require $B = O(d)$, \cite{raginsky2017non} assume $B = O\left(\mathsf{poly}\left(\nicefrac{d}{\epsilon}\right)\right)$. The VR-SGLD algorithm of \cite{suzuki-VR-SGLD-2022} requires a batch size and inner loop length of $O(\sqrt{n})$
%Compare our batch-sizes to those proposed in the literature: $O(\mathsf{poly}(\tfrac{d}{\epsilon}))$ \citep{raginsky2017non}, $O(d)$ \citep{quanquan-sgld}, and $O(\sqrt{n})$ \citep{suzuki-VR-SGLD-2022}.
%\paragraph{Batch Size} Both SGLD and AB-SGLD assume $O(\sqrt{d})$ batch size to attain $\epsilon$-convergence in KL, which is much lower than what is assumed in prior work. 
%which is the same as that implied by Theorem \ref{thm:sgld-lsi-stable} for SGLD, but with the additional \ref{eqn:defn-moment-growth} assumption. Since $\pistar$ satisfies LSI, AB-SGLD also enjoys $ \Otilde\left(\nicefrac{d^{1.5}}{\epsilon^2}\right)$ oracle complexity in expectation for $\epsilon$-convergence in TV and $\wass{2}$. A detailed discussion is deferred to Appendix \ref{oracle:thm-absgld-lsi-stable}. \ad{TODO: Add comparison to prior works}

\subsection{First Order Stationarity of Averaged SGLD under Smoothness}
Recalling the definition of $\mu_t$ from Section~\ref{subsec:evi_descent}, we denote the averaged law of SGLD as $\mubar_{K\eta} = \nicefrac{1}{K \eta} \int_{0}^{K\eta} \mu_t dt$. We prove a stable convergence guarantee for $\FD{\mubar_{K\eta}}{\pistar}$ without imposing any isoperimetry assumptions on $\pistar$. Such guarantees for sampling are considered analogous to first-order stationary point analysis for nonconvex optimization \citep{sinho-non-log-concave-lmc, sinho-fisher-lb}. Our result, which follows from Lemma \ref{lem:sgld-descent-lemma} and convexity of Fisher Divergence, implies a stable convergence guarantee in $\TV$ if $\pistar$ satisfies \ref{eqn:defn-poincare}. We present the proof in Appendix \ref{proof:thm-sgld-fd-and-pi}. % and discuss how to  draw a sample from $\mubar_{K\eta}$ in Appendix \ref{oracle:thm-sgld-fd-pi} without incurring additional oracle complexity. 
\begin{theorem}[First Order Stationarity of SGLD]
\label{thm:sgld-fd-and-pi}
Let the \ref{eqn:defn-smoothness}, \ref{eqn:defn-grad-growth-subg} and \ref{eqn:defn-moment-growth} conditions be satisfied with $p = 4$. For any $\eta \leq \nicefrac{1}{6L}$, the following holds
\small
\begin{align*}
    \FD{\mubar_{K\eta}}{\pistar} \lesssim \frac{\KL{\Law{\hvx_0}}{\pistar}}{K \eta} + L^2 \eta d + \frac{L\eta(M^2 C_{2} d + G^2)}{B} + \frac{\eta (M^4 C_{4} d^{2} + G^4)}{B^2}
\end{align*}\normalsize
Additionally, the following holds if $\pistar$ satisfies \ref{eqn:defn-poincare},\small
\begin{align*}
     \TV(\mubar_{K\eta}, \pistar)^2 \lesssim \frac{ \KL{\Law{\hvx_0}}{\pistar}}{\CPI K \eta} + \frac{L^2 \eta d}{\CPI} + \frac{L\eta(M^2 C_{2} d + G^2)}{\CPI B} + \frac{ \eta (M^4 C_{4} d^{2} + G^4)}{\CPI B^2}
\end{align*} \normalsize
\end{theorem}
\paragraph{Sampling from $\mubar_{K\eta}$} It is easy to sample from $\mubar_{K\eta}$ without any added sample complexity as follows : 1. Sample $t_0 \sim \mathsf{Unif}[0, K\eta]$. Let $k_0$ be the largest integer such that $\eta k_0 \leq t_0$. 2. Run SGLD for $k_0$ steps 3. Perform a partial update $\bar{\vx} = \hvx_{k_0} - (t_0 - \eta k_0) \vg_{k_0} + \sqrt{2(t_0 - \eta k_0)} \epsilon, \ \epsilon \sim \cN(0, \vI)$, where $\vg_{k_0}$ is the stochastic gradient at step $k_0$. By the construction of the interpolating process, it is easy to show that $\Law{\bar{\vx}} = \mubar_{K\eta}$.

\paragraph{Oracle Complexity} Appropriate Gaussian initialization ensures that $\KL{\Law{\hvx_{0}}}{\pistar} = O(d)$ Thus, for $\epsilon \leq 1$, setting $\eta = O\left(\nicefrac{\epsilon}{d}\right)$, $B = O(\sqrt{d})$ and $K = O\left(\nicefrac{d^2}{\epsilon^2}\right)$ suffices to ensure $\epsilon$-convergence of averaged SGLD in FD, thereby implying an oracle complexity of $O\left(\nicefrac{d^{2.5}}{\epsilon^2}\right)$. Under the same setting, Theorem \ref{thm:sgld-fd-and-pi} also implies an oracle complexity of $O\left(\nicefrac{d^{2.5}}{\epsilon^4}\right)$ for $\epsilon$-convergence in $\TV$ when $\pistar$ satisfies \ref{eqn:defn-poincare}. %A detailed discussion is deferred to Appendix \ref{oracle:thm-sgld-fd-pi}

\paragraph{Stability} For any $\eta \leq \nicefrac{1}{6L}$ and $B \geq 1$, the rates presented in Theorem \ref{thm:sgld-fd-and-pi} are stable as they don't diverge when $K \to \infty$, nor do they require any uniform warm start assumption. Furthermore, in the noise-free setting (i.e. $M,G=0$), Theorem \ref{thm:sgld-fd-and-pi} recovers the result of \cite[Theorem 1 and Corollary 5]{sinho-non-log-concave-lmc} upto constant factors.  

\paragraph{Batch Size} Theorems \ref{thm:sgld-lsi-stable}, \ref{thm:absgld-lsi-stable} and \ref{thm:sgld-fd-and-pi} require a batch size of $B = O(\sqrt{d})$ for $\epsilon$-convergence, which significantly improves upon prior works. In comparison, \cite{quanquan-sgld} require $B = O(d)$ and \cite{raginsky2017non} assume $B = O\left(\mathsf{poly}\left(\nicefrac{d}{\epsilon}\right)\right)$. The VR-SGLD algorithm of \cite{suzuki-VR-SGLD-2022} requires a batch size and inner loop length of $O(\sqrt{n})$. Note that $n \gg d$ for most practical scenarios.

\subsection{Analysis of SGLD and RBM via Wasserstein CLT}
We use the arguments presented in Section~\ref{subsec:entropic_clt} to prove convergence bounds between the trajectories of SGLD and LMC (as well as RBM and IPD) under minimal assumptions (assuming only \ref{eqn:defn-grad-growth} and \ref{eqn:defn-moment-growth}).
The proof is presented in Appendix \ref{proof:thm-sgld-stat-indistinguishable}
\begin{theorem}[Statistical Indistinguishability of SGLD and LMC]
\label{thm:sgld-stat-indistinguishable}
Let $\vx_{1:K}$ and $\hvx_{1:K}$ and denote the iterates of LMC and SGLD respectively, with the same step-size and initialization. Furthermore, let the \ref{eqn:defn-grad-growth} and \ref{eqn:defn-moment-growth} conditions be satisfied. Then, 
\begin{align*}
    \KL{\hvx_{1:K}}{\vx_{1:K}} &\lesssim \frac{\eta^2 K}{B^2} \left(M^{4} C_{4} d^{2} + G^{4}\right) + \frac{\eta^3 K}{B^4}\left(M^{6} C_{6} d^{6} + G^{6} d^3 \right)\left(1 + \log B \right)^2
\end{align*}
\end{theorem}
Theorem \ref{thm:sgld-stat-indistinguishable} implies that the convergence of LMC to $\pistar$ is nearly sufficient to ensure the convergence of SGLD. To this end, Theorem \ref{thm:sgld-stat-indistinguishable} gives us the following highly general technique to derive last-iterate TV guarantees for SGLD : 1. Using any existing convergence guarantee for LMC, choose $\eta, K$ such that $\TV(\vx_{K}, \pistar) \leq \Thetatilde(\epsilon)$ 2. Set $B$ such that $\KL{\hvx_{1:K}}{\vx_{1:K}} \leq \Thetatilde(\epsilon^2)$. Consequently, $\TV(\hvx_K, \pistar) \leq \Thetatilde(\epsilon)$ by the Data Processing Inequality and Pinsker's inequality. 

\noindent We quantitatively demonstrate this technique in Appendix \ref{proof:cor-sgld-unstable-lo}, we use Theorem \ref{thm:sgld-stat-indistinguishable} to prove last-iterate TV guarantees for SGLD under the assumptions of \ref{eqn:defn-holder-smoothness} and \ref{eqn:defn-latala-oleskiewicz}, obtaining an oracle complexity of $\Thetatilde(\tfrac{d^\gamma}{\epsilon^{\nicefrac{2}{s}}})$ for $\epsilon$-convergence in TV, where $\gamma = \max \left\{ 1 + \beta\left(1 + \nicefrac{1}{s}\right), \nicefrac{3}{2}\left(1+\beta\right) + \nicefrac{\beta}{2s} \right\}$ and $\beta = \nicefrac{2}{\alpha} - 1$. For $s = 1$, this result implies an oracle complexity of $\Thetatilde(\nicefrac{d^{3.5}}{\epsilon^2})$ when $\pistar$ satisfies \ref{eqn:defn-poincare} (i.e., $\alpha=1$). This has a better $\epsilon$ dependence but worse $d$ dependence than the TV guarantee of Theorem \ref{thm:sgld-fd-and-pi}. However, unlike Theorems \ref{thm:sgld-lsi-stable}, \ref{thm:sgld-fd-and-pi} and \ref{thm:absgld-lsi-stable}, these guarantees are not stable. \\

\noindent Our techniques apply whenever a sampling algorithm and its stochastic approximation admit the same random function representation, and as such, can be extended to other sampling algorithms such as HMC. In fact, using similar techniques, we obtain the following statistical indistinguishability guarantee between RBM and IPD, whose proof is presented in Appendix \ref{proof:thm-rbm-stat-indistinguishability}.
\begin{theorem}[Statistical Indistinguishability of IPD and RBM]
\label{thm:rbm-stat-indistinguishability}
Let $(\vx^i_k)_{i \in [n], k \in [K]}$ and $(\hvx^i_k)_{i \in [n], k \in [K]}$ denote the iterates of the interacting particle method and the random batch method respectively, with the same initial distribution and step-size $\eta$. Suppose $\norm{K^{ij}_k(\hvx^i_k, \hvx^j_k)} \leq M$ holds almost surely, for every $k \in [K]$ and $i,j \in [n]$. Then, the following holds for any $\eta \leq \frac{B \sigma^2}{40 M^2 d}$
\begin{align*}
    \KL{(\hvx^i_k)_{i \in [n], k \in [K]}}{(\vx^i_k)_{i \in [n], k \in [K]}} &\lesssim  \frac{\eta^2 M^4 nK}{B^2 \sigma^4} +  \frac{d\eta^3 M^6 nK (1 + \log B)^2}{B^4 \sigma^6} 
\end{align*}
 
%where $C_{\mathsf{Cov}}$ and $C_{\mathsf{CLT}}$ are universal constants. 
\end{theorem}

\subsection{CC-SGLD and CC-RBM for IPD}
\label{subsec:cc-sgld-rbm}
We now consider covariance correction introduced in Section~\ref{sec:prelim}. For \ref{eqn:CC-SGLD}, we define the estimator $\hat{\Sigma}(\hvx_k)$ as follows: \small
    \begin{equation*}
        \hSigma(\hvx_k) = \begin{cases}
        0 \qquad \text{ if } \qquad \left(M\norm{\hvx_k} + G\right)^2 > \nicefrac{B}{5 \eta d}\\ 
        \nicefrac{1}{2B^2} \sum_{j=1}^{B} \left(\nabla f(\vx_k, \xi^{(1)}_{k, j}) - \nabla f(\vx_k, \xi^{(2)}_{k, j})\right) \left(\nabla f(\vx_k, \xi^{(1)}_{k, j}) - \nabla f(\vx_k, \xi^{(2)}_{k, j})\right)^{T} \ \text{ o/w}
        \end{cases}
    \end{equation*}\normalsize
Where $\nabla f(\hvx_k, \xi^{(1)}_{k, 1}), \nabla f(\hvx_k, \xi^{(2)}_{k, 1}), \dots, \nabla f(\hvx_k, \xi^{(1)}_{k, B}),  \nabla f(\hvx_k, \xi^{(2)}_{k, B})$ is a fresh batch of stochastic gradients. We use the entropic CLT approach discussed in Section \ref{subsec:entropic_clt} to establish Theorem~\ref{thm:cc-sgld-stat-indistinguishable} which is proved in Appendix \ref{proof:thm-cc-sgld-stat-indistinguishable}.

\begin{theorem}[Statistical Indistinguishability of CC-SGLD and LMC]
\label{thm:cc-sgld-stat-indistinguishable}
Let the \ref{eqn:defn-grad-growth} and \ref{eqn:defn-moment-growth} conditions be satisfied with $p = 6$. Then, the iterates of CC-SGLD satisfy the following guarantee. \small
\begin{align*}
    \KL{\hvx_{1:K}}{\vx_{1:K}} &\lesssim \frac{\eta^2 K}{B^3}\left(M^4 C_{4} d^{2} + G^4\right) + \frac{\eta^3 K}{B^3}\left(M^6 C_{6} d^{3} + G^6\right) + \frac{\eta^5 K d^6}{B^3}  +\frac{\eta^3 K d^6 (1 + \log B)^2}{B^4}.
\end{align*}\normalsize
\end{theorem}
The leading order term in Theorem \ref{thm:cc-sgld-stat-indistinguishable} is $O(\nicefrac{\eta^2 K}{B^3})$ whereas that in Theorem \ref{thm:sgld-stat-indistinguishable} for SGLD is $O(\nicefrac{\eta^2 K}{B^2})$. This indicates that CC-SGLD is a better stochastic approximation to LMC in comparison to SGLD, and enjoys faster convergence. Using this, we derive an unstable last-iterate TV guarantee for CC-SGLD under \ref{eqn:defn-holder-smoothness} and \ref{eqn:defn-latala-oleskiewicz}. The result, which is stated in Appendix \ref{proof:cor-cc-sgld-lo}, improves upon that implied by Theorem \ref{thm:sgld-stat-indistinguishable} under the same setting. In particular For $s=1$, we obtain a $\Thetatilde(\nicefrac{d^{\nicefrac{4}{3}}}{\epsilon^2})$ oracle complexity for CC-SGLD when $\pistar$ satisfies \ref{eqn:defn-log-sobolev}, which improves upon our previous guarantees for SGLD and AB-SGLD under smoothness and \ref{eqn:defn-log-sobolev}. When $\pistar$ satisfies \ref{eqn:defn-poincare}, we obtain an oracle complexity of $\Thetatilde(\nicefrac{d^{\nicefrac{10}{3}}}{\epsilon^2})$. In a similar way, we derive the following guarantee for covariance corrected RBM, which is proved in Appendix  \ref{proof:cc-rbm-stat-indistinguishability}.
\begin{theorem}[Statistical Indistinguishability of IPD and CC-RBM]
\label{thm:cc-rbm-stat-indistinguishability}
Let $(\vx^i_k)_{i \in [n], k \in [K]}$ and $(\hvx^i_k)_{i \in [n], k \in [K]}$ denote the iterates of the interacting particle method and the covariance corrected random batch method respectively, with the same initial distribution and step-size $\eta$. Suppose $\norm{K^{ij}_k(\hvx^i_k, \hvx^j_k)} \leq M$ holds almost surely, for every $k \in [K]$ and $i,j \in [n]$. Then, for any $\eta \leq \frac{B \sigma^2}{40 M^2 d}$
\small
\begin{align*}
    \KL{(\hvx^i_k)_{i \in [n], k \in [K]}}{(\vx^i_k)_{i \in [n], k \in [K]}} &\lesssim  \frac{\eta^2 M^4 n K}{B^2 B^{\prime} \sigma^4} + \frac{\eta^3 M^6 n K}{B^3 \sigma^6}  + \frac{d\eta^3 M^6 nK \left(1 + \log B\right)^2}{B^4 \sigma^6} 
\end{align*} \normalsize
\end{theorem}
As per Theorem~\ref{thm:rbm-stat-indistinguishability}, a batch size of $B \gg (\eta^2nK)^{1/2}$ suffices to ensure that the trajectories of all the particles of RBM and IPD stay close in distribution. Similarly, Theorem~\ref{thm:cc-rbm-stat-indistinguishability} ensures the same for CC-RBM and IPD when $B \gg (\eta^2nK)^{1/3}$ (assuming $B = B^{\prime})$. Both these results significantly improve upon prior works on the convergence of RBM \citep{jin2020random, jin2021convergence} that either require highly stringent regularity assumptions on $\vg_k$, $\vK_k$ and their derivatives, or suffer from an exponential dependence on $\eta K$ (See Table 2). 

\section{Conclusion and Future Work}
In this work, we introduced techniques to give sharp analyses of stochastic approximations of sampling algorithms, significantly improving upon prior work. Future work can further extend these techniques to understand convergence properties under the R\'enyi Divergence, and other algorithms like Hamiltonian Monte Carlo and randomized midpoint methods.

\bibliography{references}
\clearpage
\tableofcontents
\clearpage
\appendix

\section{More References}
\label{sec:lit_review_more}
\paragraph{LMC and SGLD} 
Discretizations of diffusion processes \cite{bakry2014analysis} such as Langevin Monte Carlo (LMC) and its closely related cousins like Hamiltonian monte-carlo and underdamped LMC have been widely studied in various settings, including large scale  \citep{vempalawibisono19,sinho-lmc-poincare-lsi,cheng2018convergence,dalalyan2019user,lovasz-simonovits, lovasz-vempala-hitrun,lee2018convergence}. Recent advancements have focused on the sharp study of LMC with non-log-concave densities and have obtained sharp analyses with minimal assumptions under various metrics like the KL divergence and Renyi divergences whenever the target distribution satisfies certain functional inequalities. \cite{vempalawibisono19} established precise and succinct convergence bounds whenever the target distribution satisfies log-Sobolev inequalities (LSI) and \citep{sinho-lmc-poincare-lsi,murat-mlsi-paper} considered settings more general settings like the Lata\l{}a-Oleskiewicz (LO) Inequality and the Modified Log-Sobolev Inequality (MLSI)

\paragraph{Interacting Particle Methods for Sampling and Optimization}
 Equation~\eqref{eqn:IPD} a system which is made up of multiple agents which evolve in time by interacting with the environment and each other via the aggregation-diffusion equations \citep{carrillo2019aggregation}. These models give rise to rich dynamics, and can lead to various phenomenon like swarming, flocking, chemotaxis and vortex formation \citep{degond2017coagulation, vicsek1995novel,cucker2007emergent,bertozzi2012characterization,chorin1973numerical} and hence are widely used by the scientific community. This has also has generated significant interest from a mathematical perspective \citep{backhoff2020mean,carrillo2019aggregation,bertozzi2012characterization,ha2009simple} where precise behavior of popular models, large deviations principles and convergence to mean field limits have been studied. Inspired by these models, various optimization, sampling and control algorithms have been designed \citep{carrillo2018analytical,carrillo2021consensus,duncan2019geometry,liu2016stein,craig2022blob,chen2021density}.

\section{Stable Convergence of SGLD under Smoothness}
\label{app-sec:sgld-smooth-stable-convergence}
\subsection{Technical Lemmas}
\begin{lemma}
\label{lem:tail-markov-control}
Let $X$ be a non-negative random variable such that $\bE\left[X^k\right] < \infty$ for some $k > 1$. Then, the following holds for any $\alpha > 0$,
\begin{align*}
    \bE\left[X \bEvent{X \geq \alpha }\right] \leq \frac{\bE\left[X^k\right]}{\alpha^{k-1}}\left(1 + \frac{1}{k-1}\right)
\end{align*}
\end{lemma}
\begin{proof}
Since $X$ is a non-negative random variable, the following holds almost surely,
\begin{align*}
    X &= \int_{0}^{\infty} \bEvent{X \geq t} \dd t \\
    X \bEvent{X \geq \alpha} &= \int_{0}^{\infty} \bEvent{X \geq \alpha} \bEvent{X \geq t} \dd t \\
    &= \int_{0}^{\infty} \bEvent{X \geq \max\{\alpha, t\}} \dd t \\
    &= \int_{0}^{\alpha} \bEvent{X \geq \alpha} \dd t + \int_{\alpha}^{\infty} \bEvent{X \geq t} \dd t \\
    &= \alpha \bEvent{X \geq \alpha} + \int_{\alpha}^{\infty} \bEvent{X \geq t} \dd t
\end{align*}
Then, by Fubini's Theorem and Markov's Inequality,
\begin{align*}
    \bE\left[X \bEvent{X \geq \alpha }\right] &=\alpha \bProb{X \geq \alpha} + \int_{\alpha}^{\infty} \bProb{X \geq t} \dd t \\
    &\leq \frac{\bE\left[X^k\right]}{\alpha^{k-1}} + \bE\left[X^k\right] \int_{\alpha}^{\infty} t^{-k} \dd t \\
    &= \frac{\bE\left[X^k\right]}{\alpha^{k-1}}\left(1 + \frac{1}{k-1}\right)
\end{align*}
\end{proof}
\begin{lemma}
\label{lem:moment-growth-subg}
Let the \ref{eqn:defn-grad-growth-subg} condition be satisfied. Then, for any $\vx \in \bR^d$ and $k \in \mathbb{N}$
\begin{align*}
    \bE\left[ \norm{\nabla f(\vx, \xi) - \nabla F(\vx)}^{2k} \biggr | \vx\right] \leq 2^{k+1} k! \left(M \norm{\vx} + G\right)^{2k}
\end{align*}
\end{lemma}
\begin{proof}
Let $z = \norm{\nabla f(\vx, \xi) - \nabla F(\vx)}$ and $u = M \norm{\vx} + G$. Since $z$ is a non-negative random variable, it follows that,
\begin{align*}
    \bE\left[\left(\nicefrac{z}{u}\right)^{2k} \biggr| \vx \right] &= \int_{0}^{\infty} \bP\left\{  \left(\nicefrac{z}{u}\right)^{2k}\geq y \biggr| \vx \right\} \dd y \\
    &= 2k \int_{0}^{\infty} v^{2k-1} \bP\left\{  \nicefrac{z}{u}\geq v \biggr| \vx \right\} \dd v \\
    &\leq 4k \int_{0}^{\infty} x^{2k-1} e^{-\nicefrac{x^2}{2}} \dd x \\
    &= k 2^{k+1} \int_{0}^{\infty} t^{k-1} e^{-t} \dd t \\
    &= 2^{k+1} k!
\end{align*}
where we use the \ref{eqn:defn-grad-growth-subg} condition in step 3. Hence, we conclude that,
\begin{align*}
    \bE\left[ \norm{\nabla f(\vx, \xi) - \nabla F(\vx)}^{2k} \biggr | \vx\right] \leq 2^{k+1} k! \left(M \norm{\vx} + G\right)^{2k}
\end{align*}
\end{proof}
\begin{lemma}[Otto-Villani Theorem {\cite{ottovillanithm}}]
\label{lem:otto-villani-thm}
Let $\pistar$ satisfy \ref{eqn:defn-log-sobolev} with constant $\CLSI$. Then, $\pistar$ satisfies Talagrand's inequality $\mathsf{T}_p $ for any $p \in [1, 2]$, i.e., for any probability measure $\mu$
\begin{align*}
    \wass{p}^2 (\mu, \pistar) \leq \frac{2}{\CLSI} \KL{\mu}{\pistar}, \ \ \forall p \in [1, 2]
\end{align*}
\end{lemma}
\begin{lemma}[{\cite[Theorem 3.1]{guillin2009transportation}}]
\label{lem:gullin-fd-info-transport}  
Let $\pistar$ satisfy \ref{eqn:defn-poincare} with constant $\CPI$. Then, the following holds for any probability measure $\mu$.
\begin{align*}
    \TV(\mu, \pistar)^2 \leq \frac{4}{\CPI} \FD{\mu}{\pistar}
\end{align*}
\end{lemma}
\begin{lemma}[{\cite[Lemma 16]{sinho-lmc-poincare-lsi}}]
\label{lem:grad-fd-bound}
Assume $F$ satisfies the \ref{eqn:defn-smoothness} condition. Then, for any arbitrary probability measure $\mu$, 
\begin{align*}
    \bE_{\vx \sim \mu}\left[\norm{\nabla F(\vx)}^2\right] \leq \FD{\mu}{\pistar} + 2dL
\end{align*}
\end{lemma}
\begin{lemma}[Chi Square Between Two Isotropic Gaussians]\label{lem:chi-sq-gaussian-control}
Let $p$ be the density of $\cN(\mu_1, \sigma^2 \vI)$ and $q$ be the density of $\cN(\mu_2, \sigma^2 \vI)$. Then,
\begin{align*}
    \bE_{x \sim q}\left[\left(\frac{p}{q}\right)^2 - 1\right] =  e^{\nicefrac{\norm{\mu_1 - \mu_2}^2}{\sigma^2}} - 1
\end{align*}
\end{lemma}
\begin{lemma}[Controlling Moments under \ref{eqn:defn-log-sobolev}]
\label{lem:moment-control-talagrand}
Assume $\pistar$ satisfies \ref{eqn:defn-log-sobolev}. Then, for any arbitrary probability measure $\mu$ and any $p \in [1,2]$,
\begin{align*}
    \bE_{\vx \sim \mu}\left[\norm{\vx}^p\right] \leq 2^{p-1} \left(\frac{2}{\CLSI}\KL{\mu}{\pistar}\right)^{\nicefrac{p}{2}} + 2^{p-1} \vm^p_p
\end{align*}
where $\vm_p = \bE_{\vx \sim \pistar}\left[\norm{\vx}^p\right]^{\nicefrac{1}{p}}$
\end{lemma}
\begin{proof}
Let $\Gamma$ denote the $\wass{p}$ optimal coupling between $\mu$ and $\pistar$. Then, by Jensen's inequality
\begin{align*}
    \bE_{\vx \sim \mu}\left[\norm{\vx}^p\right] &\leq 2^{p-1} \bE_{(\vx, \vx^*) \sim \Gamma} \left[\norm{\vx - \vx^*}^p\right] + 2^{p-1} \bE_{\vx^* \sim \pistar}\left[\norm{\vx^*}^p\right] \\
    &\leq 2^{p-1} \wass{p}^{p} \left(\mu, \pistar\right) + 2^{p-1} \vm^p_p \\
    &\leq 2^{p-1} \left(\frac{2}{\CLSI}\KL{\mu}{\pistar}\right)^{\nicefrac{p}{2}} + 2^{p-1} \vm^p_p
\end{align*}
where the last inequality follows from Lemma \ref{lem:otto-villani-thm}. 
\end{proof}
\subsection{Proof of Lemma \ref{lem:kl-flow-interpolating}}
\label{proof:lem-kl-flow-interpolating}
\begin{proof}
Consider any $k \in (K)$ and $t \in [k\eta, (k+1)\eta]$. Defining $g(\vx_{k\eta}, \xi_k) = \nabla F(\vx_{k\eta}) + \vN(\vx_{k\eta}, \xi_k)$, we note that the interpolating process is described by the following SDE for $t \in [k\eta, (k+1)\eta]$,
\begin{align*}
    \dd \vx_t = - g(\vx_{k\eta}, \xi_k) \dd t + \sqrt{2} \dd B_t
\end{align*}
Let $\mu_{t|k\eta}(\vx_t | \vx_{k\eta}, \xi_k)$ denote the density of $\Law{\vx_t | \vx_{k\eta}, \xi_k}$. Then, the Fokker Planck Equation for $\mu_{t|k\eta}$ is given by,
\begin{align*}
    \frac{\partial}{\partial t} \mu_{t|k\eta}(\vx_t | \vx_{k\eta}, \xi_k) = \nabla . (g(\vx_{k\eta}, \xi_k) \mu_{t|k\eta}(\vx_t | \vx_{k\eta}, \xi_k)) + \Delta \mu_{t|k\eta}(\vx_t | \vx_{k\eta}, \xi_k)
\end{align*}
Marginalizing out $\vx_{k\eta}$ and $\xi$, we obtain,
\begin{align*}
    \mu_t(\vx_t) = \int_{\bR^d \times \Xi^{B}} \mu_{t|k\eta}(\vx_t | \vx_{k\eta}, \xi_k) \mu_{k\eta} (\vx_{k\eta}) \dd \vx_{k\eta} \dd \bP^{B}_{\xi} 
\end{align*}
Taking the partial differential with respect to $t$ on both sides, and interchanging the differential and the integral wherever appropriate,
\begin{align*}
    \frac{\partial}{\partial t}\mu_t(\vx_t) &= \int_{\bR^d \times \Xi^{B}} \left(\frac{\partial}{\partial t}\mu_{t|k\eta}(\vx_t | \vx_{k\eta}, \xi_k)\right) \mu_{k\eta} (\vx_{k\eta}) \dd \vx_{k\eta} \dd \bP^{B}_{\xi} \\
    &= \int_{\bR^d \times \Xi^{B}}  \left(\nabla . (g(\vx_{k\eta}, \xi_k) \mu_{t|k\eta}(\vx_t | \vx_{k\eta}, \xi_k))\right) \mu_{k\eta}(\vx_{k\eta}) \dd \vx_{k\eta} \dd \bP^{B}_{\xi} \\
    &+ \int_{\bR^d \times \Xi^{B}} \left(\Delta \mu_{t|k\eta}(\vx_t | \vx_{k\eta}, \xi_k)\right) \mu_{k\eta}(\vx_{k\eta}) \dd \vx_{k\eta} \dd \bP^{B}_{\xi} \\
    &= \nabla . \left(\int_{\bR^d \times \Xi^B}g(\vx_{k\eta}, \xi_k) \mu_{t|k\eta}(\vx_t | \vx_{k\eta}, \xi_k) \mu_{k\eta}(\vx_{k\eta}) \dd \vx_{k\eta} \dd \bP^{B}_{\xi}\right) \\
    &+ \Delta \left(\int_{\bR^d \times \Xi^B} \mu_{t|k\eta}(\vx_t | \vx_{k\eta}, \xi_k) \mu_{k\eta}(\vx_{k\eta}) \dd \vx_{k\eta} \dd \bP^{B}_{\xi}\right) \\
    &= \nabla . \left(\int_{\bR^d \times \Xi^B}g(\vx_{k\eta}, \xi_k) \mu_{t|k\eta}(\vx_t | \vx_{k\eta}, \xi_k) \mu_{k\eta}(\vx_{k\eta}) \dd \vx_{k\eta} \dd \bP^{B}_{\xi}\right) + \Delta \mu_t(\vx_t)
\end{align*}
We note that, by disintegration theorem,
\begin{align*}
    \int_{\bR^d \times \Xi^B}g(\vx_{k\eta}, \xi_k) \mu_{t|k\eta}(\vx_t | \vx_{k\eta}, \xi_k) \mu_{k\eta}(\vx_{k\eta}) \dd \vx_{k\eta} \dd \bP^{B}_{\xi} &= \mu_t(\vx_t) \int_{\bR^d \times \Xi^B}g(\vx_{k\eta}, \xi_k) \dd \bP_{\vx_{k\eta}, \xi_k | \vx_t} (\vx_{k\eta}, \xi_k) \\
    &=\mu_t(\vx_t) \bE\left[g(\vx_{k\eta}, \xi_k) | \vx_t \right]
\end{align*}
Hence, we obtain the following continuity equation for $\mu_t(\vx_t)$
\begin{align}
\label{eqn:interpolating-continuity-pde}
    \frac{\partial}{\partial t} \mu_t(\vx_t) = \nabla . \left(\mu_t(\vx_t) \bE\left[g(\vx_{k\eta}, \xi_k) | \vx_t \right]\right) + \Delta \mu_t(\vx_t) 
\end{align}
We now analyze the time-evolution of $\KL{\mu_t}{\pistar}$ using equation \eqref{eqn:interpolating-continuity-pde}. Interchanging differentials and integrals wherever appropriate,
\begin{align}
\label{eqn:kl-differential-eq1}
    \frac{\dd}{\dd t} \KL{\mu_t}{\pistar} &= \frac{\dd }{\dd t} \int_{\bR^d} \mu_t(\vx) \log\left(\frac{\mu_t(\vx)}{\pistar(\vx)}\right) \dd \vx \nonumber \\
    &= \int_{\bR^d} \left(\frac{\partial}{\partial t} \mu_t(\vx)\right) \log\left(\frac{\mu_t(\vx)}{\pistar(\vx)}\right) \dd \vx + \int_{\bR^d} \frac{\partial}{\partial t} \mu_t(\vx) \dd \vx \nonumber \\
    &= \int_{\bR^d} \left(\frac{\partial}{\partial t} \mu_t(\vx)\right) \log\left(\frac{\mu_t(\vx)}{\pistar(\vx)}\right) \dd \vx + \frac{\partial}{\partial t} \int_{\bR^d} \mu_t(\vx) \dd \vx \nonumber \\
    &= \int_{\bR^d} \nabla . \left(\mu_t(\vx) \bE\left[g(\vx_{k\eta}, \xi_k) | \vx_t = \vx \right]\right) \log\left(\frac{\mu_t(\vx)}{\pistar(\vx)}\right) \dd \vx \nonumber \\
    &+ \int_{\bR^d} \Delta \mu_t(\vx) \log\left(\frac{\mu_t(\vx)}{\pistar(\vx)}\right) \dd \vx
\end{align}
where the last step uses the fact that $\frac{\partial}{\partial t} \int_{\bR^d} \mu_t(\vx) \dd \vx = 0$ since $\int_{\bR^d} \mu_t(\vx) \dd \vx = 1$ for any $t \in [k\eta, (k+1)\eta]$. We now note that,
\begin{align*}
    \mu_t(\vx) \nabla \log\left(\frac{\mu_t(\vx)}{\pistar(\vx)}\right) &= \nabla \mu_t(\vx) + \mu_t(\vx) \nabla F(\vx)
\end{align*}
Taking the divergence on both sides, we obtain the following identity,
\begin{align*}
    \Delta \mu_t(\vx) = \nabla . \left(\mu_t(\vx) \left(\nabla \log\left(\frac{\mu_t(\vx)}{\pistar(\vx)}\right) - \nabla F(\vx)\right)\right)
\end{align*}
It follows that,
\begin{align}
\label{eqn:kl-differential-eq2}
    \int_{\bR^d} \Delta \mu_t(\vx) \log\left(\frac{\mu_t(\vx)}{\pistar(\vx)}\right) \dd \vx &= \int_{\bR^d} \nabla . \left(\mu_t(\vx) \left(\nabla \log\left(\frac{\mu_t(\vx)}{\pistar(\vx)}\right) - \nabla F(\vx)\right)\right) \log\left(\frac{\mu_t(\vx)}{\pistar(\vx)}\right) \dd \vx \nonumber \\
    &= \int_{\bR^d} \mu_t(\vx)\dotp{\nabla F(\vx) - \nabla \log\left(\frac{\mu_t(\vx)}{\pistar(\vx)}\right)}{\nabla \log\left(\frac{\mu_t(\vx)}{\pistar(\vx)}\right)} \dd \vx \nonumber \\
    &= \bE\left[\dotp{\nabla F(\vx_t)}{\nabla \log\left(\frac{\mu_t(\vx_t)}{\pistar(\vx_t)}\right)}\right] - \bE\left[\norm{\nabla \log\left(\frac{\mu_t(\vx_t)}{\pistar(\vx_t)}\right)}^2 \right] \nonumber \\
    &= -\FD{\mu_t}{\pistar} + \bE\left[\dotp{\nabla F(\vx_t)}{\nabla \log\left(\frac{\mu_t(\vx_t)}{\pistar(\vx_t)}\right)}\right]
\end{align}
where the second equality follows from integration by parts on $\bR^d$, the third follows from the fact that $\mu_t = \Law{\vx_t}$ and the last follows from the definition of Fisher Divergence. Following similar steps, we obtain,
\begin{align}
\label{eqn:kl-differential-eq3}
    &\int_{\bR^d} \nabla . \left(\mu_t(\vx) \bE\left[g(\vx_{k\eta}, \xi_k) | \vx_t = \vx \right]\right) \log\left(\frac{\mu_t(\vx)}{\pistar(\vx)}\right) \dd \vx \nonumber \\ &= - \int_{\bR^d} \mu_t(\vx) \dotp{\bE\left[g(\vx_{k\eta}, \xi_k) | \vx_t = \vx \right]}{\nabla \log\left(\frac{\mu_t(\vx)}{\pistar(\vx)}\right)} \dd \vx \nonumber \\
    &= -\bE\left[\dotp{\bE\left[g(\vx_{k\eta}, \xi_k) | \vx_t\right]}{\nabla \log\left(\frac{\mu_t(\vx_t)}{\pistar(\vx_t)}\right)}\right] \nonumber \\
    &= - \bE\left[\dotp{g(\vx_{k\eta}, \xi_k)}{\nabla \log\left(\frac{\mu_t(\vx_t)}{\pistar(\vx_t)}\right)}\right] \nonumber \\
    &= - \bE\left[\dotp{\nabla F(\vx_{k\eta}) + \vN}{\nabla \log\left(\frac{\mu_t(\vx_t)}{\pistar(\vx_t)}\right)}\right] 
\end{align}
where $\vN$ stands for $\vN(\vx_{k\eta}, \xi_k)$. We note that, 
\begin{align*}
    \bE\left[\dotp{\vN}{\nabla \log\left(\frac{\mu_t(\vx_t)}{\pistar(\vx_t)}\right)}\right] &= \bE\left[\dotp{\bE\left[\vN | \vx_t, \vx_{k\eta}\right]}{\nabla \log\left(\frac{\mu_t(\vx_t)}{\pistar(\vx_t)}\right)}\right]
\end{align*}
Substituting the above into equation \eqref{eqn:kl-differential-eq3}, we obtain,
\begin{align}
\label{eqn:kl-differential-eq4}
\int_{\bR^d} \nabla . \left(\mu_t(\vx) \bE\left[g(\vx_{k\eta}, \xi_k) | \vx_t = \vx \right]\right) \log\left(\frac{\mu_t(\vx)}{\pistar(\vx)}\right) \dd \vx &= - \bE\left[\dotp{\nabla F(\vx_{k\eta})}{\nabla \log\left(\frac{\mu_t(\vx_t)}{\pistar(\vx_t)}\right)}\right] \nonumber \\
&- \bE\left[\dotp{\bE\left[\vN | \vx_t, \vx_{k\eta}\right]}{\nabla \log\left(\frac{\mu_t(\vx_t)}{\pistar(\vx_t)}\right)}\right]
\end{align}
From equations \eqref{eqn:kl-differential-eq1}, \eqref{eqn:kl-differential-eq2} and \eqref{eqn:kl-differential-eq4}, we obtain,
\begin{align*}
    \frac{\dd}{\dd t} \KL{\mu_t}{\pistar} &= -\FD{\mu_t}{\pistar} + \bE\left[\dotp{\nabla F(\vx_t)-\nabla F(\vx_{k\eta})}{\nabla \log\left(\frac{\mu_t(\vx_t)}{\pistar(\vx_t)}\right)}\right] \\
    &- \bE\left[\dotp{\bE\left[\vN | \vx_t, \vx_{k\eta}\right]}{\nabla \log\left(\frac{\mu_t(\vx_t)}{\pistar(\vx_t)}\right)}\right] \\
    &\leq -\frac{1}{2}\FD{\mu_t}{\pistar} + \bE\left[\norm{\nabla F(\vx_t) - \nabla F(\vx_{k\eta})}^2\right] + \bE\left[\norm{\bE[\vN | \vx_{k\eta}, \vx_t]}^2\right]
\end{align*}
\end{proof}
\subsection{Proof of Lemma \ref{lem:noise-control-chisquare}}
\label{proof:lem-noise-control-chisquare}
\begin{proof}
Define $u_k = M \norm{\vx_{k\eta}} + G$ and $h = t - k\eta$. Denoting $\vN = \nicefrac{1}{B} \sum_{j=1}^{B} \vN_j$, where $\vN_j = \nabla f(\vx_{k\eta}, \xi_{k,j}) - \nabla F(\vx_{k\eta})$, we conclude from \ref{eqn:defn-grad-growth-subg} and Lemma \ref{lem:moment-growth-subg} that $\bE\left[ \norm{\vN_j}^{2m} | \vx_{k\eta}\right] \leq  2^{m+1} m! u^{2m}_k$ for any $j \in [B]$ and $m \in \mathbb{N}$.

Our proof begins by obtaining a coarse bound for $\bE_{\vx_t | \vx_{k\eta}}\left[\norm{\bE[\vN | \vx_{k\eta}, \vx_t]}^2\right]$. Using the conditional CLT structure of $\vN$ (i.e., using the fact that $\vN$ is an empirical average of zero-mean random variables that are i.i.d conditioned on $\vx_{k\eta}$), we obtain the following via Jensen's inequality.
\begin{align}
\label{eqn:stoc-noise-coarse-bound}
\bE_{\vx_t | \vx_{k\eta}}\left[\norm{\bE[\vN | \vx_{k\eta}, \vx_t]}^2\right] &\leq \bE_{\vx_t | \vx_{k\eta}}\left[\bE[\norm{\vN}^2 | \vx_{k\eta}, \vx_t]\right] = \bE\left[\norm{\vN}^2 | \vx_{k\eta}\right]  \nonumber \\ 
&= \tfrac{1}{B} \bE\left[\norm{\vN_1}^2 | \vx_{k\eta}\right] \leq \nicefrac{4 u^2_k}{B}
\end{align}
We now proceed to refine this bound as follows. Note that, since $\vN_1, \dots, \vN_B$ are i.i.d conditioned on $\vx_{k\eta}$ and $\Law{\vx_t | \vx_{k\eta}, \vN_1, \dots, \vN_B} = \cN(\vx_{k\eta}-h \nabla F(\vx_{k\eta}) - \nicefrac{h}{B} \sum_{j=1}^{B} \vN_j, 2h)$, it is easy to see that they remain i.i.d even after conditioning on $\vx_t$, due to the permutation invariance of the function $\sum_{j=1}^{B} \vN_j$ . Hence $\bE\left[\vN | \vx_{k\eta}, \vx_t\right] \disteq \bE\left[\vN_1 | \vx_{k\eta}, \vx_t\right]$ by linearity of conditional expectation. It follows by Jensen's inequality,
\begin{align}
\label{eqn:conditional-norm-control-1}
    \bE_{\vx_t | \vx_{k\eta}} \left[\norm{\bE[\vN | \vx_{k\eta}, \vx_t]}^2\right] &= \bE_{\vx_t  | \vx_{k\eta}}\left[\norm{\bE\left[ \vN_1 | \vx_{k\eta}, \vx_t\right]}^2\right] \nonumber \\
    &= \bE_{\vx_t  | \vx_{k\eta}}\left[\norm{\bE_{\vN_{2:B} | \vx_{k\eta}}\left[\bE\left[ \vN_1 | \vx_{k\eta}, \vN_{2:B}, \vx_t\right]\right]}^2\right] \nonumber \\
    &\leq \bE_{\vN_{2:B}, \vx_t  | \vx_{k\eta}}\left[\norm{\bE\left[ \vN_1 | \vx_{k\eta}, \vN_{2:B}, \vx_t\right]}^2\right]
\end{align}
Let $\rho(\vx_t | \vx_{k\eta}, \vN_1, \vN_{2:B})$ denote the density of $\Law{\vx_t | \vx_{k\eta}, \vN_1, \vN_{2:B}}$ and $\rho(\vx_t | \vx_{k\eta}, \vN_{2:B})$ denote the density of $\Law{\vx_t | \vx_{k\eta}, \vN_{2:B}}$.
\begin{align*}
    \bE\left[ \vN_1 | \vx_{k\eta}, \vN_{2:B}, \vx_t\right] &= \int_{\bR^d} \vN_1 \dd \bP_{\vN_1 | \vx_{k\eta}, \vN_{2:B}, \vx_t}(\vN_1) \\
    &= \int_{\bR^d} \vN_1 \frac{\rho(\vx_t | \vx_{k\eta}, \vN_1, \vN_{2:B})}{\rho(\vx_t | \vx_{k\eta}, \vN_{2:B})} \dd P_{\vN_1 | \vN_{2:B}, \vx_{k\eta}}(\vN_1) \\
    &= \int_{\bR^d} \vN_1 \frac{\rho(\vx_t | \vx_{k\eta}, \vN_1, \vN_{2:B})}{\rho(\vx_t | \vx_{k\eta}, \vN_{2:B})} \dd P_{\vN_1 | \vx_{k\eta}}(\vN_1) \\
    &= \int_{\bR^d} \vN_1 \left(\frac{\rho(\vx_t | \vx_{k\eta}, \vN_1, \vN_{2:B})}{\rho(\vx_t | \vx_{k\eta}, \vN_{2:B})} - 1\right) \dd P_{\vN_1 | \vx_{k\eta}}(\vN_1) \\
    &= \bE_{\vN_1 | \vx_{k\eta}}\left[\vN_1 \left(\frac{\rho(\vx_t | \vx_{k\eta}, \vN_1, \vN_{2:B})}{\rho(\vx_t | \vx_{k\eta}, \vN_{2:B})} - 1\right)\right]
\end{align*}
where the third equality uses the fact that $\vN_1, \dots, \vN_B$ are i.i.d conditioned on $\vx_{k\eta}$ and the fourth equality uses the fact that $\vN_1$ is zero-mean conditioned on $\vx_{k\eta}$. Now, applying Cauchy-Schwarz inequality,
\begin{align}
\label{eqn:conditional-norm-control-2}
    \norm{\bE\left[ \vN_1 | \vx_{k\eta}, \vN_{2:B}, \vx_t\right]}^2 &\leq \bE\left[\norm{\vN_1}^2 | \vx_{k\eta}\right] \bE\left[\left(\frac{\rho(\vx_t | \vx_{k\eta}, \vN_1, \vN_{2:B})}{\rho(\vx_t | \vx_{k\eta}, \vN_{2:B})} - 1\right)^2\right] \nonumber \\
    &\leq 4u^2_k  \bE_{\vN_1 | \vx_{k\eta}}\left[\left(\frac{\rho(\vx_t | \vx_{k\eta}, \vN_1, \vN_{2:B})}{\rho(\vx_t | \vx_{k\eta}, \vN_{2:B})}\right)^2 - 1\right]
\end{align}
Let $\tvN_1$ be an identical and independent copy of $\vN_1$ conditioned on $\vx_{k\eta}$. It follows that, 
\begin{align*}
\bE_{\vN_1 | \vx_{k\eta}}\left[\left(\frac{\rho(\vx_t | \vx_{k\eta}, \vN_1, \vN_{2:B})}{\rho(\vx_t | \vx_{k\eta}, \vN_{2:B})}\right)^2 - 1 \right] &= \bE_{\vN_1 | \vx_{k\eta}}\left[\left(\frac{\rho(\vx_t | \vx_{k\eta}, \vN_1, \vN_{2:B})}{\bE_{\tvN_1 | \vx_{k\eta}}\left[\rho(\vx_t | \vx_{k\eta}, \tvN_1, \vN_{2:B})\right]}\right)^2 - 1 \right] \\
&\leq \bE_{\vN_1, \tvN_1 | \vx_{k\eta}} \left[\left(\frac{\rho(\vx_t | \vx_{k\eta}, \vN_1, \vN_{2:B})}{\rho(\vx_t | \vx_{k\eta}, \tvN_1, \vN_{2:B})}\right)^2 - 1\right]
\end{align*}
where the last step is an application of Jensen's inequality and the convexity of $\nicefrac{1}{x^2}$. Taking expectations wrt $\vN_{2:B}, \vx_t | \vx_{k\eta}$ on both sides,
\begin{align*}
   & \bE_{\vN_1, \vN_{2:B}, \vx_t | \vx_{k\eta}}\left[\left(\frac{\rho(\vx_t | \vx_{k\eta}, \vN_1, \vN_{2:B})}{\rho(\vx_t | \vx_{k\eta}, \vN_{2:B})}\right)^2 - 1 \right] \leq \bE_{\vN_1, \tvN_1, \vN_{2:B}, \vx_t | \vx_{k\eta}} \left[\left(\frac{\rho(\vx_t | \vx_{k\eta}, \vN_1, \vN_{2:B})}{\rho(\vx_t | \vx_{k\eta}, \tvN_1, \vN_{2:B})}\right)^2 - 1\right] \\
    &= \bE_{\vN_1, \tvN_1, \vN_{2:B} | \vx_{k\eta}} \left[\bE_{\vx_t | \vx_{k\eta}, \tvN_1, \vN_{2:B}}\left[\left(\frac{\rho(\vx_t | \vx_{k\eta}, \vN_1, \vN_{2:B})}{\rho(\vx_t | \vx_{k\eta}, \tvN_1, \vN_{2:B})}\right)^2 - 1\right]\right]
\end{align*}
Since $\rho(\vx_t | \vx_{k\eta}, \vN_1, \vN_{2:B}) = \cN(\vx_{k\eta}-h \nabla F(\vx_{k\eta}) -\nicefrac{h}{B}\vN_1 - \nicefrac{h}{B} \sum_{j=2}^{B} \vN_j, 2h)$, we conclude from Lemma \ref{lem:chi-sq-gaussian-control} that,
\begin{align*}
    \bE_{\vx_t | \vx_{k\eta}, \tvN_1, \vN_{2:B}}\left[\left(\frac{\rho(\vx_t | \vx_{k\eta}, \vN_1, \vN_{2:B})}{\rho(\vx_t | \vx_{k\eta}, \tvN_1, \vN_{2:B})}\right)^2 - 1\right] = e^{\frac{h^2\norm{\vN_1 - \tvN_1}^2}{2B^2}} - 1
\end{align*}
Furthermore, by Jensen's inequality and Lemma \ref{lem:moment-growth-subg}, 
\begin{align*}
    \bE_{\vN_1, \tvN_1 | \vx_{k\eta}}\left[\norm{\vN_1-\tvN_1}^{2m} \right] \leq 2^{2m} \bE\left[\norm{\vN_1}^{2m} | \vx_{k\eta}\right] \leq 2^{3m+1} m! \ u^{2m}_k \  \ \forall \ m \in \mathbb{N} \\
\end{align*}
Hence, the exponential can be controlled as follows
\begin{align*}
    \bE_{\vN_1, \Tilde{\vN_1}, \vN_{2:B}| \vx_{k\eta}}[e^{\frac{h \norm{\vN_1 - \Tilde{\vN}_1}^2}{2B^2}} - 1] &= \bE_{\vN_1, \Tilde{\vN_1}| \vx_{k\eta}}[e^{\frac{h \norm{\vN_1 - \Tilde{\vN}_1}^2}{2B^2}} - 1] \\
    &= \sum_{m=1}^{\infty} \frac{h^m}{2^m B^{2m} m!} \bE_{\vN_1, \Tilde{\vN_1}| \vx_{k\eta}}\left[\norm{\vN_1 - \Tilde{\vN}_1}^{2m}\right] \\
    &\leq \sum_{m=1}^{\infty} \frac{h^m 2^{2m+1} u^{2m}_k}{B^{2m}} \\
    &\leq \frac{8h u^2_k}{B^2} \sum_{m=0}^{\infty} \left(\tfrac{4h u^2_k}{B^2}\right)^{m}
\end{align*}
It follows that,
\begin{align}
\label{eqn:conditional-norm-control-3}
    \bE_{\vN_1, \vN_{2:B}, \vx_t | \vx_{k\eta}}\left[\left(\frac{\rho(\vx_t | \vx_{k\eta}, \vN_1, \vN_{2:B})}{\rho(\vx_t | \vx_{k\eta}, \vN_{2:B})}\right)^2 - 1 \right] \leq \frac{8h u^2_k}{B^2} \sum_{m=0}^{\infty} \left(\tfrac{4h u^2_k}{B^2}\right)^{m}
\end{align}
From equations \eqref{eqn:conditional-norm-control-1}, \eqref{eqn:conditional-norm-control-2} and \eqref{eqn:conditional-norm-control-3}, we obtain the following finer bound, 
\begin{align}
\label{eqn:stoc-noise-fine-bound}
\bE_{\vx_t | \vx_{k\eta}}\left[\norm{\bE[\vN | \vx_{k\eta}, \vx_t]}^2\right] \leq \frac{8h u^4_k}{B^2} \sum_{m=0}^{\infty} \left(\tfrac{4h u^2_k}{B^2}\right)^{m}
\end{align}
Thus, from equations \eqref{eqn:stoc-noise-coarse-bound} and \eqref{eqn:stoc-noise-fine-bound}, it follows that,
\begin{align*}
    \bE_{\vx_t | \vx_{k\eta}}\left[\norm{\bE[\vN | \vx_{k\eta}, \vx_t]}^2\right] \leq \min \left\{\frac{8h u^4_k}{B^2} \sum_{m=0}^{\infty} \left(\tfrac{4h u^2_k}{B^2}\right)^{m}, \frac{4u^2_k}{B} \right\}
\end{align*}
Note that, when $u^2_k \leq \nicefrac{B}{6h}$, the following holds,
\begin{align*}
    \frac{8h u^4_k}{B^2} \sum_{m=0}^{\infty} \left(\tfrac{4h u^2_k}{B^2}\right)^{m} &\leq \frac{8h u^4_k}{B^2} \sum_{m=0}^{\infty} \left(\nicefrac{2}{3}\right)^{m} \leq \frac{24 h u^4_k}{B^2} \leq \frac{4u^2_k}{B}
\end{align*}
Hence, the coarse bound and fine bound are combined as,
\begin{align*}
    \bE_{\vx_t | \vx_{k\eta}}\left[\norm{\bE[\vN | \vx_{k\eta}, \vx_t]}^2\right] &\leq \frac{4u^2_k}{B} \bEvent{u^2_k > \nicefrac{B}{6h}} + \left[\frac{8h u^4_k}{B^2} \sum_{m=0}^{\infty} \left(\tfrac{4h u^2_k}{B^2}\right)^{m}\right] \bEvent{u^2_k \leq \nicefrac{B}{6h}} \\
    &\leq \frac{24h u^4_k}{B^2} + \frac{4u^2_k}{B} \bEvent{u^2_k > \nicefrac{B}{6h}}
\end{align*}
Since the \ref{eqn:defn-moment-growth} condition holds with $p=4$, and $\Law{\vx_{k\eta}} = \Law{\hvx_k}$ by construction of the interpolating process,
\begin{align*}
    \bE_{\vx_{k\eta}}\left[u^4_k\right] \leq 8\left(M^4 \bE\left[\norm{\hvx_k}^4\right] + G^4\right) \leq 8\left(M^4 C_4 d^2 + G^4\right)
\end{align*}
Furthermore, applying Lemma \ref{lem:tail-markov-control},
\begin{align*}
    \bE\left[u^2_k \bEvent{u^2_k > \nicefrac{B}{6h}}\right] \leq \frac{12h}{B}\bE[u^4_k]
\end{align*}
Hence, it follows that,
\begin{align*}
    \bE\left[\norm{\bE[\vN | \vx_{k\eta}, \vx_t]}^2\right] &\leq \frac{24h}{B^2} \bE\left[u^4_k\right] + \bE\left[\frac{4u^2_k}{B} \bEvent{u^2_k > \nicefrac{B}{6h}}\right] \\
    &\leq \frac{72 h}{B^2} \bE\left[u^4_k\right] \\
    &\leq \frac{576(t - k\eta)(M^4 C_{4} d^{2 } + G^4)}{B^2}
\end{align*}
\end{proof}
\subsection{Proof of Lemma \ref{lem:sgld-descent-lemma}}
\label{proof:lem-sgld-descent-lemma}
\begin{proof}
From Lemmas \ref{lem:kl-flow-interpolating} and \ref{lem:noise-control-chisquare}, we infer that,
\begin{align}
\label{eqn:sgld-descent-mid-diffeq}
    \ddt \KL{\mu_t}{\pistar} \leq -\frac{1}{2}\FD{\mu_t}{\pistar} + \bE\left[\norm{\nabla F(\vx_t) - \nabla F(\vx_{k\eta})}^2\right] + \frac{576(t - k\eta)(M^4 C_{4} d^{2 } + G^4)}{B^2}
\end{align}
The remainder of this proof controls the discretization error term $\bE\left[\norm{\nabla F(\vx_t) - \nabla F(\vx_{k\eta})}^2\right]$. To this end, let $h = t - k \eta$. We note that, $0 \leq h \leq \eta \leq \nicefrac{1}{6L}$. Furthermore, via direct integration of the interpolating process, we infer,
\begin{align*}
    \vx_t \disteq \vx_{k\eta} - h (\nabla F(\vx_{k\eta}) + \vN) + \sqrt{2h} \vz_t, \ \vz_t \sim \cN(0, \vI)
\end{align*}
Hence, by the $L$-smoothness of $F$,
\begin{align*}
    \bE\left[\norm{\nabla F(\vx_{t}) - \nabla F(\vx_{k \eta})}^2\right] &\leq L^2 \bE\left[ \norm{\vx_t - \vx_{k \eta}}^2 \right] \\
    &\leq L^2 h^2 \bE\left[\norm{\nabla F(\vx_{k \eta})}^2\right] + L^2 h^2 \bE\left[\norm{\vN}^2\right] + 2L^2 h d 
\end{align*}
We now rewrite the RHS in terms of $\vx_t$. To this end, we note that by the $L$-lipschitzness of $\nabla F$,
\begin{align}
\label{eqn:sgld-descent-disc-term}
    \norm{\nabla F(\vx_{k\eta})} &\leq \norm{\nabla F(\vx_t)} + L \norm{\vx_t - \vx_{k \eta}} \nonumber \\
    &\leq \norm{\nabla F(\vx_t)} + L h \norm{\nabla F(\vx_{k \eta})} + L h \norm{\vN} + L\sqrt{2h}\norm{\vz_t}
\end{align}
Using the fact that $h \leq \eta < \nicefrac{1}{3L}$ and rearranging, 
\begin{align*}
    \norm{\nabla F(\vx_{k\eta})} &\leq \frac{3}{2}\norm{\nabla F(\vx_t)} + \frac{1}{2}\norm{\vN} + \frac{3L \sqrt{h}}{\sqrt{2}}\norm{\vz_t} \\
    \norm{\nabla F(\vx_{k\eta})}^2 &\leq \frac{27}{4}\norm{\nabla F(\vx_t)}^2 + \frac{3}{4}\norm{\vN}^2 + \frac{27L^2 h}{2}\norm{\vz_t}^2 
\end{align*}
Substituting the above into \eqref{eqn:sgld-descent-disc-term}, we get,
\begin{align*}
    \bE\left[ \norm{\nabla F(\vx_t) - \nabla F(\vx_{k\eta})}^2 \right] &\leq \frac{27 L^2 h^2}{4} \bE\left[\norm{\nabla F(\vx_t)}^2\right] + 2L^2 h^2 \bE\left[\norm{\vN}^2\right] + \frac{27 L^4 h^3 d}{2} + 2L^2hd 
\end{align*}
Using the fact that $\bE\left[\norm{\nabla F(\vx_t)}^2\right] \leq \FD{\mu_t}{\pistar} + 2dL$ as per Lemma \ref{lem:grad-fd-bound}, and $h \leq \nicefrac{1}{6L}$, we get
\begin{align*}
    \bE\left[ \norm{\nabla F(\vx_t) - \nabla F(\vx_{k\eta})}^2 \right] &\leq \frac{27 L^2 h^2}{4} \FD{\rho_t}{\gamma} + 2L^2 h^2 \bE\left[\norm{\vN}^2\right] + 5 L^2 h d \\
    &\leq \frac{1}{4} \FD{\mu_t}{\pistar} + 2L^2 h^2 \bE\left[\norm{\vN}^2\right] + 5 L^2 h d 
\end{align*}
Recall that, due to the \ref{eqn:defn-grad-growth-subg} and \ref{eqn:defn-moment-growth} conditions, the following holds as a consequence of Lemma \ref{lem:moment-growth-subg}.
\begin{align*}
    \bE\left[\norm{\vN}^2\right] \leq \frac{4}{B}\bE\left[\left(M \norm{\hvx_k} + G\right)^2\right] \leq \frac{8}{B} \left(M^2 C_2 d + G^2\right)
\end{align*}
Hence, 
\begin{align*}
    \bE\left[ \norm{\nabla F(\vx_t) - \nabla F(\vx_{k\eta})}^2 \right] &\leq \frac{1}{4} \FD{\mu_t}{\pistar} + \frac{3Lh}{B}\left(M^2 C_{2} d + G^2\right) + 5 L^2 h d
\end{align*}
Substituting the above into \eqref{eqn:sgld-descent-mid-diffeq},
\begin{align*}
    \ddt \KL{\mu_t}{\pistar} &\leq -\frac{1}{4}\FD{\mu_t}{\pistar} + 5L^2 (t - k \eta)d + \frac{3L(t - k \eta)(M^2 C_{2} d + G^2)}{B} \\
    &+ \frac{576(t - k \eta) (M^4 C_{4} d^{2 } + G^4)}{B^2}
\end{align*}
\end{proof}
\subsection{Proof of Theorem \ref{thm:sgld-lsi-stable}}
\label{proof:thm-sgld-lsi-stable}
\begin{proof}
Since $\pistar$ satisfies \ref{eqn:defn-log-sobolev}, we know that $\KL{\mu_t}{\pistar} \leq \frac{1}{2 \CLSI}\FD{\mu_t}{\pistar}$. Thus, from Lemma \ref{lem:sgld-descent-lemma}, we conclude,
\begin{align*}
    \ddt \KL{\mu_t}{\pistar} &\leq -\frac{\CLSI}{2}\KL{\mu_t}{\pistar} + 5L^2 (t - k \eta)d + \frac{3L(t - k \eta)(M^2 C_{2} d + G^2)}{B} \\
    &+ \frac{576(t - k \eta) (M^4 C_{4} d^{2} + G^4)}{B^2}
\end{align*}
Multiplying both sides by $e^{\nicefrac{\CLSI (t-k\eta)}{2}}$ and using the fact that $k \eta \leq t \leq (k+1)\eta$,
\begin{align*}
    \ddt \left(e^{\nicefrac{\CLSI (t - k \eta)}{2}} \KL{\mu_t}{\pistar}  \right) &\leq e^{\nicefrac{\CLSI (t - k\eta)}{2}} \left[5L^2 \eta d + \frac{3L\eta(M^2 C_{2} d + G^2)}{B} + \frac{576\eta (M^4 C_{4} d^{2 } + G^4)}{B^2} \right]
\end{align*}
Applying Gr\"{o}nwall's Lemma for $t \in [k\eta, (k+1)\eta]$
\begin{align*}
   & e^{\nicefrac{\CLSI \eta}{2}} \KL{\mu_{(k+1)\eta}}{\pistar} - \KL{\mu_{k\eta}}{\pistar} \nonumber\\
   &\leq \frac{2}{\CLSI}\left(e^{\nicefrac{\CLSI \eta}{2}} - 1\right) \left[5L^2 \eta d + \frac{3L\eta(M^2 C_{2} d + G^2)}{B} + \frac{576\eta (M^4 C_{4} d^{2} + G^4)}{B^2} \right]  \\
    &\leq 5.25L^2 \eta^2 d + \frac{3.15L\eta^2(M^2 C_{2} d + G^2)}{B} + \frac{604.8\eta^2 (M^4 C_{4} d^{2} + G^4)}{B^2}
\end{align*}
Where the last inequality uses the fact that $\CLSI \eta \leq \nicefrac{\CLSI^2}{6L^2} \leq \nicefrac{1}{6}$, and $e^{\nicefrac{x}{2}} - 1 \leq 0.525x$ for $x \leq \nicefrac{1}{6}$. Furthermore, recalling the fact that $\mu_{k\eta} = \Law{\hvx_k}$, we obtain the following descent lemma for the discrete-time iterates of SGLD,
\begin{align*}
    \KL{\Law{\hvx_{k+1}}}{\pistar} &\leq e^{-\nicefrac{\CLSI \eta}{2}} \KL{\Law{\hvx_{k+1}}}{\pistar} \\
    &+ e^{-\nicefrac{\CLSI \eta}{2}} \left[5.25L^2 \eta^2 d + \frac{3.15L\eta^2(M^2 C_{2} d + G^2)}{B} + \frac{604.8\eta^2 (M^4 C_{4} d^{2 } + G^4)}{B^2} \right]
\end{align*}
Iterating through the above lemma, we get,
\begin{align*}
    \KL{\Law{\hvx_{K}}}{\pistar} &\leq e^{-\nicefrac{\CLSI \eta K}{2}} \KL{\Law{\hvx_{0}}}{\pistar} \\
    &+ \frac{e^{-\nicefrac{\CLSI \eta}{2}}}{1 - e^{-\nicefrac{\CLSI \eta}{2}}} \left[5.25L^2 \eta^2 d + \frac{3.15L\eta^2(M^2 C_{2} d + G^2)}{B} + \frac{604.8\eta^2 (M^4 C_{4} d^{2 } + G^4)}{B^2} \right]
\end{align*}
Using the fact that $e^x \geq 1 + x$, we obtain,
\begin{align*}
    \KL{\Law{\hvx_{K}}}{\pistar} &\leq e^{-\nicefrac{\CLSI \eta K}{2}} \KL{\Law{\hvx_{0}}}{\pistar} + \frac{11L^2 \eta d}{\CLSI}\nonumber\\&\quad + \frac{7L\eta(M^2 C_{2} d + G^2)}{B\CLSI} + \frac{1210\eta (M^4 C_{4} d^{2 } + G^4)}{\CLSI B^2} 
\end{align*}
\end{proof}
\subsection{Analysis of AB-SGLD}
\label{app-sec:absgld-smooth-stable-convergence}
Recall that $F$ has a finite-sum structure of the form $F(\vx) = \sum_{i=1}^{n} f_i(\vx)$. Similar to our analysis of smooth SGLD, we construct the following interpolating process $(\vx_t)_{t \in [0, K \eta]}$
\begin{align*}
    \Law{\vx_0} &= \Law{\hvx_0} \ \ \xi_{k,1}, \dots, \xi_{k, n} \ \iidsim \ \mathsf{Unif}[n] \ \forall k \in [K]  \\
    \dd \vx_t &= -\left[\nabla F(\vx_{k\eta}) + \vN(\vx_{k\eta}, \xi_k)\right] \dd t + \sqrt{2} \dd B_t, \ t \in [k\eta, (k+1)\eta], \ k \in [K]
\end{align*}
where $B_t$ is a Brownian motion independent of $\vx_0$ and $\vN(\vx_{k\eta}, \xi_k)$ is defined as,
\begin{align*}
    \vN(\vx_{k\eta}, \xi_k) &= \frac{1}{B_k} \sum_{j=1}^{B_k} f_{\xi_{k,j}}(\hvx_k) - \nabla F(\hvx_k) \\
    B_k &= \min \{ n, 1 + \left \lceil{M \norm{\hvx_k} + G}\right \rceil  \}
\end{align*}
Since $\Law{\vx_0} = \Law{\hvx_0}$, and $B_k$ is a deterministic function of $\vx_{k\eta}$, writing a closed form for $\Law{\vx_t | \vx_{k\eta}}$ and performing an inductive argument shows that $\Law{\vx_{k\eta}} = \Law{\hvx_k}$, i.e., the above stochastic process is an interpolating process for AB-SGLD. As before, let $\mu_t = \Law{\vx_t}$. Our first step is to analyze the time-evolution of $\KL{\mu_t}{\pistar}$. 
\begin{lemma}[ABSGLD : Flow of KL along Interpolating Process]
\label{lem:ab-sgld-kl-flow-interpolating}
Assume $\KL{\mu_0}{\pistar} < \infty$. Then, the following differential inequality is satisfied for any $k \in (K)$ and $t \in [k\eta, (k+1)\eta]$,
\begin{align*}
    \ddt \KL{\mu_t}{\pistar} \leq -\frac{1}{2}\FD{\mu_t}{\pistar} + \bE\left[\norm{\nabla F(\vx_t) - \nabla F(\vx_{k\eta})}^2\right] + \bE\left[\norm{\bE[\vN | \vx_{k\eta}, \vx_t]}^2\right]
\end{align*}
\end{lemma}
\begin{proof}
Since $B_k$ is a deterministic function of $\vx_{k\eta}$, this lemma can be considered a special case of Lemma \ref{lem:kl-flow-interpolating} with $\Xi = [n]$ and $P_{\xi} = \mathsf{Unif}[n]$. The proof is completed by applying the same arguments. 
\end{proof}
We shall now use the adaptive batch-size property to control the stochastic gradient error term. Our analysis crucially uses smoothness and LSI.  
\begin{lemma}[AB-SGLD: Controlling Stochastic Gradient Error]
\label{lem:absgld-noise-control}
Let the \ref{eqn:defn-grad-growth} condition be satisfied and let $\pistar$ satisfy \ref{eqn:defn-log-sobolev}. Then, the following holds for any $t \in [k\eta, (k+1)\eta], \ k \in (K), \ \eta \leq 1$.
\begin{align*}
    \bE\left[\norm{\bE\left[\vN \ | \ \vx_t, \vx_{k\eta} \right]}^2\right] &\leq \frac{128 (t - k \eta) M^2}{\CLSI} \KL{\mu_{k\eta}}{\pistar} + 32 (t - k\eta) \left(2M^2 \vm^2_2 + G^2 \right)
\end{align*}
where $\vm^2_2 = \bE_{\pistar}\left[\norm{\vx}^2\right]$
\end{lemma}
\begin{proof}
Let $u_k = M \norm{\vx_{k\eta}} + G$ and $h = t - k \eta$.  Denoting $\vN = \nicefrac{1}{B} \sum_{j=1}^{B_k} \vN_j$, where $\vN_j = \nabla f(\vx_{k\eta}, \xi_{k,j}) - \nabla F(\vx_{k\eta})$,
we conclude from \ref{eqn:defn-grad-growth-subg} and Lemma \ref{lem:moment-growth-subg} that $\bE\left[ \norm{\vN_j}^{2m} | \vx_{k\eta}\right] \leq  2^{m+1} m! u^{2m}_k$ for any $j \in [B]$ and $m \in \mathbb{N}$. Furthermore, since $B_k$ is a deterministic function of $\vx_{k\eta}$ and $\vN_1, \dots, \vN_{B_k}$ are zero-mean i.i.d conditioned on $\vx_{k\eta}$, we can repeat the same arguments as Lemma \ref{lem:noise-control-chisquare} to obtain the following bound,
\begin{align*}
    \bE_{\vx_t \ | \ \vx_{k\eta}}\left[\norm{\bE\left[\vN \ | \ \vx_t, \vx_{k\eta} \right]}^2\right] &\leq \frac{8h u^4_k}{B^2_k} \sum_{m=0}^{\infty} \left(\tfrac{4h u^2_k}{B^2_k}\right)^{m} 
    \leq 16 h u^2_k \leq 32 h \left(M^2 \norm{\vx_{k\eta}}^2 + G^2\right)   
\end{align*}
where the second inequality follows from the fact that $B_k \geq u_k$ and $h \leq \eta \leq \nicefrac{1}{8}$. Now, applying Lemma \ref{lem:moment-control-talagrand}, we obtain
\begin{align*}
    \bE\left[\norm{\bE\left[\vN \ | \ \vx_t, \vx_{k\eta} \right]}^2\right] &\leq 32h \left( M^2 \bE\left[\norm{\vx_{k\eta}}^2\right] + G^2 \right) \\
    &\leq \frac{128 M^2 h}{\CLSI} \KL{\mu_{k\eta}}{\pistar} + 32h \left(2M^2 \vm^2_2 + G^2\right)
\end{align*}
\begin{comment}
It follows that, 
\begin{align*}
    \bE\left[\norm{\bE\left[\vN \ | \ \vx_t, \vx_{k\eta} \right]}^2\right] &\leq 4h \left( L^2 \bE\left[\norm{\vx_{k\eta}}^2\right] + G^2 \right) 
\end{align*}
To compute the expectation, consider the $\wass{2}$-optimal coupling $(\vx_{k\eta}, \vx^{*})$ between $\mu_{k\eta}$ and $\pistar$. Then, 
\begin{align*}
    \bE\left[\norm{\vx_{k\eta}}^2\right] &\leq 2\bE\left[\norm{\vx_{k\eta} - \vx^*}^2\right] + 2 \bE\left[\norm{\vx^*}^2\right] \\
    &\leq 2 \wass{2}^2 \left( \mu_{k\eta}, \pistar\right) + \vm^2_2 \\
    &\leq \frac{4}{\CLSI} \KL{\mu_{k\eta}}{\pistar} + \vm^2_2
\end{align*}
Hence, 
\begin{align*}
    \bE\left[\norm{\bE\left[\vN \ | \ \vx_t, \vx_{k\eta} \right]}^2\right] &\leq \frac{16 h L^2}{\CLSI} \KL{\mu_{k\eta}}{\pistar} + 4h \left(L^2 \vm^2_2 + G^2 \right)
\end{align*}
\end{comment}
\end{proof}
\subsubsection{Proof of Theorem \ref{thm:absgld-lsi-stable}}
\label{proof:thm-absgld-lsi-stable}
\begin{proof}
Consider any $k \in (K)$ and $t \in [k\eta, (k+1)\eta]$. Furthermore, let $h = t - k \eta \leq \eta \leq \nicefrac{1}{6L}$. From Lemma \ref{lem:ab-sgld-kl-flow-interpolating} and Lemma \ref{lem:absgld-noise-control}, it follows that,
\begin{align*}
    \ddt \KL{\mu_t}{\pistar} &\leq -\frac{1}{2}\FD{\mu_t}{\pistar} + \bE\left[\norm{\nabla F(\vx_t) - \nabla F(\vx_{k\eta})}^2\right] \nonumber\\ &\quad+ \frac{128 M^2 h}{\CLSI} \KL{\mu_{k\eta}}{\pistar} + 32h \left(2M^2 \vm^2_2 + G^2 \right)
\end{align*}
To control $\bE\left[\norm{\nabla F(\vx_t) - \nabla F(\vx_{k\eta})}^2\right]$, we recall that $h \leq \nicefrac{1}{6L}$ and follow the same steps as Lemma \ref{lem:sgld-descent-lemma} to obtain,
\begin{align*}
    \bE\left[ \norm{\nabla F(\vx_t) - \nabla F(\vx_{k\eta})}^2 \right] &\leq \frac{1}{4} \FD{\mu_t}{\pistar} + 2L^2 h^2 \bE\left[\norm{\vN}^2\right] + 5 L^2 h d 
\end{align*}
We now control $\bE\left[\norm{\vN}^2\right]$. From Lemma \ref{lem:moment-growth-subg}, we obtain the following.
\begin{align*}
    \bE\left[\norm{\vN}^2 | \vx_{k\eta}\right] &\leq \frac{4(M \norm{\vx_{k\eta}} + G)^2}{B_k} \leq 4M \norm{\vx_{k\eta}} + 4G
\end{align*}
Now, applying Lemma \ref{lem:moment-control-talagrand},
\begin{align*}
    \bE\left[\norm{\vN}^2\right] &\leq 4M\sqrt{\frac{2}{\CLSI} \KL{\mu_{k\eta}}{\pistar}} + 4M\vm_1 + 4G \\
    &\leq 4M + \frac{2M}{ \CLSI} \KL{\mu_{k\eta}}{\pistar} + 4M\vm_1 + 4G 
\end{align*}
where the last inequality uses $\sqrt{x} \leq 1 + \nicefrac{x}{4}$ for any $x \geq 0$. Hence, 
\begin{align*}
    \bE\left[ \norm{\nabla F(\vx_t) - \nabla F(\vx_{k\eta})}^2 \right] &\leq \frac{1}{4} \FD{\mu_t}{\pistar} + \frac{4 L^2 h^2 M}{ \CLSI} \KL{\mu_{k\eta}}{\pistar} \\
    &+ 5 L^2 h d + 8L^2 M h^2 + 8L^2 M h^2\vm_1 + 8L^2 h^2 G
\end{align*}
Using $h \leq \eta \leq \nicefrac{1}{6L}$, it follows that, 
\begin{align*}
\ddt \KL{\mu_t}{\pistar} &\leq -\frac{1}{4}\FD{\mu_t}{\pistar} + \frac{(LM + 128M^2) h}{\CLSI}\KL{\mu_{k\eta}}{\pistar} \\
&+ 32h \left(L^2 d + 2M^2 \vm^2_2 + G^2 \right) + 8L^2 h^2 \left(M\vm_1 + M + G \right) \\
&\leq -\frac{\CLSI}{2} \KL{\mu_t}{\pistar} + \frac{(LM + 128M^2) \eta}{\CLSI}\KL{\mu_{k\eta}}{\pistar} \\
&+ 32\eta \left(L^2 d + 2M^2 \vm^2_2 + G^2 \right) + 8L^2 \eta^2 \left(M\vm_1 + M + G \right)
\end{align*}
Multiplying both sides by $e^{\nicefrac{\CLSI (t - k \eta)}{2}}$, applying Gr\"{o}nwall's Lemma for $t \in [k\eta, (k+1)\eta]$ and using the fact that $e^{\nicefrac{x}{2}} - 1 \leq x$ for $x \leq \nicefrac{1}{6}$,
\begin{align*}
    e^{\nicefrac{\CLSI \eta}{2}} \KL{\mu_{(k+1)\eta}}{\pistar}  - \KL{\mu_{k\eta}}{\pistar} &\leq \frac{(2LM + 256M^2) \eta^2}{\CLSI}\KL{\mu_{k\eta}}{\pistar} \\
    &+ 64\eta^2 \left(L^2 d + 2M^2 \vm^2_2 + G^2 \right) + 16L^2 \eta^3 \left(M\vm_1 + M + G \right)
\end{align*}
It follows that
\begin{align*}
    \KL{\mu_{(k+1)\eta}}{\pistar} &\leq e^{-\nicefrac{\CLSI \eta}{2}}\left(1 + \frac{(2LM + 256M^2) \eta^2}{\CLSI}\right)\KL{\mu_{k\eta}}{\pistar} \\
    &+ e^{-\nicefrac{\CLSI \eta}{2}} \left[64\eta^2 \left(L^2 d + 2M^2 \vm^2_2 + G^2 \right) + 16L^2 \eta^3 \left(M\vm_1 + M + G \right)\right] \\
    &\leq e^{-\nicefrac{\CLSI \eta}{4}}\KL{\mu_{k\eta}}{\pistar} \nonumber\\&\quad+ e^{-\nicefrac{\CLSI \eta}{4}} \left[64\eta^2 \left(L^2 d + 2M^2 \vm^2_2 + G^2 \right) + 16L^2 \eta^3 \left(M\vm_1 + M + G \right)\right]
\end{align*}
Where the last inequality follows from the fact that $\eta \leq \frac{\CLSI^2}{8L\left(LM+128M^2\right)}$, implies $1 + \frac{(2LM + 256 M^2) \eta^2}{\CLSI} \leq 1 + \frac{\eta \CLSI}{4} \leq e^{\nicefrac{\eta \CLSI}{4}}$. Iterating through the above recurrence, we conclude that for any $k \in (K)$,
\begin{align}
\label{eqn:absgld-lsi-recurrence}
    \KL{\mu_{(k+1)\eta}}{\pistar} &\leq e^{-\nicefrac{\CLSI \eta k}{4}}\KL{\mu_{0}}{\pistar} \nonumber \\ &\quad+ \frac{e^{-\nicefrac{\CLSI \eta }{4}}}{1 - e^{-\nicefrac{\CLSI \eta }{4}}} \left[64\eta^2 \left(L^2 d + 2M^2 \vm^2_2 + G^2 \right) + 16L^2 \eta^3 \left(M\vm_1 + G \right)\right] \nonumber \\
    &\leq e^{-\nicefrac{\CLSI \eta k}{4}}\KL{\mu_{0}}{\pistar} + \frac{256\eta}{\CLSI} \left(L^2 d + 2M^2 \vm^2_2 + G^2 \right) + \frac{64L^2 \eta^2}{\CLSI} \left(M\vm_1 + M + G \right)
\end{align}
Recalling that $\mu_{k\eta} = \Law{\hvx_k}$, the desired last iterate guarantee is,
\begin{align*}
    \KL{\Law{\hvx_{K+1}}}{\pistar} &\leq e^{-\nicefrac{\CLSI \eta K}{4}}\KL{\mu_{0}}{\pistar} + \frac{256\eta}{\CLSI} \left(L^2 d + 2M^2 \vm^2_2 + G^2 \right) + \frac{64L^2 \eta^2}{\CLSI} \left(M\vm_1 +  M + G \right)
\end{align*}
We now control the expected amortized batch size $\Bar{B} = \nicefrac{1}{K} \sum_{k=1}^{K} \bE\left[B_k\right]$. We note that,
\begin{align*}
    \bE\left[B_k\right] &\leq 1 + \bE\left[1 + \left \lceil{M \norm{\hvx_k} + G}\right \rceil\right] \leq 2 + M \bE\left[\norm{\vx_{k\eta}}\right] + G \\
    &\leq 2 + \frac{M\sqrt{2}}{\sqrt{\CLSI}} \sqrt{\KL{\mu_{k\eta}}{\pistar}} + G
\end{align*}
From the above inequality and \eqref{eqn:absgld-lsi-recurrence}, it follows that,
\begin{align*}
    \Bar{B} &= \frac{1}{K} \sum_{k=1}^{K} \bE\left[B_k\right] \\
    &\leq 2 + G + \frac{M\sqrt{2 \KL{\Law{\hvx_0}}{\pistar}}}{K\sqrt{\CLSI}}\sum_{k=1}^{K} e^{-\nicefrac{\CLSI \eta (k-1)}{8}} \\
    &+ \frac{28M \sqrt{\eta}}{\CLSI}\left(L \sqrt{d} + M \vm_2 + G\right) + \frac{8L \eta}{\CLSI} \sqrt{2 M \vm_1 + G} \\
    &\leq 2 + G + \frac{50 M}{\CLSI^{\nicefrac{3}{2}}\eta K}\sqrt{\KL{\Law{\hvx_0}}{\pistar}} + \frac{28M \sqrt{\eta}}{\CLSI}\left(L \sqrt{d} + M \vm_2 + G\right) + \frac{8L \eta}{\CLSI} \sqrt{2 M \vm_1 + 2M  + 2G}
\end{align*}
\end{proof}

\subsection{Proof of Theorem \ref{thm:sgld-fd-and-pi}}
\label{proof:thm-sgld-fd-and-pi}
\begin{proof}
From Lemma \ref{lem:sgld-descent-lemma}, and using the fact that $k \eta \leq t \leq (k+1)\eta$, we obtain,
\begin{align*}
    \ddt \KL{\mu_t}{\pistar} &\leq -\frac{1}{4}\FD{\mu_t}{\pistar} + 5L^2 \eta d + \frac{3L\eta(M^2 C_{2} d + G^2)}{B} + \frac{576\eta (M^4 C_{4} d^{2 } + G^4)}{B^2}
\end{align*}
Integrating from $t \in [k\eta, (k+1)\eta]$, and rearranging, we get,
\begin{align*}
    \int_{k\eta}^{(k+1)\eta} \FD{\mu_t}{\pistar} &\leq 4\left[\KL{\mu_{k\eta}}{\pistar} - \KL{\mu_{(k+1)\eta}}{\pistar}\right] \\
    &+ 20L^2 \eta^2 d + \frac{12L\eta^2(M^2 C_{2} d + G^2)}{B} + \frac{2304\eta^2 (M^4 C_{4} d^{2 } + G^4)}{B^2}
\end{align*}
Averaging the above inequality for $k \in (K)$ and using the fact that $\mu_{k\eta} = \Law{\hvx_k}$, we obtain.
\begin{align*}
    \frac{1}{K \eta} \int_{0}^{K\eta} \FD{\mu_t}{\pistar} \leq \frac{4\KL{\Law{\hvx_0}}{\pistar}}{K \eta} + 20L^2 \eta d + \frac{12L\eta(M^2 C_{2} d + G^2)}{B} + \frac{2304\eta (M^4 C_{4} d^{2 } + G^4)}{B^2}
\end{align*}
Using the fact that $\FD{\mu}{\pi}$ is a convex functional of $\mu$ for any probability measure $\pi$ \citep{wu2000-fisher-convexity}, we conclude that $\FD{\mubar_{K\eta}}{\pistar} \leq \frac{1}{K \eta} \int_{0}^{K\eta} \FD{\mu_t}{\pistar}$. Hence, we obtain the following guarantee.
\begin{align*}
    \FD{\mubar_{K\eta}}{\pistar} \leq \frac{4\KL{\Law{\hvx_0}}{\pistar}}{K \eta} + 20L^2 \eta d + \frac{12L\eta(M^2 C_{2} d + G^2)}{B} + \frac{2304\eta (M^4 C_{4} d^{2 } + G^4)}{B^2}
\end{align*}
We now consider the case when $\pistar$ satisfies \ref{eqn:defn-poincare}. From Lemma \ref{lem:gullin-fd-info-transport}, we conclude that,
\begin{align*}
     \TV(\mubar_{K\eta}, \pistar)^2 \leq \frac{16 \KL{\Law{\hvx_0}}{\pistar}}{\CPI K \eta} + \frac{80  L^2 \eta d}{\CPI} + \frac{48  L\eta(M^2 C_{2} d + G^2)}{\CPI B} + \frac{9216  \eta (M^4 C_{4} d^{2 } + G^4)}{\CPI B^2}
\end{align*}
\end{proof}
% \subsection{Oracle Complexity of Theorem \ref{thm:sgld-lsi-stable}}
% \label{oracle:thm-sgld-lsi-stable}

% \subsection{Oracle Complexity of Theorem \ref{thm:absgld-lsi-stable}}
% \label{oracle:thm-absgld-lsi-stable}
% \subsection{Oracle Complexity of Theorem \ref{thm:sgld-fd-and-pi}}
% \label{oracle:thm-sgld-fd-pi}

\section{Improved Wasserstein CLT with Gaussian Convolutions}
\label{proof:lem-wass-clt-new}
We now prove the $\wass{2}$ CLT for Gaussian convolutions of bounded random vectors, as stated in Lemma \ref{lem:wass-clt-new}. Our proof is an extension of the high-dimensional $\wass{2}$ CLT of \cite{zhai2018clt}, whose proof structure we closely follow and adapt. In particular, the presence of the Gaussian convolution leads to improved regularity, which allows us to derive sharper bounds than \cite{zhai2018clt} by adapting their proof. Without loss of generality, we begin by assuming that $\Sigma_{\vY}$ is diagonal with entries $\varsigma_i^2$ for $i \in [d]$. Under our assumptions, $\varsigma^2_i \leq \nicefrac{1}{5d}$ for every $i \in [d]$, a fact which is crucial to our proof. Let $\hvX = \sqrt{\vI - \Sigma_{\vY}} \vX + \vY$. Then, $\hvX$ can be written as,
\begin{align*}
    \hvX \disteq \frac{1}{\sqrt{B}} \sum_{k=1}^{B} \sqrt{\vI - \Sigma_{\vY}} \vX_k + \vY_k
\end{align*}
where $\vX_k \ \iidsim \cN(0, \vI)$ are sampled independently of $\vY_1, \dots, \vY_B$. Our proof relies on the following key technical lemma, which is analogous to Lemma 1.6 of \cite{zhai2018clt}. We present a proof of this lemma in Appendix \ref{proof:basic-lem-clt}
\begin{lemma}\label{lem:basic_lem_clt}
$$\wass{2}\left(\vX_1,\sqrt{I-\frac{\Sigma_Y}{k}}\vX_1 + \frac{1}{\sqrt{k}}\vY_1 \right)\leq \sqrt{\frac{25d\beta^6}{k^3}}$$
\end{lemma}
\subsection{Proof of Lemma \ref{lem:wass-clt-new}}
\begin{proof}
The proof proceeds similar to the proof of Theorem 1.1 in \cite{zhai2018clt}, using Lemma~\ref{lem:basic_lem_clt} as the key ingredient (analogous to Lemma 1.6 in \cite{zhai2018clt}). First, by the property that a sum of independent Gaussian random variables is also a Gaussian random variable, we have:
\begin{align}
    \wass{2}(\sum_{j=1}^{k} \vX_j, \sum_{j=1}^{k-1}\vX_j + \sqrt{\vI-\Sigma_Y}\vZ_k + \vY_k) &= \wass{2}\left(\sqrt{k}\vZ_1,\sqrt{k}\sqrt{\vI-\frac{\Sigma_Y}{k}}\vZ_1 + \vY_k\right) \nonumber \\
    &=\sqrt{k} \wass{2}\left(\vZ_1,\sqrt{\vI-\frac{\Sigma_Y}{k}}\vZ_1 + \frac{\vY_k}{\sqrt{k}}\right) \nonumber \\
    &\leq \frac{\sqrt{25\beta^6 d}}{k}
\end{align}
In the third step, we have applied Lemma~\ref{lem:basic_lem_clt}. Now, by triangle inequality, we must have:
\begin{align}
    &\sqrt{B}\wass{2}\left(\vX,\hvX\right) = \wass{2}(\sum_{j=1}^{B}\hvX_j,\sum_{j=1}^{B} \sqrt{\vI-\Sigma_Y}\vX_j + \vY_j) \nonumber \\&\leq \sum_{k=1}^{B} \wass{2}\left(\sum_{j=1}^{k}\vX_j + \sum_{j=k+1}^{B}\sqrt{\vI-\Sigma_Y}\vX_j + \vY_j, \sum_{j=1}^{k-1}\vX_j + \sum_{j=k}^{B}\sqrt{\vI-\Sigma_Y}\vX_j + \vY_j\right) \nonumber\\
    &\leq \sum_{k=1}^{B} \wass{2}\left(\sum_{j=1}^{k}\vX_j , \sum_{j=1}^{k-1}\vX_j + \sqrt{\vI-\Sigma_Y}\vX_k + \vY_k\right) \leq \sum_{k=1}^{B}\frac{5\beta^3\sqrt{d}}{k} \nonumber \\ &\leq 5\beta^3\sqrt{d}(1+\log B)
\end{align}
In the third step, we have used the fact that $\wass{2}(\vZ+\vA,\vZ+\vB) \leq \wass{2}(\vA,\vB)$ whenever $\vZ,\vA$ and $\vZ,\vB$ are independent random variables. Hence, we have proved that $\wass{2}^2\left(\vX,\hvX\right) \leq \nicefrac{25 \beta^6 d}{B} \left(1 + \log B\right)^2$
\end{proof}
\subsection{Proof of Lemma \ref{lem:basic_lem_clt}}
\label{proof:basic-lem-clt}
Define $n_i := \frac{k}{\varsigma_i^2} $. Suppose $\vY,\vZ$ are independent random variables such that $\vZ \sim \mathcal{N}(0,\vI)$ and $\vY$ has the same distribution as any one of $\vY_1,\dots,\vY_B$ used in the statement of Lemma~\ref{lem:wass-clt-new}.  In this section only, for any vector $\vx \in \mathbb{R}^d$, we will let $\vx_i$ denote its component along the $i$-th standard basis vector. Notice that the i-th co-ordinate of $\sqrt{\vI-\frac{\Sigma_\vY}{k}}\vZ + \frac{1}{\sqrt{k}}\vY$ is given by $\sqrt{1-\frac{1}{n_i}}\vZ_{i} + \bar{\vY}_{i}$, where, $\bar{\vY} := \frac{\vY}{\sqrt{k}}$. Clearly, $\|\bar{\vY}\| \leq \frac{\beta}{\sqrt{k}}$, a fact which we will use heavily below. We will also use the observation that $\sum_i \varsigma_i^2 \leq \beta^2$ and $\beta^2 n_j \geq k$.

Let $f(x) := \frac{\tau(\vx)}{\rho(\vx)}$ where $\tau$ is the density of $\sqrt{\vI-\frac{\Sigma_\vY}{k}}\vZ_{1} + \bar{\vY}_{1}$ and $\rho$ is the density of $\vZ_1$. The proof of the following lemma is identical to that of Lemma 4.1 of \cite{zhai2018clt}, with $n$ replaced with the co-ordinate dependent $n_i$, and $\sigma_i$ replaced with $1$. The result relies on the fact that $\varsigma^2_i \leq \nicefrac{1}{5d} < 1$.
 \begin{lemma}
 $$\mathbb{E}\left[f(\vZ_1)^2\right] = \mathbb{E}_{\bar{\vY},\bar{\vY}^{\prime}}\left[\exp\left(\sum_{i=1}^{d} \frac{2n_i^2 \bar{\vY}_i\bar{\vY}_i^{\prime} - n_i \bar{\vY}^2_i - n_i (\bar{\vY}^{\prime}_i)^2 + 1}{2(n_i^2 - 1)} - r(n_i)\right)\right]$$
 Where $\bar{\vY}^{\prime}_i$ is an independent copy of $\bar{\vY}_i$ and $r(n):= \frac{1}{2(n^2-1)} - \frac{1}{2}\log(1+\frac{1}{n^2-1})$
 \end{lemma}
Proceeding similarly, we let $Q_i := \frac{2n_i^2 \bar{\vY}_i\bar{\vY}_i^{\prime} - n_i \bar{\vY}^2_i - n_i (\bar{\vY}^{\prime}_i)^2 + 1}{2(n_i^2 - 1)} - r(n_i)$ and $Q := \sum_{i=1}^{d}Q_i$. Let $f_{(i)}(\vx)$ be the ratio $\frac{\tau_{(i)}(\vx)}{\gamma_{(i)}(\vx)}$, where $\tau_{(i)}(\vx)$ denotes the marginal density under $\tau$ of all co-ordinates other than the $i$-th co-ordinate.  The following lemma is a rewriting of Lemmas 4.4 and 4.5 in \cite{zhai2018clt}

\begin{lemma}
$\sup_i 5\varsigma_i^{2}d<1$, $5\beta^2 \leq 1$. Then, the following bounds hold:

\begin{enumerate}
    \item $|Q_i| \leq \frac{n_i^2|\bar{Y}_i||\bar{Y}_i^{\prime}|}{n_i^2 -1} + \frac{1}{2n_i}$, $|Q| \leq 1$, $|Q-Q_i|\leq 1$
    \item $$\mathbb{E}Q_i = -\frac{1}{2(n_i^2-1)}-r(n_i)$$
    \item 
    $$\mathbb{E}Q_iQ_j \leq \frac{n_i^2\delta_{ij}}{(n_i^2-1)^2} + \frac{n_in_j\mathbb{E}\bar{Y}_i^2\bar{Y}_j^2}{2(n_i^2-1)(n_j^2-1)} + \frac{1}{2(n_i^2-1)(n_j^2-1)}$$ 
    \item $$\mathbb{E}Q_i^2 \leq \frac{2n_i^2 +n_i +1}{2(n_i^2-1)^2}$$
    \item $$\mathbb{E}(Q-Q_i)Q_i \leq  \sup_{j}\frac{k(d-1)+\beta^2n_j}{2k(n_j^2-1)^2}\,.$$
    %\item $\mathbb{E}Q^2 \leq \sup_j \frac{2\varsigma_j^4 dk^2 + d(d+1)(\beta^2\varsigma_j^6+\varsigma_j^8)}{2(k^2-\varsigma_i^4)^2}$
    \item $$\mathbb{E}Q^2\leq \sup_j\frac{\beta^2n_jd + 3n^2_jkd}{2k(n^2_j-1)^2} \leq \sup_j\frac{ 2n^2_jd}{(n^2_j-1)^2}$$
\end{enumerate}

\end{lemma}
\begin{proof}
Items 1 - 4 and item 6 can be shown by essentially the same methods used in Lemmas 4.4 and 4.5 of \cite{zhai2018clt}. For item 5, we have:
\begin{align}
    \mathbb{E}(Q-Q_i)Q_i &= \sum_{j\neq i}Q_jQ_i \leq \sup_l\frac{d-1}{2(n_l^2-1)^2} + \sum_{j\neq i} \frac{n_in_j\mathbb{E}\bar{Y}_i^2\bar{Y}_j^2}{2(n_i^2-1)(n_j^2-1)}  \nonumber \\
    &\leq \sup_l\frac{d-1}{2(n_l^2-1)^2} + \sup_l \frac{n_l}{2(n_l^2-1)^2}  \sum_{j\neq i}n_i\mathbb{E}\bar{Y}_i^2\bar{Y}_j^2 \nonumber \\
    &\leq \sup_l\frac{d-1}{2(n_l^2-1)^2} + \sup_l \frac{n_l}{2(n_l^2-1)^2}  \frac{\beta^2}{k}n_i\mathbb{E}\bar{Y}_i^2 \nonumber \\
    &= \sup_l \frac{n_l\beta^2 + k(d-1)}{2k(n_l^2-1)^2}
\end{align}
In the last step we have used the fact that $\mathbb{E}\bar{Y}_i^2  = \frac{1}{n_i}$.
\end{proof}

The proof of Lemma~\ref{lem:basic_lem_clt} now follows by using the bounds established above along with the proof of Lemma 1.6 in \cite{zhai2018clt} and by noting that $\beta^2n_i^2 \geq k $ for every $i \in [d]$. 

\section{Analysis of SGLD via Entropic CLT}
\label{app-sec:sgld-entropic-clt}
\subsection{Technical Lemmas}
\begin{lemma}[Chain Rule for KL Divergence ]\label{lem:kl_chain_rule}
Suppose $\nu$ is a distribution over some Polish space $\Xi^T$ and $\mu$ be a product distribution over $\Xi^T$ given as $\mu = \otimes_{t=1}^{T}\mu_t$. Let $\nu_t(\cdot|X_{<t})$ denote the conditional distribution of the $t$-th co-ordinate conditioned on the co-ordinates $1,\dots,t-1$ (and the marginal of the first co-ordinate under $\nu$ when $t = 1$) and let $\nu_{< t}$ denote the joint marginal law of the co-ordinates $1,\dots,t-1$ under the measure $\nu$.
$$\KL{\nu}{\mu} = \sum_{t=1}^{T}\mathbb{E}_{X_{<t}\sim \nu_{< t}} \KL{\nu_t(\cdot|X_{<t})}{\mu_t}$$
\end{lemma}
\begin{proof}
Lemma 4.18 in \cite{van2014probability}
\end{proof}
\begin{lemma}[Tensorization of $\mathsf{T}_2$]\label{lem:t2_tensor}
Suppose $P$ is the law of $\mathcal{N}(0,\Sigma)$ for some non-singular $\Sigma$. Then for any probability measure $Q$ over $\mathbb{R}^d$, we have:
$$\wass{2}^2(P,Q) \leq 2\lambda_{\max}(\Sigma)\KL{Q}{P}$$
\end{lemma}
\begin{proof}
Proposition 1.8 in \cite{gozlan2010transport}
\end{proof}
\begin{lemma}[Reverse $\mathsf{T}_2$ for Gaussian Convolutions]
\label{lem:rev_t_2}
Suppose $\vZ \sim \mathcal{N}(0,\sigma^2 \vI)$, $\vA,\vB$ are random variables independent of $\vZ$. Then, $\KL{\vZ+\vA}{\vZ+\vB} \leq \frac{1}{2\sigma^2}\wass{2}^2(\vA,\vB)$
\end{lemma}
% \begin{lemma}[Proposition 7, \cite{givens1984class}]
% Suppose $P_1$ (resp. $P_2$) is the law of the $d$ dimensional Gaussian with mean $m_1$(resp. $m_2$) and non-singluar covariance $\Sigma_1$ (resp. $\Sigma_2$), we have:
% $$\wass{2}^2(P_1,P_2) = \|m_1-m_2\|^2 +\tr(\Sigma_1) +\tr(\Sigma_2) - 2\tr\left[\left(\Sigma_1^{\tfrac{1}{2}}\Sigma_2\Sigma_1^{\tfrac{1}{2}}\right)^{\tfrac{1}{2}}\right] $$
% \end{lemma}
\begin{proof}
Let $\Gamma$ be a $\wass{2}$ optimal coupling between the laws of $\vA$ and $\vB$. Let $f:\mathbb{R}^d \to \mathbb{R}$ be defined by $f(\vx) = \frac{1}{(2\pi\sigma^2)^{d/2}}\exp\left(-\tfrac{\|\vx\|^2}{2\sigma^2}\right)$.
The density of $Z+A$ with respect to the Lebesgue measure is given by $P(\vx) = \mathbb{E} f(\vx-\vA)$ and similarly the density of $Z+B$ is given by $Q(\vx) = \mathbb{E}f(\vx-\vB)$. Therefore, the KL divergence can be written as:
\begin{align}
    \KL{\vZ+\vA}{\vZ+\vB} &= \int f(\vx-\vA) \log\left(\frac{\int f(\vx-\vA)d\Gamma(\vA\times \vB)}{\int f(\vx-\vB)d\Gamma(\vA\times \vB)}\right)d\Gamma(\vA\times \vB)\nonumber d \vx \nonumber\\
    &\leq \int  f(\vx-\vA) \log\left(\frac{ f(\vx-\vA)}{ f(\vx-\vB)}\right)d\Gamma(\vA\times \vB)\nonumber d \vx \nonumber\\
    &= \frac{1}{2\sigma^2}\int f(\vx-\vA) \left[\|\vx-\vB\|^2 - \|\vx-\vA\|^2\right]d\Gamma(\vA\times \vB) d \vx \nonumber \\
    &= \frac{1}{2\sigma^2}\int \|\vA-\vB\|^2 d\Gamma(\vA\times \vB) = \frac{1}{2\sigma^2}\wass{2}^2(\vA,\vB) 
\end{align}
The second step above follows from the log-sum inequality. The third step follows from the definition of $f$.  In the fourth step we have used Fubini's theorem to integrate out $\vx$, noting that $f(\vx-\vA)$ is the density of a Gaussian with covariance $\sigma^2 \vI$ and mean $\vA$. We have finally used the fact that $\Gamma$ is a $\wass{2}$ optimal coupling between $\vA$ and $\vB$.
\end{proof}
Under the (low probability) event that the conditions in Lemma \ref{lem:wass-clt-new} are not satisfied, we use the Wasserstein CLT of \cite{zhai2018clt} to quantify the approximate subgaussianity of the stochastic gradient noise. 
\begin{lemma}[\cite{zhai2018clt}]
\label{lem:zhai-clt}
Let $\vY = \nicefrac{1}{\sqrt{B}} \sum_{i=1}^{B} \vY_i$ where $\vY_1, \ldots, \vY_B \in \bR^d$ are zero-mean i.i.d random vectors with covariance matrix $\Sigma$ such that $\norm{\vY_i} \leq \beta$ holds almost surely. Let $\vZ \sim \cN(0, \Sigma)$ be sampled independently of $\vY_i$. Then,
\begin{align*}
    \wass{2}^2 (\vY, \vZ) \leq \frac{25 \beta^2 d \left(1 + \log(B)\right)^2}{B}
\end{align*}
\end{lemma}
\subsection{Proof of Theorem \ref{thm:sgld-stat-indistinguishable}}
\label{proof:thm-sgld-stat-indistinguishable}
\begin{proof}
Denoting $\vN_k = \vN(\hvx_k, \xi_k)$, we note that the iterates of SGLD and LMC (with the same step-size and initialization) can be written as,
\begin{align*}
    \vx_{k+1} &= \vx_k - \eta \nabla F(\vx_k) + \sqrt{2 \eta} \vz_k \\
    \hvx_{k+1} &= \hvx_k - \eta \nabla F(\hvx_k) + \sqrt{2 \eta} \hvz_k 
\end{align*}
where $\vz_k$ and $\hvz_k$ are defined as,
\begin{align*}
    \vz_k &= \epsilon_k \sim \cN(0, \vI), \\
    \hvz_k &= \sqrt{\nicefrac{\eta}{2}} \vN_k + \heps_k, \ \ \heps_k \sim \cN(0, \vI) \\
\end{align*}
Defining the filtration $\cF_k = \sigma(\hvx_0, \dots, \hvx_k, \hvz_0, \dots, \hvz_{k-1})$, we observe that SGLD and LMC admit the same random function representation, i.e., there exists a measurable function $H_K$ such that,
\begin{align*}
    (\hvx_1, \dots, \hvx_{K+1}) &= H_K(\hvx_0, \hvz_0, \hvz_1, \dots, \hvz_K) \\
    (\vx_1, \dots, \vx_{K+1}) &= H_K(\vx_0, \vz_0, \vz_1, \dots, \vz_K)
\end{align*}
Since $\Law{\hvx_0} = \Law{\vx_0}$, we use the data processing inequality and Lemma \ref{lem:kl_chain_rule} to conclude the following.
\begin{align*}
    \KL{\hvx_1, \dots, \hvx_{K+1}}{\vx_1, \dots, \vx_{K+1}} &= \sum_{k=0}^{K} \bE\left[\KL{\Law{\hvz_k \ | \ \cF_k} }{\vz_k}\right]
\end{align*}
We shall now control each term in the above summation. To this end, let $\vX, \vX_1, \vX_2, \vZ, \vZ_1, \vZ_2, \vW \ \iidsim \ \cN(0, \vI)$. It follows that,
\begin{align*}
    \KL{\Law{\hvz_k \ | \ \cF_k} }{\vz_k} &= \KL{\vX_1 + \sqrt{\nicefrac{\eta}{2}}\vN_k}{\vZ_1\biggr|\cF_k} \\
    &= \KL{\sqrt{\nicefrac{1}{2}}\vX_2 + \sqrt{\nicefrac{1}{2}}\vX + \sqrt{\nicefrac{\eta}{2}}\vN_k}{\sqrt{\nicefrac{1}{2}}\vZ_2 + \sqrt{\nicefrac{1}{2}}\vZ\biggr|\cF_k} \\
    &\leq \frac{1}{2}\wass{2}^2\left(\vX + \sqrt{\eta}\vN_k, \vZ\biggr|\cF_k\right) 
\end{align*}
Where the last inequality follows from Lemma \ref{lem:rev_t_2}. Now, let $\vY = \sqrt{\eta} \vN_k$ and define $\Sigma_{\vY} = \bE\left[ \vY \vY^T \biggr|\cF_k \right]$. It follows that,
\begin{align}
\label{eqn:sgld-clt-kl-decomposition}
    \KL{\Law{\hvz_k \ | \ \cF_k} }{\vz_k} &\leq \frac{1}{2}\wass{2}^2\left(\vX + \sqrt{\eta}\vN_k, \vZ\biggr|\cF_k\right) \nonumber \\
    &\leq \underbrace{\wass{2}^2\left(\vX + \vY, \sqrt{\vI + \Sigma_{\vY}}\vZ\biggr|\cF_k\right)}_{\textrm{Wasserstein CLT Term}} + \underbrace{\wass{2}^2\left(\sqrt{\vI + \Sigma_{\vY}}\vZ, \vW\biggr|\cF_k\right)}_{\textrm{Covariance Mismatch Term}}
\end{align}
We first bound the covariance mismatch term via direct computation of the Wasserstein distance between two zero-mean Gaussians. It follows that, 
\begin{align*}
    \wass{2}^2\left(\sqrt{\vI + \Sigma_{\vY}}\vZ, \vW \biggr|\cF_k\right) = \Tr\left(2 \vI + \Sigma_{\vY} - 2 \sqrt{\vI + \Sigma_{\vY}} \right)
\end{align*}
Let $\lambda_1, \dots, \lambda_d \geq 0$ denote the eigenvalues of $\Sigma_{\vY}$. Then,
\begin{align*}
    \wass{2}^2\left(\sqrt{\vI + \Sigma_{\vY}}\vZ, \vW \biggr|\cF_k\right) &= \sum_{i=1}^d \left(2 + \lambda_i - 2 \sqrt{1 + \lambda_i}\right) \\
    &\leq \sum_{i=1}^{d} \nicefrac{\lambda^2_i}{4} \leq \nicefrac{1}{4}\left( \sum_{i=1}^{d} \lambda_i \right)^2 = \nicefrac{\Tr\left( \Sigma_{\vY} \right)^2}{4}
\end{align*}
Define $u_k = M \norm{\hvx_k} + G$. From the definition of $\vY$ and the \ref{eqn:defn-grad-growth} condition, it follows that
\begin{align*}
    \Tr\left( \Sigma_{\vY} \right) &= \nicefrac{\eta}{B} \bE\left[ \norm{\nabla f(\hvx_k, \xi_1) - \nabla F(\vx_k)}^2 \ | \ \hvx_k\right] \leq \frac{\eta u^2_k}{B} 
\end{align*}
Hence, the covariance mismatch term is bounded as follows
\begin{equation}
\label{eqn:sgld-cov-mismatch-bound}
    \wass{2}^2\left(\sqrt{\vI + \Sigma_{\vY}}\vZ, \vW\biggr|\cF_k\right) \leq \frac{\eta^2 u^4_k }{4 B^2} 
\end{equation}
For controlling the Wasserstein CLT term, we observe that $\vY = \nicefrac{1}{\sqrt{B}}\sum_{i=1}^{B} \vY_i$ where $\vY_i = \sqrt{\nicefrac{\eta}{B}} \vN(\hvx_k, \xi_{k, i})$. We note that $\norm{\vY_i} \leq u_k \sqrt{\nicefrac{\eta}{B}} $. Our proof proceeds by considering two cases
\paragraph{Case 1 : $u^2_k > \nicefrac{B}{5 \eta d}$} In this case, the conditions required to apply Lemma \ref{lem:wass-clt-new} are not satisfied, and hence, we control this term using the Wasserstein CLT of \cite{zhai2018clt}. 
\begin{align*}
    \wass{2}^2\left(\vX + \vY, \sqrt{\vI + \Sigma_{\vY}}\vZ\biggr|\cF_k\right) &= \wass{2}^{2}\left(\vX + \vY, \vZ_1 + \sqrt{\Sigma_{\vY}}\vZ_2\biggr|\cF_k\right) \\
    &= \wass{2}^{2}\left(\vY, \sqrt{\Sigma_{\vY}}\vZ_2\biggr|\cF_k\right) \leq \frac{25 \eta d  (1 + \log(B))^2 u^2_k}{B^2} 
\end{align*}
where the last inequality follows from Lemma \ref{lem:zhai-clt}
\paragraph{Case 2 : $u^2_k \leq \nicefrac{B}{5 \eta d}$} We first note that, $\Tr\left(\Sigma_{\vY}\right) \leq \nicefrac{\eta u^2_k}{B} \leq \nicefrac{1}{5d}$. We then define $\vG_i = (\vI + \Sigma_{\vY})^{-\nicefrac{1}{2}} \vY_i$ and $\vG = \nicefrac{1}{\sqrt{B}}\sum_{i=1}^{B} \vG_i$. Moreover, let $\Gamma_{\vG} = \bE\left[ \vG_i \vG^T_i \right]$. We note that, since $(\vI + \Sigma_{\vY})^{-\nicefrac{1}{2}} \preceq \vI$, $\norm{\vG_i}^2 \leq \norm{\vY_i}^2 \leq \nicefrac{1}{5}$. Furthermore,
\begin{align*}
\Gamma_{\vG} = \left(\vI + \Sigma_{\vY}\right)^{-1} \Sigma_{\vY} = \vI - \left(\vI + \Sigma_{\vY}\right)^{-1}
\end{align*}
Moreover, let $\lambda_1, \dots, \lambda_d$ denote the eigenvalues of $\Sigma_{\vY}$ and $\mu_1, \dots, \mu_d$ denote the eigenvalues of $\Gamma_{\vG}$. From the above identity, we note that, $\mu_i = \nicefrac{\lambda_i}{1 + \lambda_i}$ hence $\norm{\Gamma_{\vG}} \leq \norm{\Sigma_{\vY}} \leq \nicefrac{1}{5d}$ and $\Tr\left(\Gamma_{\vG}\right) \leq \Tr\left(\Sigma_{\vY}\right)$. Hence, the conditions required to apply Lemma \ref{lem:wass-clt-new} are satisfied. It follows that, 
\begin{align*}
    \wass{2}^2\left(\vX + \vY, \sqrt{\vI + \Sigma_{\vY}}\vZ\biggr|\cF_k\right) &\leq \norm{\vI + \Sigma_{\vY}} \wass{2}^2\left( \left(\vI + \Sigma_{\vY}\right)^{-\nicefrac{1}{2}} \vX + \vG, \vZ  \biggr|\cF_k\right) \\
    &\leq \frac{6}{5}\wass{2}^2 \left( \sqrt{\vI - \Gamma_{\vG}} \vX + \vG, \vZ \biggr|\cF_k \right) 
\end{align*}
where the first inequality follows from the fact that $\norm{\vI + \Sigma_{\vY}} = 1 + \norm{\Sigma_{\vY}} \leq 1 + \nicefrac{1}{5d} \leq \nicefrac{6}{5}$. Now, using the sharper Wasserstein CLT of Lemma \ref{lem:wass-clt-new}, we conclude that, 
\begin{align*}
    \wass{2}^2\left(\vX + \vY, \sqrt{\vI + \Sigma_{\vY}}\vZ\biggr|\cF_k\right) &\leq \frac{6}{5}\wass{2}^2 \left( \sqrt{\vI - \Gamma_{\vG}} \vX + \vG, \vZ \biggr|\cF_k \right) \\
    &\leq \frac{30 \eta^3 d \left(1 + \log(B) \right)^2 u^6_k}{B^4}
\end{align*}
From Case 1 and Case 2, it follows that,
\begin{align}
\label{eqn:sgld-wass-clt-bound}
\wass{2}^2\left(\vX + \vY, \sqrt{\vI + \Sigma_{\vY}}\vZ \biggr|\cF_k \right) &\leq \frac{25 \eta d  (1 + \log(B))^2 u^2_k}{B^2} \bEvent{u^2_k > \nicefrac{B}{5\eta} d}  \nonumber \\ 
&+ \frac{30 \eta^3 d \left(1 + \log(B) \right)^2 u^6_k}{B^4} \bEvent{u^2_k \leq \nicefrac{B}{5\eta} d} \nonumber \\
&\leq \frac{25 \eta d  (1 + \log(B))^2 u^2_k}{B^2} \bEvent{u^2_k > \nicefrac{B}{5\eta} d} + \frac{30 \eta^3 d \left(1 + \log(B) \right)^2 u^6_k}{B^4} 
\end{align}
Thus, from \eqref{eqn:sgld-clt-kl-decomposition}, \eqref{eqn:sgld-cov-mismatch-bound} and \eqref{eqn:sgld-wass-clt-bound}, it follows that,
\begin{align*}
    \KL{\Law{\hvz_k \ | \ \cF_k} }{\vz_k} &\leq \frac{\eta^2 u^4_k}{4 B^2} + \frac{30 \eta^3 d \left(1 + \log(B) \right)^2 u^6_k}{B^4} \\
    &+ \frac{25 \eta d  (1 + \log(B))^2 u^2_k}{B^2} \bEvent{u^2_k > \nicefrac{B}{5\eta} d}
\end{align*}
Recalling that $u_k = M \norm{\hvx_k} + G$, we use the \ref{eqn:defn-grad-growth} condition for $p = 6$ to conclude that,
\begin{align*}
    \bE\left[u^4_k\right] &\leq 8\left(M^{4} \bE\left[\norm{\hvx}^{4}\right] + G^{4}\right) \leq 8\left(M^{4} d^{2} + G^{4}\right) \\
    \bE\left[u^6_k\right] &\leq 32\left(M^{6} \bE\left[\norm{\hvx}^{6}\right] + G^{6}\right) \leq 32\left(M^{6} d^{3} + G^{6}\right) \\
    \bE\left[u^2_k \bEvent{u^2_k > \nicefrac{B}{5\eta} d} \right] &\leq \frac{75 \eta^2 d^2}{2B^2} \bE\left[u^6_k\right] \leq \frac{1200 \eta^2 d^2}{B^2}\left(M^6 d^3 + G^6\right)
\end{align*}
where the last inequality uses Lemma \ref{lem:tail-markov-control}. From the above inequalities, we obtain
\begin{align*}
    \bE\left[\KL{\Law{\hvz_k \ | \ \cF_k} }{\vz_k}\right] &\leq \frac{2 \eta^2}{B^2} \left(M^{4} d^{2} + G^{4}\right) + \frac{960 \eta^3}{B^4} \left(M^{6} d^{4} + G^{6} d\right)\left(1 + \log B \right)^2 \\
    &+ \frac{2000 \eta^3}{B^4}\left(M^{6} d^{3 + 3} + G^{6} d^3 \right)\left(1 + \log B \right)^2 \\
    &\leq \frac{2 \eta^2}{B^2} \left(M^{4} d^{2} + G^{4}\right) + \frac{3000 \eta^3}{B^4}\left(M^{6} d^{3 + 3} + G^{6} d^3 \right)\left(1 + \log B \right)^2
\end{align*}
Thus, we finally obtain the following statistical indistinguishability guarantee
\begin{align*}
    \KL{\hvx_{1:K}}{\vx_{1:K}} &\leq \frac{2 \eta^2 K}{B^2} \left(M^{4} d^{2} + G^{4}\right) + \frac{3000 \eta^3 K}{B^4}\left(M^{6} d^{6} + G^{6} d^3 \right)\left(1 + \log B \right)^2
\end{align*}
\end{proof}

\subsection{Convergence of Non-Smooth SGLD under \ref{eqn:defn-latala-oleskiewicz}}
In this section, we use $\Renyi{\mu}{\nu}{q}$ to denote the R\'enyi divergence of order $q$ between two measures $\mu$ and $\nu$.
\label{proof:cor-sgld-unstable-lo}
\begin{corollary}[Convergence of SGLD under LO]
\label{cor:sgld-unstable-lo}
Let the \ref{eqn:defn-holder-smoothness}, \ref{eqn:defn-grad-growth}, and \ref{eqn:defn-moment-growth} be satisfied with $p = 6$. Furthermore, assume the target $\pistar$ satisfies \ref{eqn:defn-latala-oleskiewicz} for some $\alpha \in [1, 2]$ and define $\beta = \nicefrac{2}{\alpha} - 1$. Then, for  $\epsilon \leq \nicefrac{1}{\mathsf{poly}(d)}$, the last iterate of SGLD, under appropriate Gaussian initialization, requires $N$ stochastic gradient oracle calls to ensure $\TV(\Law{\hvx_K}, \pistar) \leq \epsilon$, where
\begin{align*}
     N = \Thetatilde\left(\frac{d^{\max \left\{ 1 + \beta\left(1 + \nicefrac{1}{s}\right), \nicefrac{3}{2}\left(1+\beta\right) + \nicefrac{\beta}{2s} \right\}}}{\epsilon^{\nicefrac{2}{s}}}\right)
\end{align*}
\end{corollary}
\begin{proof}
Our proof relies on the recent result of \cite{sinho-lmc-poincare-lsi} where the authors establish an unstable convergence guarantee for LMC under H\"{o}lder smoothness and Lata\l{}a-Oleskiewicz inequality. In particular, Theorem 7 of \cite{sinho-lmc-poincare-lsi} implies that, under appropriate Gaussian initialization such that $\Delta_0 = \Renyi{\Law{\vx_0}}{\pistar}{3} = O(d)$, it suffices to set $\eta$ and $K$ as follows to achieve $\TV\left(\Law{\vx_{K+1}}, \pistar\right) \leq \epsilon$
\begin{align*}
    \eta &= \Thetatilde\left( \frac{\epsilon^{\nicefrac{2}{s}}}{d \Delta^{\nicefrac{\beta}{s}}_0} \right) \\
    K &= \Thetatilde\left(\frac{d \Delta^{\beta(1 + \nicefrac{1}{s})}_0}{\epsilon^{\nicefrac{2}{s}}}\right) \\
    T &= \eta K = \Thetatilde\left(\Delta^{\beta}_0\right) 
\end{align*}
where $\beta = \nicefrac{2}{\alpha} - 1$. Under this choice of $\eta$ and $T$, Theorem \ref{thm:sgld-stat-indistinguishable} suggests that $\KL{\hvx_{1:K}}{\vx_{1:K}} \leq O(\nicefrac{\eta T d^2}{B^2})$. Thus, to ensure $\TV(\Law{\hvx_{K+1}}, \Law{\vx_{K+1}}) \leq \nicefrac{\epsilon}{2}$, it suffices to set $B \geq \Otilde\left(\max \left\{ 1, d^{\nicefrac{1}{2}} \Delta^{\nicefrac{\beta}{2}\left(1-\nicefrac{1}{s}\right)}_0 \right\}\right)$. Hence, $\TV(\Law{\hvx_{K+1}}, \pistar) \leq \epsilon$ by subadditivity of Total Variation. The required stochastic gradient complexity is,
\begin{align*}
    N &= KB = \Thetatilde\left(\max\left\{\frac{d \Delta^{\beta\left(1 + \nicefrac{1}{s}\right)}_0}{\epsilon^{\nicefrac{2}{s}}}, \frac{d^{\nicefrac{3}{2}} \Delta^{\nicefrac{3\beta}{2} + \nicefrac{\beta}{2s}}_0}{\epsilon^{\nicefrac{2}{s}}}\right\}\right) \\
    &= \Thetatilde\left(\frac{d^{\max \left\{ 1 + \beta\left(1 + \nicefrac{1}{s}\right), \nicefrac{3}{2}\left(1+\beta\right) + \nicefrac{\beta}{2s} \right\}}}{\epsilon^{\nicefrac{2}{s}}}\right)
\end{align*}
where we use the fact that $\Delta_0 = O(d)$.
\end{proof}

\paragraph{Rates Under LSI and Smoothness} We recall that the LO inequality of order $\alpha = 2$ is equivalent to the Logarithmic Sobolev Inequality. Moreover, as a consequence of the fact $\alpha - 1 \leq s \leq 1$, we note that LSI and Holder continuity together imply smoothness. To this end, we observe that Corollary \ref{cor:sgld-unstable-lo} implies an oracle complexity of $\Thetatilde(\nicefrac{d^{  \nicefrac{3}{2}}}{\epsilon^2})$ to ensure $\epsilon$-convergence in TV for the last iterate. We note that this matches the guarantee implied by Theorem \ref{thm:sgld-lsi-stable} via Pinsker's inequality. However, we emphasize that Theorem \ref{thm:sgld-lsi-stable} is a stronger result as it establishes convergence in KL divergence, which automatically implies convergence in TV (due to Pinsker's inequality) and $\wass{2}$ (due to the Otto-Villani theorem). Furthermore, unlike Corollary \ref{cor:sgld-unstable-lo}, Theorem \ref{thm:sgld-lsi-stable} is a stable convergence guarantee.

\paragraph{Last-Iterate Convergence Under Poincare Inequality} From Assumption \ref{as:target-lo}, we recall that the LO inequality of order $\alpha = 1$ is equivalent to the Poincare Inequality. In this setting, we observe that that Corollary \ref{cor:sgld-unstable-lo} implies an oracle complexity of $\Thetatilde\left(\frac{d^{\max\left\{2 + \nicefrac{1}{s}, 3 + \nicefrac{1}{2s}\right\}}}{\epsilon^{\nicefrac{2}{s}}}\right)$ to ensure $\epsilon$-convergence in TV for the last iterate. To the best of our knowledge, this is the first known result for the last-iterate convergence of SGLD under the Poincare Inequality, as well as the first analysis of SGLD under Poincare Inequality that does not assume smoothness. For the smooth case, i.e., when $s=1$, we note that the implied oracle complexity is $\Thetatilde\left(\frac{d^{ 3.5}}{\epsilon^{2}}\right)$ for last-iterate convergence. When compared to the $\Otilde\left(\frac{d^{2.5}}{\epsilon^{4}}\right)$ guarantee of Theorem \ref{thm:sgld-fd-and-pi} for \emph{average-iterate convergence in TV}, we note that the rates implied by Corollary \ref{cor:sgld-unstable-lo} exhibit an improved $\epsilon$-dependence at the cost of a worse $d$ dependence. However, we highlight that, unlike Corollary \ref{cor:sgld-unstable-lo}, Theorem \ref{thm:sgld-fd-and-pi} is a stable convergence guarantee

%\subsection{Oracle Complexity of Corollary \ref{cor:sgld-unstable-lo}}
%\label{oracle:sgld-unstable-lo}

\section{Analysis of Covariance Corrected SGLD}
\subsection{Technical Lemmas}
\begin{lemma}\label{lem:kl_triangle}
Let $\vZ, \vZ_2 \ \iidsim \ \mathcal{N}(0,\vI)$ and let $\vN$ be a zero-mean random vector independent of $\vZ, \vZ_2$ such that $\Sigma = \mathsf{Cov}\left[\vN\right] \prec \nicefrac{\vI}{2}$. Let $\hSigma$ be an arbitrary PSD matrix with $\hSigma \prec \nicefrac{\vI}{2}$ and let $\bar{\vZ} = (\vI - \nicefrac{\hSigma}{2})\vZ + \vN$. Then,
\begin{align*}
\KL{\bar{\vZ}}{\vZ} &\leq 2\wass{2}^2\left(\frac{\vZ_2}{\sqrt{2}} ,\sqrt{\tfrac{1}{2}\vI-\Sigma}\vZ_2 + \vN\right) \\
&\quad+8\lambda_{\max}(\tfrac{\vI}{2}-\Sigma)\KL{\sqrt{\tfrac{1}{2}\vI-\hat{\Sigma}}\vZ_2 }{\sqrt{\tfrac{1}{2}\vI-\Sigma}\vZ_2} + 2 \Tr\left(\hSigma\right)^3
\end{align*}
\end{lemma}
\begin{proof}
Since $\hSigma \prec \nicefrac{\vI}{2}$, $(\vI - \nicefrac{\hSigma}{2})^2 \succ \nicefrac{\vI}{2}$, and thus $\vB = \sqrt{(\vI - \nicefrac{\hSigma}{2})^2 - \nicefrac{\vI}{2}}$ is a well defined symmetric positive definite matrix. Furthermore, if $\vZ_1,\vZ_2$ are i.i.d. isotropic Gaussians over $\mathbb{R}^d$, we can write: $\vZ \disteq \frac{1}{\sqrt{2}}\vZ_1 + \frac{1}{\sqrt{2}}\vZ_2 $ and $\bar{\vZ} \disteq \frac{1}{\sqrt{2}}\vZ_1+ \vB\vZ_2 + \vN$, where $\vN$ is independent of both $\vZ_1$ and $\vZ_2$. Applying Lemma~\ref{lem:rev_t_2} with $\sigma = \frac{1}{\sqrt{2}}$ and $\vZ$ replaced with $\vZ_1$, we have:
\begin{align}
\label{eqn:kl-triangle-eqn1}
    &\KL{\bar{\vZ}}{\vZ} \leq \wass{2}^2\left(\frac{\vZ_2}{\sqrt{2}} ,\vB\vZ_2 + \vN\right)  \nonumber \\
    &\leq 2\wass{2}^2\left(\frac{\vZ_2}{\sqrt{2}} ,\sqrt{\tfrac{1}{2}\vI-\Sigma}\vZ_2 + \vN\right)+2\wass{2}^2\left(\sqrt{\tfrac{1}{2}\vI-\Sigma}\vZ_2 + \vN ,\vB\vZ_2 + \vN\right)  \nonumber\\
    &\leq 2\wass{2}^2\left(\frac{\vZ_2}{\sqrt{2}} ,\sqrt{\tfrac{1}{2}\vI-\Sigma}\vZ_2 + \vN\right)+2\wass{2}^2\left(\sqrt{\tfrac{1}{2}\vI-\Sigma}\vZ_2 ,\vB\vZ_2\right)  \nonumber\\
    &\leq 2\wass{2}^2\left(\frac{\vZ_2}{\sqrt{2}} ,\sqrt{\tfrac{1}{2}\vI-\Sigma}\vZ_2 + \vN\right)+4\wass{2}^{2}\left(\vB \vZ_2, \sqrt{\tfrac{1}{2}\vI-\hat{\Sigma}}\vZ_2\right)+4\wass{2}^2\left(\sqrt{\tfrac{1}{2}\vI-\Sigma}\vZ_2 ,\sqrt{\tfrac{1}{2}\vI-\hat{\Sigma}}\vZ_2 \right)\nonumber \\
    &\leq 2\wass{2}^2\left(\frac{\vZ_2}{\sqrt{2}} ,\sqrt{\tfrac{1}{2}\vI-\Sigma}\vZ_2 + \vN\right) +8\lambda_{\max}(\tfrac{\vI}{2}-\Sigma)\KL{\sqrt{\tfrac{1}{2}\vI-\hat{\Sigma}}\vZ_2 }{\sqrt{\tfrac{1}{2}\vI-\Sigma}\vZ_2} \nonumber \\
    &+ 4\wass{2}^{2}\left(\vB \vZ_2, \sqrt{\tfrac{1}{2}\vI-\hat{\Sigma}}\vZ_2\right),
\end{align}
where the second step uses the triangle inequality for $\wass{2}$ along with the inequality $(a+b)^2 \leq 2a^2 + 2b^2$, the third step follows from the definition of $\wass{2}$ by restricting to the $\wass{2}$-optimal coupling between $\sqrt{\tfrac{1}{2}\vI-\Sigma}\vZ_2$ and $\vB \vZ_2$, and the last step follows from an application of Lemma~\ref{lem:t2_tensor}. The remainder of our proof controls the term $\wass{2}^{2}\left(\vB \vZ_2, \sqrt{\tfrac{1}{2}\vI-\hat{\Sigma}}\vZ_2\right)$. To this end, we define the matrices $\vA_1, \vA_2$ and $\vA_3$ as follows,
\begin{align*}
    \vA_1 &= \tfrac{1}{2}\vI - \hSigma \\
    \vA_2 &= \vB^2 = \left(\vI - \tfrac{1}{2}\hSigma\right)^2 - \tfrac{1}{2}\vI \\
    \vA_3 &= \vA_1 + \vA_2 - 2 \left(\vA^{\nicefrac{1}{2}}_{1} \vA_2 \vA^{\nicefrac{1}{2}}_{1}\right)^{\nicefrac{1}{2}}
\end{align*}
Since $0 \preceq \hSigma \prec \nicefrac{\vI}{2}$, $\vA_1, \vA_2, \vA_3, \hSigma$ are simultaneously diagonalizable PSD matrices. To this end, let $\lambda_1, \dots, \lambda_d$ denote the eigenvalues of $\hSigma$. Clearly, $0 \leq \lambda_i < \nicefrac{1}{2}$ for all $i \in [d]$. Moreover, the corresponding eigenvalue $\mu_1, \dots, \mu_d$ of $\vA_3$ is given by 
\begin{align*}
    \mu_i = \left(1 - \nicefrac{\lambda_i}{2}\right)^2 - \lambda_i - 2 \sqrt{\left(\nicefrac{1}{2} - \lambda_i\right)\left[\left(1-\nicefrac{\lambda_i}{2}\right)^2 - \nicefrac{1}{2}\right]} \leq \nicefrac{\lambda^3_i}{2}
\end{align*}
where the last inequality follows from a Taylor expansion. By direct computation of the Wasserstein distance between zero-mean Gaussians, we have,
\begin{align*}
    \wass{2}^{2}\left(\vB \vZ_2, \sqrt{\tfrac{1}{2}\vI-\hat{\Sigma}}\vZ_2\right) &= \Tr(\vA_3) = \sum_{i=1}^{d} \mu_i \leq \sum_{i=1}^{d} \nicefrac{\lambda^3_i}{2} \leq \tfrac{1}{2}\left(\sum_{i=1}^{d} \lambda_i\right)^3 = \tfrac{1}{2}\Tr(\hSigma)^3
\end{align*}
Substituting the above inequality into \eqref{eqn:kl-triangle-eqn1}, we obtain,
\begin{align*}
    \KL{\bar{\vZ}}{\vZ} &\leq 2\wass{2}^2\left(\frac{\vZ_2}{\sqrt{2}} ,\sqrt{\tfrac{1}{2}\vI-\Sigma}\vZ_2 + \vN\right)\\
    &+8\lambda_{\max}(\tfrac{\vI}{2}-\Sigma)\KL{\sqrt{\tfrac{1}{2}\vI-\hat{\Sigma}}\vZ_2 }{\sqrt{\tfrac{1}{2}\vI-\Sigma}\vZ_2} + 2 \Tr(\hSigma)^3
\end{align*}
\end{proof}
\begin{lemma}
\label{lem:kl_mismatch}
For any PSD matrices $\Sigma, \hSigma \prec \nicefrac{\vI}{2}$
\begin{align}
    \KL{\sqrt{\tfrac{1}{2}\vI-\hat{\Sigma}}\vZ_2 }{\sqrt{\tfrac{1}{2}\vI-\Sigma}\vZ_2} &= \tr\left((\vI-2\Sigma)^{-1}\left(\Sigma - \hat{\Sigma}\right)\right)\nonumber \\&\quad+ \sum_{k=1}^{\infty}2^{k-1}\frac{\left[\tr(\hat{\Sigma}^k) - \tr(\Sigma^k)\right]}{k} 
\end{align}
\end{lemma}
\begin{proof}
From standard formula for KL divergence between two multi-variate Gaussians, it follows that:
\begin{equation*}
  \KL{\sqrt{\tfrac{1}{2}\vI-\hat{\Sigma}}\vZ_2 }{\sqrt{\tfrac{1}{2}\vI-\Sigma}\vZ_2} = \frac{1}{2}\left[\log\tfrac{\det(\tfrac{\vI}{2}-\Sigma)}{\det(\tfrac{\vI}{2}-\hat{\Sigma})}  - d + \tr\left((\tfrac{\vI}{2}-\Sigma)^{-1}(\tfrac{\vI}{2}-\hat{\Sigma})\right)\right]
\end{equation*}
We first note that by basic algebraic manipulation, 

$$ \tr\left((\tfrac{\vI}{2}-\Sigma)^{-1}(\tfrac{\vI}{2}-\hat{\Sigma})\right) - d = 2\tr\left((\vI-2\Sigma)^{-1}\left(\Sigma - \hat{\Sigma}\right)\right) $$
Now, consider $\log(\det(A))$ for a $d\times d$ PSD matrix $ 0\prec A \prec \vI$. Taking $\lambda_1,\dots,\lambda_d$ to be the eigenvalues of $A$, we have \begin{align}
    \log(\det(A)) = \sum_{i=1}^{d}\log(\lambda_i) &= \sum_{i=1}^{d} \log(1+\lambda_i-1) = \sum_{i=1}^{d}\sum_{k=1}^{\infty}(-1)^{k-1}\frac{(\lambda_i-1)^k}{k} \nonumber \\
    &= -\sum_{k=1}^{\infty}\frac{\tr\left[(\vI-A)^k\right]}{k} \label{eq:log_det}
\end{align} 
Therefore, $\log\tfrac{\det\left(\tfrac{\vI}{2}-\Sigma\right)}{\det\left(\tfrac{\vI}{2}-\hat{\Sigma}\right)} = \log\tfrac{\det(\vI-2\Sigma)}{\det(\vI-2\hat{\Sigma})} = \sum_{k=1}^{\infty}\frac{2^{k}}{k}\left[\tr(\hat{\Sigma}^k) - \tr(\Sigma^k)\right]$, which proves the result. 
\end{proof}
\begin{lemma}[Covariance Estimation Guarantees]
\label{lem:cov-estimation-guarantee}
Consider any $B^{\prime} \geq 1$. Let $\vy_1, \dots, \vy_{B^{\prime}}$ and $\tilvy_1, \dots, \tilvy_{B^{\prime}}$ be i.i.d samples from some probability measure $P$ supported on $\bR^d$, with covariance matrix $\Tilde{\Sigma}$. Furthermore, assume $\norm{\vy - \bE\left[\vy\right]} \leq M$ holds almost surely for any $\vy \sim P$. Define $\Sigma$ as $\nicefrac{\Tilde{\Sigma}}{B}$ and $\hSigma$ as follows,
\begin{align*}
    \hSigma = \frac{1}{2 B B^{\prime}} \sum_{j=1}^{B^{\prime}} \left(\vy_j - \tilvy_j\right)\left(\vy_j - \tilvy_j\right)^T
\end{align*}
Then, the following holds,
\begin{enumerate}
    \item $\bE\left[\hSigma\right] = \Sigma$
    \item $\bE\left[\Tr\left(\hSigma^2\right)\right] - \Tr\left(\Sigma^2\right) \leq \frac{4M^4}{B^{2} B^{\prime}}$
    \item $\Tr\left(\hSigma^{k}\right) \leq \left(\frac{2M^2}{B}\right)^k$
\end{enumerate}
\end{lemma}
\begin{proof}
The first property can be verified directly by taking expectations and using the fact that $\vy_{1:B^{\prime}}, \tilvy_{1:B^{\prime}}$ are independent. To prove the third property, we note that $\norm{\vy - \bE\left[\vy\right]} \leq M$ implies that, $\Tr\left(\Sigma\right), \Tr\left(\hSigma\right) \leq \nicefrac{2M^2}{B}$. Hence, $\Tr\left(\hSigma^{k}\right) \leq \Tr\left(\hSigma\right)^k \leq \left(\frac{2M^2}{B}\right)^k$. Finally, to prove the second property, define $\vA_j = \left(\vy_j - \tilvy_j\right)\left(\vy_j - \tilvy_j\right)^T$. Then,
\begin{align*}
    \hSigma^2 = \frac{1}{4B^2 (B^{\prime})^2} \sum_{j=1}^{B^{\prime}} \vA^2_j + \frac{1}{4B^2 (B^{\prime})^2} \sum_{k, l \in [B^{\prime}], k \neq l}  \vA_k \vA_l
\end{align*}
Since  $\vy_{1:B^{\prime}}, \tilvy_{1:B^{\prime}}$ are sampled i.i.d, it follows that, $\bE\left[\vA_k \vA_l\right] = 4 \Sigma^2_i$ whenever $k \neq l$. Hence, we can conclude that,
\begin{align*}
    \bE\left[\hSigma^2\right] - \Sigma^2 = \frac{\bE\left[\vA^2_1\right]}{B^2 B^{\prime}} - \frac{\Sigma^2}{B^2 B^{\prime}}
\end{align*}
The final result follows by taking the trace on both sides and observing that $\Tr\left(\vA^2_1\right) \leq 16 M^4$
\end{proof}
\subsection{Proof of Theorem \ref{thm:cc-sgld-stat-indistinguishable}}
\label{proof:thm-cc-sgld-stat-indistinguishable}
\begin{proof}
Define $u_k = M \norm{\hvx_k} + G$. Denoting $\vN_k = \vN(\hvx_k, \xi_k)$ and $\Sigma_k = \bE \left[ \vN_k \vN^T_k \ | \ \hvx_k \right]$, we note that the iterates of CC-SGLD and LMC (with the same step-size and initialization) can be written as,
\begin{align*}
    \vx_{k+1} &= \vx_k - \eta \nabla F(\vx_k) + \sqrt{2 \eta} \vz_k \\
    \hvx_{k+1} &= \hvx_k - \eta \nabla F(\hvx_k) + \sqrt{2 \eta} \hvz_k 
\end{align*}
where $\vz_k$ and $\hvz_k$ are defined as,
\begin{align*}
    \vz_k &= \epsilon_k \sim \cN(0, \vI), \\
    \hvz_k &= \sqrt{\nicefrac{\eta}{2}} \vN_k + \heps_k \mathbb{I}_{\{u^2_k > \nicefrac{B}{5 \eta d}\}} + \left(\vI - \tfrac{\eta}{4}\hSigma_k\right) \heps_k \mathbb{I}_{\{u^2_k \leq \nicefrac{B}{5 \eta d}\}} , \ \ \heps_k \sim \cN(0, \vI) 
\end{align*}
Defining the filtration $\cF_k = \sigma(\hvx_0, \dots, \hvx_k, \hvz_0, \dots, \hvz_{k-1})$, we observe that CC-SGLD and LMC admit the same random function representation, i.e., there exists a measurable function $H_K$ such that,
\begin{align*}
    (\hvx_1, \dots, \hvx_{K+1}) &= H_K(\hvx_0, \hvz_0, \hvz_1, \dots, \hvz_K) \\
    (\vx_1, \dots, \vx_{K+1}) &= H_K(\vx_0, \vz_0, \vz_1, \dots, \vz_K)
\end{align*}
Since $\Law{\vx_0} = \Law{\hvx_0}$, we use the data processing inequality and Lemma \ref{lem:kl_chain_rule} to obtain,
\begin{align*}
    \KL{\hvx_1, \dots, \hvx_{K+1}}{\vx_1, \dots, \vx_{K+1}} &= \sum_{k=0}^{K} \bE\left[\KL{\Law{\hvz_k \ | \ \cF_k} }{\vz_k}\right]
\end{align*}
Define the filtration $\cG_k = \cF_k \vee \sigma(\xi^{(1)}_{k, j}, \xi^{(2)}_{k, j}  \ | \ i \in \{ 1, 2\}, j \in [B])$. By Jensen's inequality, it follows that, 
\begin{align*}
    \bE\left[\KL{\Law{\hvz_k \ | \ \cF_k} }{\vz_k}\right] \leq \bE\left[\KL{\Law{\hvz_k \ | \ \cG_k} }{\vz_k}\right]
\end{align*}
Let $E$ denote the event $E = \{ u^2_k > \nicefrac{B}{5\eta d} \}$. We note that $\hSigma_k$ is $\cG_k$ measurable and $E$ is $\cF_k$-measurable. Furthermore, let $\vX, \vW, \vZ \sim \cN(0, \vI)$. Hence,
\begin{align*}
    \bE\left[\KL{\Law{\hvz_k \ | \ \cG_k} }{\vz_k}\right] &= \KL{\sqrt{\nicefrac{\eta}{2}} \vN_k + \vX \bEvent{E} + \left(\vI - \nicefrac{\eta}{4}\hSigma_k\right)\vW \bEvent{E^C}}{\vZ\biggr|\cG_k} \\
    &=  \underbrace{\KL{\sqrt{\nicefrac{\eta}{2}}\vN_k + \vX}{\vZ\biggr|\cG_k} \bEvent{E}}_{\textrm{Uncorrected KL}} + \underbrace{\KL{\sqrt{\nicefrac{\eta}{2}}\vN_k + \left(\vI - \nicefrac{\eta}{4}\hSigma_k\right)\vW}{\vZ\biggr|\cG_k}\bEvent{E^C}}_{\textrm{Covariance Corrected KL}}
\end{align*}
We proceed by controlling the uncorrected and covariance corrected KL terms separately.
\paragraph{Bounding the Uncorrected KL} To control the uncorrected KL, we essentially repeat the same arguments as Theorem \ref{thm:sgld-stat-indistinguishable}. In particular, let $\vX_1, \vX_2, \vZ_1, \vZ_2, \vZ_3 \iidsim \cN(0, \vI)$. From Lemma \ref{lem:rev_t_2}, it follows that,
\begin{align*}
    \KL{\sqrt{\nicefrac{\eta}{2}}\vN_k + \vX}{\vZ\biggr|\cG_k} &= \KL{\sqrt{\nicefrac{1}{2}}\vX_1 + \sqrt{\nicefrac{1}{2}}\vX_2 + \sqrt{\nicefrac{\eta}{2}}\vN_k}{\sqrt{\nicefrac{1}{2}}\vZ_1 + \sqrt{\nicefrac{1}{2}}\vZ_2\biggr|\cG_k} \\
    &\leq \frac{1}{2}\wass{2}^2\left(\vX_1 + \sqrt{\eta}\vN_k, \vZ_1\biggr|\cG_k\right) \\
    &\leq \wass{2}^2\left(\vX_1 + \sqrt{\eta}\vN_k, \sqrt{\vI + \eta \Sigma_{k}}\vZ_3\biggr|\cG_k\right) + \wass{2}^2\left(\sqrt{\vI + \eta \Sigma_{k}}\vZ_3, \vZ_2\biggr|\cG_k\right)
\end{align*}
Hence, 
\begin{align*}
    \KL{\sqrt{\nicefrac{\eta}{2}}\vN_k + \vX}{\vZ\biggr|\cG_k} \mathbb{I}_{\{u^2_k > \nicefrac{B}{5\eta d}\}} &\leq \wass{2}^2\left(\vX_1 + \sqrt{\eta}\vN_k, \sqrt{\vI + \eta \Sigma_{k}}\vZ_3\biggr|\cG_k\right) \mathbb{I}_{\{u^2_k > \nicefrac{B}{5\eta d}\}} \\
    &+ \wass{2}^2\left(\sqrt{\vI + \eta \Sigma_{k}}\vZ_3, \vZ_2\biggr|\cG_k\right) \mathbb{I}_{\{u^2_k > \nicefrac{B}{5\eta d}\}}
\end{align*}
We recall from our analysis of Theorem \ref{thm:sgld-stat-indistinguishable} that the first term corresponds to the Wasserstein CLT term whereas the second term corresponds to the covariance mismatch. Thus, repeating the same arguments as Theorem \ref{thm:sgld-stat-indistinguishable}, we conclude that,
\begin{align*}
    \wass{2}^2\left(\sqrt{\vI + \eta \Sigma_{k}}\vZ_3, \vZ_2\biggr|\cG_k\right) &\leq \frac{\eta^2 u^4_k }{4 B^2}
\end{align*}
For the Wasserstein CLT term, we note that $\vN_k = \nicefrac{1}{\sqrt{B}}\sum_{j=1}^{B} \nicefrac{\vN_{k, j}}{\sqrt{B}}$ where $\norm{\vN_{k,j}} \leq u_k$. Thus, from the Wasserstein CLT of \cite{zhai2018clt}, it follows that,
\begin{align*}
    \wass{2}^2\left(\vX_1 + \sqrt{\eta}\vN_k, \sqrt{\vI + \eta \Sigma_{k}}\vZ_3\biggr|\cG_k\right) &= \eta \wass{2}^2\left(\vN_k, \sqrt{\Sigma_k}\vZ_3\biggr|\cG_k\right) \leq \frac{25 \eta d (1 + \log(B))^2 u^2_k}{B^2}
\end{align*}
Hence, we obtain,
\begin{align*}
    \KL{\sqrt{\nicefrac{\eta}{2}}\vN_k + \vX}{\vZ\biggr|\cG_k} \mathbb{I}_{\{u^2_k > \nicefrac{B}{5\eta d}\}} &\leq \frac{\eta^2 u^4_k }{4 B^2}\mathbb{I}_{\{u^2_k > \nicefrac{B}{5\eta d}\}} 
    + \frac{25 \eta d (1 + \log(B))^2 u^2_k}{B^2} \mathbb{I}_{\{u^2_k > \nicefrac{B}{5\eta d}\}}
\end{align*}
From Lemma \ref{lem:tail-markov-control}, we note that,
\begin{align*}
    \bE\left[u^2_k \bEvent{u^2_k > \nicefrac{B}{5\eta d}}\right] &\leq \frac{75 \eta^2 d^2}{B^2} \bE\left[u^6_k\right] \\
    \bE\left[u^4_k \bEvent{u^2_k > \nicefrac{B}{5\eta d}} \right] &= \bE\left[u^4_k \bEvent{u^4_k > \left(\nicefrac{B}{5\eta d}\right)^2} \right] \\
    &\leq \frac{375 \eta^3 d^3}{B^3} \bE\left[u^6_k\right]
\end{align*}
It follows that,
\begin{align*}
\bE\left[\KL{\sqrt{\nicefrac{\eta}{2}}\vN_k + \vX}{\vZ\biggr|\cG_k} \bEvent{u^2_k > \nicefrac{B}{5\eta d}}\right] &\leq \frac{100 \eta^5 d^3}{B^5}\bE\left[u^6_k\right] + \frac{1875 \eta^3 d^3\left(1 + \log(B)\right)^2}{B^4} \bE\left[u^6_k\right]  
\end{align*}
From the \ref{eqn:defn-grad-growth} and \ref{eqn:defn-moment-growth} conditions, we know that, $\bE\left[u^6_k\right] \leq 32 \left( M^6 C_{6} d^{3} + G^6\right)$. Hence, the uncorrected KL term is bounded as follows,
\begin{align}
\label{eqn:cc-sgld-uncorrected-kl}
\bE\left[\KL{\sqrt{\nicefrac{\eta}{2}}\vN_k + \vX}{\vZ\biggr|\cG_k} \bEvent{E}\right] \leq \left(\frac{3200 \eta^5}{B^3} + \frac{1875 \eta^3 \left(1 + \log(B)\right)^2}{B^4}\right)\left( M^6 C_{6} d^{6} + G^6 d^3\right)
\end{align}
\paragraph{Bounding the Covariance Corrected KL} Let $\vZ_1, \vZ_2 \ \iidsim \cN(0, \vI)$. We note that the presence of the indicator $\bEvent{E^c} = \bEvent{u^2_k \leq \nicefrac{B}{5\eta d}}$ ensures $\Sigma_k, \hSigma_k \prec \nicefrac{\vI}{2}$. Hence, applying Lemma \ref{lem:kl_triangle},
\begin{align*}
    \KL{\sqrt{\nicefrac{\eta}{2}}\vN_k + \left(I - \nicefrac{\eta}{4}\hSigma_k\right)\vW}{\vZ\biggr|\cG_k}\bEvent{E^C} &\leq  2 \wass{2}^2 \left(\sqrt{\nicefrac{1}{2}\vI - \nicefrac{\eta}{2}\Sigma_k}\vZ_1 + \sqrt{\nicefrac{\eta}{2}}\vN_k, \tfrac{\vZ_2}{\sqrt{2}} \biggr|\cG_k\right) \bEvent{E^C} \\
    &+ 4\norm{\vI - \eta \Sigma_k}\KL{\sqrt{\nicefrac{1}{2}\vI - \nicefrac{\eta}{2}\hSigma_k}\vZ_1}{\sqrt{\nicefrac{1}{2}\vI - \nicefrac{\eta}{2}\Sigma_k}\vZ_2\biggr|\cG_k} \bEvent{E^C} \\
    &+ \tfrac{\eta^3}{4}\Tr\left(\hSigma_k\right)^3 \bEvent{E^C} \\
    &\leq \wass{2}^2 \left(\sqrt{\vI - \eta\Sigma_k}\vZ_1 + \sqrt{\eta}\vN_k, \vZ_2\biggr|\cG_k \right) \bEvent{E^C} \\
    &+ 4\KL{\sqrt{\nicefrac{1}{2}\vI - \nicefrac{\eta}{2}\hSigma_k}\vZ_1}{\sqrt{\nicefrac{1}{2}\vI - \nicefrac{\eta}{2}\Sigma_k}\vZ_2\biggr|\cG_k} \bEvent{E^C} \\
    &+ \frac{\eta^3}{4}\Tr\left(\hSigma_k\right)^3 \bEvent{E^C}
\end{align*}
It follows that,
\begin{align*}
    \bE\left[\KL{\sqrt{\nicefrac{\eta}{2}}\vN_k + \sqrt{I - \nicefrac{\eta}{2}\hSigma_k}\vW}{\vZ\biggr|\cG_k}\bEvent{E^C}\right] &\leq \bE\left[\wass{2}^2 \left(\sqrt{\vI - \eta\Sigma_k}\vZ_1 + \sqrt{\eta}\vN_k, \vZ_2 \biggr|\cG_k\right) \bEvent{E^C}\right] \\
    &+ \frac{\eta^3}{4}\bE\left[\Tr\left(\hSigma_k\right)^3 \bEvent{E^C}\right] \\
    &+ \bE\left[\KL{\sqrt{\nicefrac{1}{2}\vI - \nicefrac{\eta}{2}\hSigma_k}\vZ_1}{\sqrt{\nicefrac{1}{2}\vI - \nicefrac{\eta}{2}\Sigma_k}\vZ_2\biggr|\cG_k} \bEvent{E^C} \right]
\end{align*}
As before, the first term corresponds to the Wasserstein CLT, the second term corresponds to the error due to linearization  of the matrix square root, and the last term corresponds to the covariance mismatch of CC-SGLD, which depends upon how well $\hSigma_k$ approximates $\Sigma_k$. We note that the error due to linearization can be easily controlled via the \ref{eqn:defn-grad-growth} and \ref{eqn:defn-moment-growth} conditions as follows:
\begin{align*}
    \frac{\eta^3}{4}\bE\left[\Tr\left(\hSigma_k\right)^3 \bEvent{E^C}\right] &\leq \frac{\eta^3}{4B^3} \bE\left[u^6_k\right] \leq \frac{8\eta^3}{B^3}\left( M^6 C_{6} d^{3} + G^6\right) 
\end{align*}
where we use the fact that $\Tr\left(\hSigma_k\right) \leq \tfrac{u^2_k}{B}$ and $\bE\left[u^6_k\right] \leq 32 \left( M^6 C_{6} d^{3} + G^6\right)$ as per the \ref{eqn:defn-grad-growth} and \ref{eqn:defn-moment-growth} conditions
We then control the Wasserstein CLT term via Lemma \ref{lem:wass-clt-new}. We note that, similar to Theorem \ref{thm:sgld-stat-indistinguishable}, the presence of the indicator $\bEvent{E^c} = \bEvent{u^2_k \leq \nicefrac{B}{5\eta d}}$ ensures that all conditions of Lemma \ref{lem:wass-clt-new} are satisfied. Hence, repeating similar arguments, we obtain the following bound
\begin{align*}
    \wass{2}^2 \left(\sqrt{\vI - \eta\Sigma_k}\vZ_1 + \sqrt{\eta}\vN_k, \vZ_2 \biggr|\cG_k\right) \bEvent{E^C} &\leq \frac{25 \eta^3 d \left(1 + \log(B)\right)^2 u^6_k}{B^4} 
\end{align*}
From the \ref{eqn:defn-grad-growth} and \ref{eqn:defn-moment-growth} conditions, we know that, $\bE\left[u^6_k\right] \leq 32 \left( M^6 C_{6} d^{3} + G^6\right)$. Hence,
\begin{align*}
    \bE\left[\wass{2}^2 \left(\sqrt{\vI - \eta\Sigma_k}\vZ_1 + \sqrt{\eta}\vN_k, \vZ_2 \biggr|\cG_k\right) \bEvent{E^C}\right] &\leq \frac{800 \eta^3 d}{B^4}\left(1 + \log(B)\right)^2 \left( M^6 C_{6} d^{3+1} + G^6 d\right)
\end{align*}
We now control the covariance mismatch term by applying Lemma \ref{lem:kl_mismatch}. We note that, since $\Tr(\Sigma), \Tr(\hSigma) \leq \nicefrac{u^2_k}{B}$, the presence of the indicator $\bEvent{u^2_k \leq \nicefrac{B}{5\eta d}}$ ensures that $\eta \Sigma_k, \eta \hSigma_k \prec \nicefrac{\vI}{2}$. Hence, it follows that,
\begin{align*}
    \KL{\sqrt{\nicefrac{1}{2}\vI - \nicefrac{\eta}{2}\hSigma_k}\vZ_1}{\sqrt{\nicefrac{1}{2}\vI - \nicefrac{\eta}{2}\Sigma_k}\vZ_2\biggr|\cG_k} \bEvent{E^C} &\leq \frac{1}{2}\Tr\left(\left(\vI - \eta \Sigma_k\right)^{-1} \left(\hSigma_k - \Sigma_k\right)\right) \bEvent{E^c} \\
    &+  \frac{1}{2} \sum_{j=1}^{\infty} \frac{\eta^j}{j} \left[\Tr\left(\hSigma^j_k\right) - \Tr\left(\hSigma^j_k\right) \right] \bEvent{E^c}
\end{align*}
We now control this quantity by first taking conditional expectation with respect to $\cF_k$. From Property 1 of Lemma \ref{lem:cov-estimation-guarantee}, we know that $\bE\left[ \hSigma_k \ | \ \cF_k\right] = \Sigma_k$. Furthermore $\bEvent{E}$ is $\cF_k$ measurable. Hence,
\begin{align*}
    \bE\left[\KL{\sqrt{\nicefrac{1}{2}\vI - \nicefrac{\eta}{2}\hSigma_k}\vZ_1}{\sqrt{\nicefrac{1}{2}\vI - \nicefrac{\eta}{2}\Sigma_k}\vZ_2\biggr|\cG_k}  | \cF_k \right] \bEvent{E^C}&\leq \frac{\eta^2}{4}\left[\bE\left[\Tr\left(\hSigma^2_k\right) \ | \ \cF_k \right] - \Tr\left(\hSigma^2_k\right) \right] \\
    &+ \frac{1}{6} \sum_{j=3}^{\infty} \eta^j \Tr\left(\hSigma^j_k\right) \bEvent{E^c}
\end{align*}
Applying Property 2 and Property 3 of Lemma \ref{lem:cov-estimation-guarantee} wherever appropriate,
\begin{align*}
    \frac{\eta^2}{4}\left[\bE\left[\Tr\left(\hSigma^2_k\right) \ | \ \cF_k \right] - \Tr\left(\hSigma^2_k\right) \right] &\leq \frac{\eta^2 u^4_k}{B^3} \\
    \sum_{j=3}^{\infty} \eta^j \Tr\left(\hSigma^j_k\right) \bEvent{E^c} &= \left(\frac{2 \eta u^2_k}{B}\right)^3 \sum_{j=0}^{\infty} \left(\frac{2\eta u^2_k}{B}\right)^j \bEvent{u^2_k \leq \nicefrac{B}{5\eta d}} \\
    &\leq \frac{8 \eta^3 u^6_k}{B^3} \sum_{j=0}^{\infty} \left(\nicefrac{2}{5}\right)^j \leq \frac{40 \eta^3 u^6_k}{3B^3}
\end{align*}
Hence, 
\begin{align*}
    \bE\left[\KL{\sqrt{\nicefrac{1}{2}\vI - \nicefrac{\eta}{2}\hSigma_k}\vZ_1}{\sqrt{\nicefrac{1}{2}\vI - \nicefrac{\eta}{2}\Sigma_k}\vZ_2\biggr|\cG_k}  | \cF_k \right] \bEvent{E^C} &\leq \frac{\eta^2 u^4_k}{B^3} + \frac{3 \eta^3 u^6_k}{B^3}
\end{align*}
Recall that, as per the \ref{eqn:defn-grad-growth} and \ref{eqn:defn-moment-growth} conditions, $\bE\left[u^4_k\right] \leq 8\left(M^4 C_{4} d^{2} + G^4\right)$ and $\bE\left[u^6_k\right] \leq 32\left(M^6 C_{6} d^{3} + G^6\right)$. Hence, the covariance corrected KL term is controlled as,
\begin{align}
\label{eqn:cc-sgld-cov-corr-kl}
    \bE\left[\KL{\sqrt{\nicefrac{1}{2}\vI - \nicefrac{\eta}{2}\hSigma_k}\vZ_1}{\sqrt{\nicefrac{1}{2}\vI - \nicefrac{\eta}{2}\Sigma_k}\vZ_2\biggr|\cG_k} \bEvent{E^C}\right]  &\leq \frac{8\eta^2}{B^3}\left(M^4 C_{4} d^{2} + G^4\right) + \frac{96 \eta^3}{B^3}\left(M^6 C_{6} d^{3} + G^6\right)
\end{align}
From \eqref{eqn:cc-sgld-uncorrected-kl} and \eqref{eqn:cc-sgld-cov-corr-kl}, we finally obtain the following statistical indistinguishability guarantee,
\begin{align*}
    \KL{\Law{\hvx_{1:K}}}{\Law{\vx_{1:K}}} &\leq \frac{8\eta^2 K}{B^3}\left(M^4 C_{4} d^{2} + G^4\right) + \frac{96 \eta^3 K}{B^3}\left(M^6 C_{6} d^{3} + G^6\right) \\
    &+ \left(\frac{3200 \eta^5 K}{B^3} + \frac{1875 \eta^3 K \left(1 + \log(B)\right)^2}{B^4}\right)\left( M^6 C_{6} d^{6} + G^6 d^3\right)
\end{align*}
\end{proof}
\subsection{Convergence of Non-Smooth CC-SGLD under LO}
\label{proof:cor-cc-sgld-lo}
As before, we use $\Renyi{\mu}{\nu}{q}$ to denote the R\'enyi divergence of order $q$ between two measures $\mu$ and $\nu$.
\begin{corollary}[Convergence of CC-SGLD under LO]
\label{cor:cc-sgld-unstable-lo}
Let the \ref{eqn:defn-holder-smoothness}, \ref{eqn:defn-grad-growth}, and \ref{eqn:defn-moment-growth} be satisfied with $p = 6$. Furthermore, assume the target $\pistar$ satisfies \ref{eqn:defn-latala-oleskiewicz} for some $\alpha \in [1, 2]$ and define $\beta = \nicefrac{2}{\alpha} - 1$. Then, for  $\epsilon \leq \nicefrac{1}{\mathsf{poly}(d)}$, the last iterate of CC-SGLD, under appropriate Gaussian initialization, requires $N$ stochastic gradient oracle calls to ensure $\TV(\Law{\hvx_K}, \pistar) \leq \epsilon$, where
\begin{align*}
     N = \Thetatilde\left(\frac{d^{\max \left\{ 1 + \beta\left(1 + \nicefrac{1}{s}\right), \nicefrac{4}{3}\left(1+\beta + \nicefrac{\beta}{2s}\right)\right\}}}{\epsilon^{\nicefrac{2}{s}}}\right)
\end{align*}
\end{corollary}
\begin{proof}
The proof of this result closely resembles that of Corollary \ref{cor:sgld-unstable-lo}. In particular, we know from Theorem 7 of \cite{sinho-lmc-poincare-lsi} that under appropriate Gaussian initialization such that $\Delta_0 = \Renyi{\Law{\vx_0}}{\pistar}{3}  = O(d)$, the following choice of $\eta$ and $K$ suffices to ensure $\TV\left(\Law{\vx_{K+1}}, \pistar\right) \leq \epsilon$
\begin{align*}
    \eta &= \Thetatilde\left( \frac{\epsilon^{\nicefrac{2}{s}}}{d \Delta^{\nicefrac{\beta}{s}}_0} \right) \\
    K &= \Thetatilde\left(\frac{d \Delta^{\beta(1 + \nicefrac{1}{s})}_0}{\epsilon^{\nicefrac{2}{s}}}\right) \\
    T &= \eta K = \Thetatilde\left(\Delta^{\beta}_0\right) 
\end{align*}
where $\beta = \nicefrac{2}{\alpha} - 1$. Since $\epsilon \leq \nicefrac{1}{\mathsf{poly}(d)}$, under this choice of $\eta$ and $T$, Theorem \ref{thm:cc-sgld-stat-indistinguishable} suggests that $\KL{\hvx_{1:K}}{\vx_{1:K}} \leq O(\nicefrac{\eta T d^2}{B^3})$. Thus, to ensure $\TV(\Law{\hvx_{K+1}}, \Law{\vx_{K+1}}) \leq \nicefrac{\epsilon}{2}$, it suffices to set $B \geq \Otilde\left(\max \left\{ 1, d^{\nicefrac{1}{3}} \Delta^{\nicefrac{\beta}{3}\left(1-\nicefrac{1}{s}\right)}_0 \right\}\right)$. Hence, $\TV(\Law{\hvx_{K+1}}, \Law{\vx_{K+1}}) \leq \epsilon$ by subadditivity of Total Variation. The required stochastic gradient complexity is,
\begin{align*}
    N &= KB = \Thetatilde\left(\max\left\{\frac{d \Delta^{\beta\left(1 + \nicefrac{1}{s}\right)}_0}{\epsilon^{\nicefrac{2}{s}}}, \frac{d^{\nicefrac{4}{3}} \Delta^{\nicefrac{4\beta}{3} + \nicefrac{2\beta}{3s}}_0}{\epsilon^{\nicefrac{2}{s}}}\right\}\right) \\
    &= \Thetatilde\left(\frac{d^{\max \left\{ 1 + \beta\left(1 + \nicefrac{1}{s}\right), \nicefrac{4}{3}\left(1+\beta + \nicefrac{\beta}{2s}\right)\right\}}}{\epsilon^{\nicefrac{2}{s}}}\right)
\end{align*}
where we use the fact that $\Delta_0 = O(d)$.
\end{proof}
We observe that for $s \geq \nicefrac{1}{2}$, the oracle complexity of CC-SGLD strictly improves upon SGLD (whereas for $s \leq \nicefrac{1}{2}$ both have the same oracle complexity since $B = O(1)$ in that setting). We elucidate this improvement for the special cases of LSI and PI as follows.

\paragraph{Rates Under Smoothness and \ref{eqn:defn-log-sobolev}} Recall that \ref{eqn:defn-latala-oleskiewicz} of order $\alpha = 2$ is equivalent to \ref{eqn:defn-log-sobolev}. Furthermore, LSI and Holder continuity together imply smoothness. To this end, Corollary \ref{cor:cc-sgld-unstable-lo} implies an oracle complexity of $\Thetatilde(\tfrac{d^{\nicefrac{4}{3}}}{\epsilon^2})$ for attaining last iterate $\epsilon$-convergence in TV. We note that this improves upon all our prior guarantees for SGLD under this setting. However, we highlight that the guarantee implied by Corollary \ref{cor:cc-sgld-unstable-lo} is unstable. 

\paragraph{Rates Under \ref{eqn:defn-poincare}} Since \ref{eqn:defn-latala-oleskiewicz} of order $\alpha = 1$ is equivalent to \ref{eqn:defn-poincare}, Corollary \ref{cor:cc-sgld-unstable-lo} implies an oracle complexity of $\Thetatilde\left(\frac{d^{\max\left\{2 + \nicefrac{1}{s}, \nicefrac{8}{3} + \nicefrac{2}{3s}\right\}}}{e^{\nicefrac{2}{s}}}\right)$ under this setting for last-iterate $\epsilon$-convergence in TV. For $s=1$, the implied oracle complexity is $\Thetatilde\left(\frac{d^{\nicefrac{10}{3}}}{\epsilon^2}\right)$. This strictly improves upon the last-iterate guarantee implied by Corollary \ref{cor:sgld-unstable-lo} in this setting. When compared to the $\Otilde\left(\frac{d^{2.5}}{\epsilon^4}\right)$ average-iterate guarantee implied by Theorem \ref{thm:sgld-fd-and-pi}, we note an improved dependence on $\epsilon$ at the cost of a sublinear additional $d$ dependence 

\section{Analysis of RBM and CC-RBM}
\subsection{Proof of Theorem \ref{thm:rbm-stat-indistinguishability}}
\label{proof:thm-rbm-stat-indistinguishability}
\begin{proof}
The proof of this result is similar to that of Theorem \ref{thm:sgld-stat-indistinguishable} and uses the same technical tools. We first note that the iterates of IPD and RBM, with the same initialization and step-size can be expressed as,
\begin{align*}
    \vx_{k+1}^{i} &= \vx_k^{i} + \eta \vg^{i}_k(\vx_k^{i}) +\frac{\eta}{n} \sum_{j =1}^{n} \vK^{ij}_{k}(\vx^{i}_k,\vx^{j}_k) + \sqrt{\eta}\sigma \vz^{i}_k \\
    \hvx_{k+1}^{i} &= \hvx_k^{i} + \eta \vg^{i}_k(\hvx_k^{i}) +\frac{\eta}{n} \sum_{j =1}^{n} \vK^{ij}_{k}(\hvx^{i}_k,\hvx^{j}_k) + \sqrt{\eta}\sigma \hvz^{i}_k
\end{align*}
where $\vz^{i}_k$ and $\hvz^{i}_k$ are defined as follows for any $k \in [K], n \in [n]$
\begin{align*}
    \vz^i_k &= \epsilon^i_k \sim \cN(0, \vI), \\
    \hvz^i_k &= \frac{\sqrt{\eta}}{\sigma} \vN^i_k + \heps^i_k, \ \ \heps^i_k \sim \cN(0, \vI) \\
    \vN^i_k &= \frac{1}{B} \sum_{j=1}^{B} \left[ \vK_k^{iI^{ij}_k}(\hvx_k^{i},\hvx_k^{I^{ij}_k}) - \frac{1}{n} \sum_{l=1}^{n} \vK^{il}_{k}(\hvx^{i}_k,\hvx^{l}_k)\right]
\end{align*}
where $I^{i1}_k, \dots,I^{iB}_k \ \iidsim \ \mathsf{Uniform}([n])$ for every $i \in [n]$. Define the filtration $\cF_k = \sigma(\hvx^{i}_0, \dots, \hvx^{i}_k, \hvz^{i}_0, \dots, \hvz^{i}_{k-1} | i \in [n])$. Clearly, $\vN^i_k$ is an empirical average of zero-mean i.i.d random variables conditioned on $\cF_k$. Furthermore, we note that IPD and RBM admit the same random function representation, i.e., there exists a measurable function $H_K$ such that,
\begin{align*}
    (\hvx^{i}_{1:K+1} | i \in [n]) &= H_K(\hvx^{i}_0, \hvz^i_0, \hvz^i_1, \dots, \hvz^i_K | i \in [n]) \\
    (\vx^{i}_{1:K+1} | i \in [n]) &= H_K(\vx^{i}_0, \vz^i_0, \vz^i_1, \dots, \vz^i_K | i \in [n])
\end{align*}
Since $\Law{\hvx^i_0} = \Law{\vx^i_0}$ for every $i \in [n]$, we use the data processing inequality and Lemma \ref{lem:kl_chain_rule} to conclude the following.
\begin{align*}
    \KL{\hvx^{i}_{1:K+1} | i \in [n]}{\vx^{i}_{1:K+1} | i \in [n]} &= \sum_{k=0}^{K} \sum_{i=1}^{n} \bE\left[\KL{\Law{\hvz^i_k \ | \ \cF_k} }{\vz^i_k}\right]
\end{align*}
We shall now control each term in the above summation by following the same steps as Theorem \ref{thm:sgld-stat-indistinguishable}. To this end, let $\vX, \vX_1, \vX_2, \vZ, \vZ_1, \vZ_2, \vW \ \iidsim \ \cN(0, \vI)$. It follows that,
\begin{align*}
    \KL{\Law{\hvz^i_k \ | \ \cF_k} }{\vz^i_k} &= \KL{\vX_1 + \nicefrac{\sqrt{\eta}}{\sigma}\vN_k^i}{\vZ_1\biggr|\cF_k} \\
    &= \KL{\sqrt{\nicefrac{1}{2}}\vX_2 + \sqrt{\nicefrac{1}{2}}\vX + \nicefrac{\sqrt{\eta}}{\sigma}\vN^i_k}{\sqrt{\nicefrac{1}{2}}\vZ_2 + \sqrt{\nicefrac{1}{2}}\vZ\biggr|\cF_k} \\
    &\leq \frac{1}{2}\wass{2}^2\left(\vX + \nicefrac{\sqrt{2\eta}}{\sigma}\vN^i_k, \vZ\biggr|\cF_k\right) 
\end{align*}
Where the last inequality follows from Lemma \ref{lem:rev_t_2}. Now, let $\vY = \nicefrac{\sqrt{2\eta}}{\sigma} \vN^i_k$ and define $\Sigma_{\vY} = \bE\left[ \vY \vY^T \biggr|\cF_k \right]$. It follows that,
\begin{align}
\label{eqn:rbm-clt-kl-decomposition}
    \KL{\Law{\hvz^i_k \ | \ \cF_k} }{\vz^i_k} &\leq \frac{1}{2}\wass{2}^2\left(\vX + \sqrt{\eta}\vN_k, \vZ\biggr|\cF_k\right) \nonumber \\
    &\leq \underbrace{\wass{2}^2\left(\vX + \vY, \sqrt{\vI + \Sigma_{\vY}}\vZ\biggr|\cF_k\right)}_{\textrm{Wasserstein CLT Term}} + \underbrace{\wass{2}^2\left(\sqrt{\vI + \Sigma_{\vY}}\vZ, \vW\biggr|\cF_k\right)}_{\textrm{Covariance Mismatch Term}}
\end{align}
We note that, $\Tr(\Sigma_{\vY}) \leq \nicefrac{4\eta M^2}{B\sigma^2}$. Controlling the covariance mismatch term in a manner similar to Theorem \ref{thm:sgld-stat-indistinguishable} (i.e., by direct computation of the Wasserstein distance), we obtain,  
\begin{align}
\label{eqn:rbm-cov-mismatch-bound}
    \wass{2}^2\left(\sqrt{\vI + \Sigma_{\vY}}\vZ, \vW\biggr|\cF_k\right) \leq \frac{\Tr(\Sigma_{\vY})^2}{4} \leq \frac{4 \eta^2 M^4}{\sigma^4 B^2}
\end{align}
To control the Wasserstein CLT term, we observe that the CLT structure of $\vN^i_k$ allows us to express $\vY$ as $\vY = \nicefrac{1}{\sqrt{B}} \sum_{j=1}^{B} \vY^{(j)}$ where $\norm{\vY^{(j)}} \leq \frac{2\sqrt{2\eta}}{\sigma \sqrt{B}} M$. We note that the condition $\eta \leq \nicefrac{B\sigma^2}{40 \eta M^2 d}$ ensures that $\norm{\vY^{(j)}}^2 \leq \nicefrac{1}{5}$ and $\Tr(\Sigma_{\vY}) \leq \nicefrac{1}{5d}$. Thus, all the conditions required to apply Lemma \ref{lem:wass-clt-new} are satisfied. Applying the same arguments used to control the Wasserstein CLT term in Case 2 of Theorem \ref{thm:sgld-stat-indistinguishable}, we conclude that the Wasserstein CLT term is bounded as follows for some universal constant $C_{\mathsf{CLT}}$,
\begin{align}
\label{eqn:rbm-wass-clt-bound}
    \wass{2}^2\left(\vX + \vY, \sqrt{\vI + \Sigma_{\vY}}\vZ\biggr|\cF_k\right) \leq \frac{C_{\mathsf{CLT}}\eta^3 M^6 d (1 + \log(B))^2}{\sigma^6 B^4}
\end{align}
Thus, from equations \eqref{eqn:rbm-clt-kl-decomposition}, \eqref{eqn:rbm-cov-mismatch-bound} and \eqref{eqn:rbm-wass-clt-bound}, we obtain the following statistical indistinguishability guarantee,
\begin{align*}
    \KL{(\hvx^i_k)_{i \in [n], k \in [K]}}{(\vx^i_k)_{i \in [n], k \in [K]}} &\leq C_{\mathsf{Cov}} \frac{\eta^2 M^4 nK}{B^2 \sigma^4} + C_{\mathsf{CLT}} \frac{d\eta^3 M^6 nK (1 + \log(B))^2}{B^4 \sigma^6} 
\end{align*}
where $C_{\mathsf{Cov}}$ and $C_{\mathsf{CLT}}$ are universal constants.
\end{proof}
\subsection{Proof of Theorem \ref{thm:cc-rbm-stat-indistinguishability}}
\label{proof:cc-rbm-stat-indistinguishability}
\begin{proof}
The proof of this result follows a structure similar to Theorem \ref{thm:rbm-stat-indistinguishability} and uses the same techniques as Theorem \ref{thm:cc-sgld-stat-indistinguishable}. We first note that the iterates of IPD and CC-RBM, with the same initialization and step-size can be expressed as,
\begin{align*}
    \vx_{k+1}^{i} &= \vx_k^{i} + \eta \vg^{i}_k(\vx_k^{i}) +\frac{\eta}{n} \sum_{j =1}^{n} \vK^{ij}_{k}(\vx^{i}_k,\vx^{j}_k) + \sqrt{\eta}\sigma \vz^{i}_k \\
    \hvx_{k+1}^{i} &= \hvx_k^{i} + \eta \vg^{i}_k(\hvx_k^{i}) +\frac{\eta}{n} \sum_{j =1}^{n} \vK^{ij}_{k}(\hvx^{i}_k,\hvx^{j}_k) + \sqrt{\eta}\sigma \hvz^{i}_k
\end{align*}
where $\vz^{i}_k$ and $\hvz^{i}_k$ are defined as follows for any $k \in [K], n \in [n]$
\begin{align*}
    \vz^i_k &= \epsilon^i_k \sim \cN(0, \vI), \\
    \hvz^i_k &= \frac{\sqrt{\eta}}{\sigma} \vN^i_k + \left(\vI - \nicefrac{\eta}{2\sigma^2} \hSigma^i_k\right)\heps^i_k, \ \ \heps^i_k \sim \cN(0, \vI) \\
    \vN^i_k &= \frac{1}{B} \sum_{j=1}^{B} \left[ \vK_k^{iI^{ij}_k}(\hvx_k^{i},\hvx_k^{I^{ij}_k}) - \frac{1}{n} \sum_{l=1}^{n} \vK^{il}_{k}(\hvx^{i}_k,\hvx^{l}_k)\right]
\end{align*}
where $I^{i1}_k, \dots,I^{iB}_k \ \iidsim \ \mathsf{Uniform}([n])$ for every $i \in [n]$. Define the filtration $\cF_k = \sigma(\hvx^{i}_0, \dots, \hvx^{i}_k, \hvz^{i}_0, \dots, \hvz^{i}_{k-1} | i \in [n])$. Clearly, $\vN^i_k$ is an empirical average of zero-mean i.i.d random variables conditioned on $\cF_k$. Furthermore, we note that IPD and RBM admit the same random function representation, i.e., there exists a measurable function $H_K$ such that,
\begin{align*}
    (\hvx^{i}_{1:K+1} | i \in [n]) &= H_K(\hvx^{i}_0, \hvz^i_0, \hvz^i_1, \dots, \hvz^i_K | i \in [n]) \\
    (\vx^{i}_{1:K+1} | i \in [n]) &= H_K(\vx^{i}_0, \vz^i_0, \vz^i_1, \dots, \vz^i_K | i \in [n])
\end{align*}
Since $\Law{\hvx^i_0} = \Law{\vx^i_0}$ for every $i \in [n]$, we use the data processing inequality and Lemma \ref{lem:kl_chain_rule} to conclude the following.
\begin{align*}
    \KL{\hvx^{i}_{1:K+1} | i \in [n]}{\vx^{i}_{1:K+1} | i \in [n]} &= \sum_{k=0}^{K} \sum_{i=1}^{n} \bE\left[\KL{\Law{\hvz^i_k \ | \ \cF_k} }{\vz^i_k}\right]
\end{align*}
We now control each term in the above summation by following the same steps as Theorem \ref{thm:cc-sgld-stat-indistinguishable}. To this end, we define the filtration $\cG_k$ as $\mathcal{G}_{k} = \mathcal{F}_{k}\vee \sigma(J_k^{ij},\bar{J}_k^{ij}: i \in [n], j \in [B^{\prime}])$, where $J_k^{ij}, \bar{J}_k^{ij}$ are the additional random variables used in the estimator $\hat{\Sigma}_k^{i}$. We note that by Jensen's inequality  $\bE\left[\KL{\Law{\hvz^i_k \ | \ \cF_k} }{\vz^i_k}\right] \leq \bE\left[\KL{\Law{\hvz^i_k \ | \ \cG_k} }{\vz^i_k}\right]$. Furthermore the condition $\eta \leq \nicefrac{B\sigma^2}{40M^2 d}$ ensures that $\tfrac{\eta}{\sigma^2}\Sigma^i_k,\tfrac{\eta}{\sigma^2} \hSigma^i_k \prec \nicefrac{\vI}{2}$. To this end, let $\vZ_2 \sim \cN(0, \vI)$ be sampled independent of everything else. Hence, from, Lemma \ref{lem:kl_chain_rule},
\begin{align}
\label{eqn:cc-rbm-kl-decompose}
    \KL{\Law{\hvz^i_k \ | \ \cG_k} }{\vz^i_k} &\leq 2\wass{2}^2\left(\frac{\vZ_2}{\sqrt{2}} ,\sqrt{\tfrac{1}{2}\vI-\frac{\eta}{\sigma^2}\Sigma_k^{i}}\vZ_2 + \frac{\sqrt{\eta}}{\sigma}\vN_k^{i}\biggr|\mathcal{G}_{k}\right) + \frac{2\eta^3}{\sigma^6} \Tr(\hSigma_k)^3\nonumber \\
    &\quad+8\lambda_{\max}(\tfrac{\vI}{2}-\tfrac{\eta}{\sigma^2}\Sigma)\KL{\sqrt{\tfrac{1}{2}\vI-\tfrac{\eta}{\sigma^2}\hat{\Sigma}_k^{i}}\vZ_2 }{\sqrt{\tfrac{1}{2}\vI-\tfrac{\eta}{\sigma^2}\Sigma_k^{i}}\vZ_2\biggr|\mathcal{G}_{k}} \nonumber \\
    &\leq 2\wass{2}^2\left(\frac{\vZ_2}{\sqrt{2}} ,\sqrt{\tfrac{1}{2}\vI-\frac{\eta}{\sigma^2}\Sigma_k^{i}}\vZ_2 + \frac{\sqrt{\eta}} {\sigma}\vN_k^{i}\biggr|\mathcal{G}_{k}\right) + \frac{2\eta^3}{\sigma^6} \Tr(\hSigma_k)^3\nonumber \\
    &\quad+4\KL{\sqrt{\tfrac{1}{2}\vI-\tfrac{\eta}{\sigma^2}\hat{\Sigma}_k^{i}}\vZ_2 }{\sqrt{\tfrac{1}{2}I-\tfrac{\eta}{\sigma^2}\Sigma_k^{i}}\vZ_2\biggr|\mathcal{G}_{t}} \nonumber \\
    &= \wass{2}^2\left(\vZ_2 ,\sqrt{I-\frac{2\eta}{\sigma^2}\Sigma_k^{i}}\vZ_2 + \frac{\sqrt{2\eta}}{\sigma}\vN_k^{i}\biggr|\mathcal{G}_{k}\right) + \frac{16\eta^3 M^6}{\sigma^6 B^3} \nonumber \\
    &\quad+4\KL{\sqrt{\tfrac{1}{2}\vI-\tfrac{\eta}{\sigma^2}\hat{\Sigma}_k^{i}}\vZ_2 }{\sqrt{\tfrac{1}{2}I-\tfrac{\eta}{\sigma^2}\Sigma_k^{i}}\vZ_2\biggr|\mathcal{G}_{k}} 
\end{align}
We note that the first term is the Wasserstein CLT term which is controlled in a manner similar to Theorem \ref{thm:rbm-stat-indistinguishability}. Thus, for some universal constant $C_{\textrm{CLT}}$,
\begin{align}
\label{eqn:cc-rbm-wass-clt-bound}
    \wass{2}^2\left(\vZ_2 ,\sqrt{I-\frac{2\eta}{\sigma^2}\Sigma_k^{i}}\vZ_2 + \frac{\sqrt{2\eta}}{\sigma}\vN_k^{i}\biggr|\mathcal{G}_{k}\right) \leq \frac{C_{\mathsf{CLT}}\eta^3 M^6 d (1 + \log B)^2}{\sigma^6 B^4}
\end{align}
To control the covariance mismatch term, we use the same techniques as Theorem \ref{thm:cc-sgld-stat-indistinguishable}. In fact, using Lemma \ref{lem:kl_mismatch} and the properties of the covariance estimator established in Lemma \ref{lem:cov-estimation-guarantee}, we obtain, 
\begin{align}
\label{eqn:cc-rbm-cov-mismatch-bound}
&\mathbb{E}\left[\KL{\sqrt{\tfrac{1}{2}\vI-\tfrac{\alpha}{\sigma^2}\hat{\Sigma}_k^{i}}\vZ_2 }{\sqrt{\tfrac{1}{2}\vI-\tfrac{\alpha}{\sigma^2}\Sigma_k^{i}}\vZ_2\biggr|\mathcal{G}_{k}}\right] \nonumber \\
&= \mathbb{E}\tr\left(\left(\vI-2\frac{\eta\Sigma_k^{i}}{\sigma^2}\right)^{-1}\left(\frac{\eta\Sigma_k^{i}}{\sigma^2} - \frac{\eta\hat{\Sigma}_k^{i}}{\sigma^2}\right)\right)+ \sum_{k=1}^{\infty}\frac{2^{k-1}\eta^k}{\sigma^{2k}}\mathbb{E}\frac{\left[\tr((\hat{\Sigma}_k^{i})^k) - \tr((\Sigma_k^{i})^k)\right]}{k} \nonumber \\
&= \sum_{k=2}^{\infty}\frac{2^{k-1}\eta^k}{\sigma^{2k}}\mathbb{E}\frac{\left[\tr((\hat{\Sigma}_k^{i})^k) - \tr((\Sigma_k^{i})^k)\right]}{k}\nonumber \\
&\leq \frac{4M^4\eta^2}{B^2 B^{\prime}\sigma^4} + \sum_{k=3}^{\infty}\frac{2^{k-1}\eta^k}{\sigma^{2k}}\mathbb{E}\frac{\left[\tr((\hat{\Sigma}_k^{i})^k) - \tr((\Sigma_k^{i})^k)\right]}{k}\nonumber \\
&\leq \frac{4M^4\eta^2}{B^2 B^{\prime}\sigma^4} + \frac{32\eta^3M^6}{3\sigma^6 B^3(1-\tfrac{4\eta M^2}{\sigma^2 B})}  \leq C_{\mathsf{Cov}}\left[\frac{M^4\eta^2}{B^2 B^{\prime}\sigma^4} + \frac{\eta^3M^6}{\sigma^6 B^3} \right]
\end{align}
where $ C_{\mathsf{Cov}}$ is some universal constant. Hence, from equations \eqref{eqn:cc-rbm-kl-decompose}, \eqref{eqn:cc-rbm-wass-clt-bound} and \eqref{eqn:cc-rbm-cov-mismatch-bound}, we obtain the following guarantee
\begin{align*}
    \KL{(\hvx^i_k)_{i \in [n], k \in [K]}}{(\vx^i_k)_{i \in [n], k \in [K]}} &\lesssim \frac{\eta^2 M^4 n K}{B^2 B^{\prime} \sigma^4} + \frac{\eta^3 M^6 n K}{B^3 \sigma^6} +  \frac{d\eta^3 M^6 nK (1 + \log B)^2}{B^4 \sigma^6} 
\end{align*}
\end{proof}

\section{Discussion on Assumptions}
\label{app-sec:assumption-discuss}
\subsection{Stochastic Gradient Growth}
\label{app-subsec:grad-growth}
We discuss how Assumption \ref{as:stoc-grad-growth} is actually weaker than the assumptions made in \cite{raginsky2017non} and \cite{quanquan-sgld}. Beyond this, we also demonstrate how, unlike our result, \cite{raginsky2017non} requires the stochastic gradient noise to become very small (i.e. $\Otilde(\epsilon^4)$) to ensure $\epsilon$-convergence.
\subsubsection{Assumption \ref{as:stoc-grad-growth} and \cite{raginsky2017non}}
We first compare Assumption \ref{as:stoc-grad-growth} to the assumptions made in \cite{raginsky2017non}. Our work analyzes SGLD with the random batch stochastic approximation, (i.e. SGLD run using mini-batch stochastic gradient estimate $\tfrac{1}{B} \sum_{j=1}^{B} \nabla f(\hvx_k, \xi_{k,j})$), which is the most commonly used stochastic approximation in the literature. Moreover, both our work and \cite{raginsky2017non} consider the general stochastic problem $F(\vx) = \bE_{\xi}[f(\vx, \xi)]$ (i.e. both are more general than finite-sum problems) \\

\noindent Under setting, we note that \cite{raginsky2017non} considers the following assumptions 
\begin{itemize}
    \item \textbf{Boundedness at a Point} $\|\nabla f(0, \xi)\| \leq B \ \forall \ \xi \in \Xi$, as implied in Assumption A.1 of \cite{raginsky2017non}
    \item \textbf{Component Smoothness} $\| \nabla f(\vx, \xi) - \nabla f(\vy, \xi) \| \leq M \|\vx - \vy\|, \ \forall \ \vx, \vy \in \bR^d, \xi \in \Xi$ as implied in Assumption A.2 of \cite{raginsky2017non}
\end{itemize}
Applying Assumption A.2 with $\vy=0$ shows that $\|\nabla f(\vx, \xi) \| \leq M \|\vx\| + B \ \forall \ \vx \in \bR^d$. Moreover, since $\nabla F(\vx) = \bE_{\xi}[\nabla f(\vx, \xi)]$, Jensen's inequality implies that $\|\nabla F(\vx)\| = \|\bE_{\xi}[\nabla f(\vx, \xi)]\| \leq \bE_{\xi}[\|\nabla f(\vx, \xi)\|] \leq M \|\vx\| + B$. Finally, from the triangle inequality, we conclude that $\|\nabla F(\vx) - \nabla f(\vx, \xi)\| \leq \|\nabla F(\vx)\| + \|\nabla f(\vx, \xi)\| \leq 2 M \|\vx\| + 2B$. This is identical to Assumption \ref{as:stoc-grad-growth} (with $M \leftrightarrow 2M$ and $G \leftrightarrow 2B$). Thus, Assumption A.1 and A.2 of \cite{raginsky2017non} imply Assumption \ref{as:stoc-grad-growth}. \\

\noindent Along similar lines, one can show that Assumption \ref{as:stoc-grad-growth} is more general than that of \cite{quanquan-sgld}. Firstly, \textbf{\cite{quanquan-sgld} consider the finite-sum problem}, i.e., $F(\vx) = \nicefrac{1}{n} \sum_{i \in [n]} f_i(\vx)$, whereas we consider the general stochastic problem $F(\vx) = \bE_{\xi}[f(\vx, \xi)]$. Moreover, Assumption 4.4 of \cite{quanquan-sgld} assumes \textbf{component smoothness}, i.e., $\|\nabla f_i(\vx) - \nabla f_i(\vy)\| \leq L\|\vx - \vy\| \ \forall \ i \in [n], \vx, \vy \in \bR^d$. Lastly, \cite{quanquan-sgld} uniformly bound the gradient at a point, i.e., their proof relies on the fact that there exists a constant $G$ such that $\| \nabla f_i(0) \| \leq G \ \forall \ i \in [n]$ (see Lemma 6.2 in \cite{quanquan-sgld} and also their proof of Theorem 4.5 in Appendix A.1 which uses Lemma 6.2). \\

\noindent Under this setting, one can apply the same arguments as above to show that Assumption \ref{as:stoc-grad-growth} of our work is weaker than the assumptions in \cite{quanquan-sgld}. In particular, the component smoothness and boundedness at a point assumptions imply $\|\nabla f_i(\vx)\| \leq L \|\vx\| + G \ \forall \ i \in [n], \vx \in \bR^d$. Then, one can apply triangle inequality and Jensen's inequality to show that $\|\nabla F(\vx) - \nabla f_i(\vx)\| \leq 2L \|\vx\| + 2G \ \forall \ i \in [n], \vx \in \bR^d$.  \\

\noindent We note that, unlike \cite{raginsky2017non, quanquan-sgld} our work does not make any component smoothness or uniform boundedness at a point assumptions. 
\subsubsection{Magnitude of Stochastic Gradient Noise in \cite{raginsky2017non}}
In Assumption A.4 of their paper, \cite{raginsky2017non} assume that, for any $\vx \in \bR^d$, \textbf{the variance of the stochastic gradient oracle at $\vx$ is bounded by $2\delta (M^2 \|\vx\|^2 + B^2)$.} \\

\noindent Beyond Assumption A.4, \cite{raginsky2017non} also assume component dissipativity (i.e., dissipativity of each $f(\vx, \xi)$, see Assumption A.3), component smoothness and boundedness at a point (discussed above), to prove the following convergence guarantee for SGLD in Wasserstein-2 distance (see Proposition 3.3 of \cite{raginsky2017non}). 
\begin{equation}
\label{eqn:raginsky-wass-rate}
    \wass{2}(\Law{\hvx_K}, \pistar) \leq C_0 K \eta \delta^{\nicefrac{1}{4}} + C_1 K \eta^{\nicefrac{5}{4}} + C_2 e^{-\nicefrac{K\eta}{\beta c_{\mathsf{LS}}}}
\end{equation}
where $C_0, C_1, C_2, c_{\mathsf{LS}}$ are problem dependent constants. Consider any $\epsilon \geq 0$. To ensure $\wass{2}(\Law{\hvx_K}, \pistar) \leq O(\epsilon)$, one must ensure each term in \eqref{eqn:raginsky-wass-rate} is $O(\epsilon)$. To set  $ C_2 e^{-\nicefrac{K\eta}{\beta c_{\mathsf{LS}}}} \leq O(\epsilon)$, $\eta$ and $K$ must satisfy $\eta K \gtrsim \ln(\nicefrac{1}{\epsilon})$. This implies that $K\eta \delta^{\nicefrac{1}{4}} \gtrsim \delta^{\nicefrac{1}{4}} \ln(\nicefrac{1}{\epsilon})$. However, to ensure $\epsilon$-convergence, $ K \eta \delta^{\nicefrac{1}{4}}$ must be $O(\epsilon)$ which implies $\delta^{\nicefrac{1}{4}} \ln(\nicefrac{1}{\epsilon}) \leq O(\epsilon)$ i.e. $\delta \leq O(\epsilon^4)$ \\

\noindent Note that since $\delta$ effectively controls the variance of the stochastic gradient oracle, $\delta \leq O(\epsilon^4)$ means that the strength of the stochastic gradient noise must be very small to ensure convergence. On the contrary, in Assumption \ref{as:stoc-grad-growth} (and in its relaxation in Assumption \ref{as:stoc-grad-subg}), we allow $M = O(1), G = O(\sqrt{d})$ and only require a constant $O(\sqrt{d})$ (or $O(d^{\nicefrac{1}{3}})$ for CC-SGLD) batch size in all our results. Hence, our results ensure $\epsilon$-convergence without needing the strength of the stochastic gradient noise to diminish with $\epsilon$. 
\subsection{Moment Bounds}
\label{app-subsec:assumption-moment-bounds}
We now demonstrate that dissipativity and smoothness of $F$ along with the growth condition on stochastic gradients implies \ref{eqn:defn-moment-growth}. This result is a straightforward adaptation of Lemma 3.2 of \cite{raginsky2017non} under relaxed conditions (as highlighted above, the growth condition is weaker than Assumptions A.1 and A.2 of \cite{raginsky2017non}. Moreover. Raginsky et. al. assume dissipativity for each component $f(\vx, \xi)$ while we only assume it for $F$).  \\

\noindent Suppose $F$ is \ref{eqn:defn-smoothness} and $(m, b)$ dissipative, i.e., $\dotp{\nabla F(\vx)}{\vx} \geq m \|\vx\|^2 - b$. Also assume that \ref{eqn:defn-grad-growth} is satisfied (the proof holds even with \ref{eqn:defn-grad-growth-subg}). For convenience, assume $\nabla F(0) = 0$, i.e., $0$ is a stationary point of $F$ (the result holds even if we consider any arbitrary stationary point $\vx^* \neq 0$). Recall that the trajectory $\hvx_{1}, \hvx_{2}, \dots, \hvx_K$ follows the update rule
\begin{align*}
    \hvx_{k+1} = \hvx_{k} - \eta \vg_k + \sqrt{2\eta} \epsilon_k, \quad \vg_k = \tfrac{1}{B} \sum_{j=1}^{B} \nabla f(\hvx_k, \xi_{k, j}) \quad \epsilon_k \sim \cN(0, \vI)
\end{align*}
Let $\hvy_k = \hvx_k - \eta \vg_k$. Using the fact that $\epsilon_k \sim \cN(0, \vI)$
\begin{equation}
\label{eqn:moment-bound-ineq1}
    \bE[\|\hvx_{k+1}\|^2|\hvx_k] = \bE\left[\|\hvy_k\|^2 + \sqrt{2\eta}\dotp{\epsilon_k}{\vg_k} + 2\eta\|\epsilon_k\|^2|\hvx_k\right] = \bE[\|\hvy_k\|^2 | \hvx_k] + 2\eta d
\end{equation}
Let $\vN_k = \vg_k - \nabla F(\hvx_k)$. It follows that, $\|\hvy_k\|^2 = \|\hvx_k\|^2 - 2\eta \dotp{\vg_k}{\hvx_k} + \eta^2 \|\vg_k\|^2$. Moreover, $\bE[\vg_k | \hvx_k] = \nabla F(\hvx_k)$, and, by the \ref{eqn:defn-smoothness} and \ref{eqn:defn-grad-growth} conditions, 
\begin{align*}
    \bE[\|\vg_k\|^2|\hvx_k] \leq \|\nabla F(\hvx_k)\|^2 + 2(M^2 \|\hvx_k\|^2 + G^2) \leq (L^2 + 2M^2)\|\hvx_k\|^2 + 2G^2
\end{align*}
Hence,
\begin{align}
\label{eqn:moment-bound-ineq2}
    \bE[\|\hvy_k\|^2 | \hvx_k] &\leq (1 + \eta^2 L^2 + 2\eta^2 M^2)\|\hvx_k\|^2 - 2\eta \dotp{\nabla F(\hvx_k)}{\hvx_k} + 2\eta^2 G^2\nonumber \\
    &\leq (1 - 2\eta m + \eta^2 L^2 + 2\eta^2 M^2)\|\hvx_k\|^2 + 2\eta^2 G^2 + 2\eta b
\end{align}
where the last inequality follows from dissipativity. Now, setting $\eta \leq \tfrac{m}{L^2 + 2M^2}$, we conclude from \eqref{eqn:moment-bound-ineq1} and \eqref{eqn:moment-bound-ineq2}, 
\begin{align*}
    \bE[\|\hvx_{k+1}\|^2] \leq (1 - \eta m) \bE[\|\hvx_k\|^2] + 2\eta d + 2\eta b + 2\eta^2 G^2
\end{align*}
Unrolling the above recurrence and using we obtain,
\begin{align*}
    \bE[\|\hvx_k\|^2] \leq 2d/m + 2b/m + 2\eta G^2/m \lesssim d \ \forall \ k \in [K]
\end{align*}
From the above and \eqref{eqn:moment-bound-ineq2} it also follows that $\bE[\|\hvy_k\|^2] \lesssim d\ \forall \ k \in [K]$. Thus, we have established \ref{eqn:defn-moment-growth} for $p=2$. We follow a similar procedure for $p=4$. In particular, expanding powers and taking expectations wrt $\epsilon_k$ gives us
\begin{align}
\label{eqn:moment-bound-ineq3}
\bE[\|\hvx_{k+1}\|^4] &\leq \bE[\|\hvy_k\|^4] +  8\eta \bE[\|\hvy_k\|^2] + 8\eta^2 d^2 + 4\eta d\bE[\|\hvy_k\|]^2 \nonumber \\
&\leq \bE[\|\hvy_k\|^4] + C\left(\eta d + \eta d^2 + \eta^2 d^2\right)
\end{align}
where $C \geq 0$ is an absolute numerical constant. The last inequality uses the fact that $\bE[\|\hvy_k\|^2] \lesssim d$. As before, using $\hvy_k = \hvx_k - \eta \nabla F(\hvx_k) - \eta \vN_k$, using the dissipativity and moment growth conditions and expanding powers, we get, 
\begin{align*}
    \|\hvy_k\|^4 \leq \left[1 - 4 \eta m + c_1 \eta^2(L^2 + M^2) + c_2 \eta^4 (L^2 + M^2)^2 \right]\|\hvx_k\|^4 + (c_3 \eta^2 G^2 + c_4 \eta b) \|\hvx_k\|^2 + c_5 \eta^4 G^4
\end{align*}
where $c_1, \ldots, c_5$ are universal constants. Taking expctations and using the fact that $\bE[\|\hvx_k\|^2] \lesssim d$, we conclude from the above inequality and \eqref{eqn:moment-bound-ineq3} that,
\begin{align*}
    \bE[\|\hvx_{k+1}\|^4] \leq \left[1 - 4 \eta m + c_1 \eta^2(L^2 + M^2) + c_2 \eta^4 (L^2 + M^2)^2 \right]\bE[\|\hvx_k\|^4] + C(\eta d + \eta d^2 + \eta^2 d^2 + \eta^4 G^4)
\end{align*}
Setting $\eta \leq \tfrac{m}{c_6(L^2 + M^2)}$ for some large enough universal constant $c_6$ and unrolling the above recurrence, we conclude $\bE[\|\hvx_{k}\|^4] \lesssim d^2 \ \forall k \in [K]$. A similar procedure can be followed for $p=6, 8, \ldots$ to show that for any even $p$, the \ref{eqn:defn-moment-growth} condition holds under some appropriate choice of $\eta$. Extension to odd $p$ follows by an application of Jensen's inequality (i.e. $\bE[\|\hvx_k\|^{2m-1}] \leq (\bE[\|\hvx_k\|^{2m}])^{\nicefrac{2m-1}{2m}} \lesssim d^{\nicefrac{m-1}{2}}$) \\

\noindent Finally, since dissipativity is a weaker condition than strong convexity outside of a compact set (as per the discussion in \cite{cheng2020sharp}, Section 2.1), the above proof also establishes that strong convexity outside of a compact set implies \ref{eqn:defn-moment-growth} 

% \section{Auxiliary Results}
% \label{app-sec:aux-results}
\end{document}